\documentclass[11pt,letterpaper]{amsart}
\usepackage[T1]{fontenc}
\usepackage[latin9]{inputenc}
\usepackage{geometry}
\usepackage{wrapfig}
\usepackage{multicol}
\usepackage{enumerate}   
\usepackage{graphicx}
\usepackage{soul}
\usepackage{xcolor}
\usepackage{amssymb}

\newtheorem{theorem}{Theorem}[section]
\newtheorem{lemma}[theorem]{Lemma}
\newtheorem{proposition}[theorem]{Proposition}

{ \theoremstyle{definition}
}
{ \theoremstyle{remark}
}
\usepackage{placeins}
\usepackage{cite}
\usepackage{caption}
\usepackage{enumerate}
\usepackage{afterpage}
\usepackage{enumitem}
\usepackage{bmpsize}
\usepackage{hyperref}

\makeatletter
\newcommand*{\rom}[1]{\expandafter\@slowromancap\romannumeral #1@}
\makeatother

\numberwithin{equation}{section}

\newtheorem{thm}{Theorem}
\numberwithin{thm}{subsection}
\newtheorem{lem}[thm]{Lemma}
{ \theoremstyle{definition} \newtheorem{defi}[thm]{Definition}}
\newtheorem{prop}[thm]{Proposition}
\newtheorem{cor}[thm]{Corollary}
\newtheorem{conjecture}[thm]{Conjecture}
{ \theoremstyle{remark} \newtheorem{rem}[thm]{Remark}}

{ \theoremstyle{remark}
\newtheorem{fact}[thm]{Fact}}

\newcommand{\iu}{{i\mkern1mu}}

\def\R{ \textup{Rem}}

\def\tc{ {\tt c}}
\def\td{ {\tt d}}
\def\tq{ {\tt q}}
\def\tM{ {\tt M}}

\def\la{ {\tt a}}
\def\lb{ {\tt b}}
\def\lc{ {\tt c}}
\def\ld{ {\tt d}}
\def\lq{ {\tt q}}
\def\lM{ {\tt M}}
\def\lu{ {\tt u}}
\def\lk{ {\tt k}}

\def\lN{ {\tt N}}
\def\lT{ {\tt T}}
\def\lS{ {\tt S}}

\def\pa{ {a}}
\def\pb{ {b}}
\def\pc{ {c}}
\def\pd{ {d}}
\def\pq{ {q}}

\newcommand{\diam}{\mathbin{\rotatebox[origin=c]{90}{$\diamondsuit$}}}

\def\i{ i}

\def\Ps{ \Psi}

\def\q{ q}
\def\t{{\bf t}}
\def\v{ {\bf v}}

\def\a{\alpha}
\def\b{\beta}

\def\E{\mathbb E_{\mathbb P_N}}

\def\Et{\mathbb E_{\mathbb P^{\t, \v}_N}}

\usepackage{tabu}
\usepackage{enumitem}

\DeclareMathOperator*{\residue}{Res}

\begin{document}

\pagestyle{plain}

\title{Log-gases on quadratic lattices via discrete loop equations and q-boxed plane partitions}
\author{Evgeni Dimitrov}

\address[Evgeni Dimitrov]{Department of Mathematics, Massachusetts Institute of Technology, Cambridge, MA, USA. E-mail: edimitro@math.mit.edu}

\author{Alisa Knizel}

\address[Alisa Knizel]{Department of Mathematics, Columbia University,
 New York, NY, USA. E-mail: knizel@math.columbia.edu}

\begin{abstract} We study a general class of log-gas ensembles on (shifted) quadratic lattices. We prove that the corresponding empirical measures satisfy a law of large numbers and that their global fluctuations are Gaussian with a universal covariance. We apply our general results to analyze the asymptotic behavior of a $q$-boxed plane partition model introduced by Borodin, Gorin and Rains. In particular, we show that the global fluctuations of the height function on a fixed slice are described by a one-dimensional section of a pullback of the two-dimensional Gaussian free field.

Our approach is based on a $q$-analogue of the Schwinger-Dyson (or loop) equations, which originate in the work of Nekrasov and his collaborators, and extends the methods developed by Borodin, Gorin and Guionnet to quadratic lattices.
\end{abstract}
\maketitle

\tableofcontents

\section{Introduction}\label{Section1}
\vspace{-1mm}
\subsection{Preface}\label{Section1.1}
A $\beta$-ensemble (or continuous log-gas) is a probability distribution $\mathbb{P}^c_N$ on $N$-tuples of ordered real numbers $x_1 < x_2 < \cdots < x_N$ with density proportional to
\begin{equation} \label{eq:distr1}
  \prod_{1\leq i<j \leq N} \left(x_j-x_i \right)^{\beta} \cdot \prod_{i=1}^{N}\exp( - N V(x_i)), 
\end{equation}
where $V(x)$ is a continuous function called {\em potential}. The study of continuous log-gases for general potentials is a rich subject that is of special interest to random matrix theory, see e.g. \cite{Meh,Forr,AGZ,PS}.  For example, when $V(x) = x^2$ and $\beta = 1,2,4$ distribution (\ref{eq:distr1}) is the joint density of the eigenvalues of random matrices from the Gaussian Orthogonal/Unitary/Symplectic ensembles \cite{AGZ}.

Recently, \cite{BGG} initiated a detailed study of a particular discrete version of (\ref{eq:distr1}) called {\em discrete $\beta$-ensembles} or {\em discrete log-gases}. These are probability distributions depending on a parameter $\theta = \beta/2 > 0$ and a positive real-valued function $w(x;N)$ of the form
\begin{equation} \label{eq:distr2}
\begin{split}
& \mathbb{P}^d_N(\lambda_1, \dots, \lambda_N)  \propto \hspace{-2mm} \prod_{1\leq i<j \leq N} \hspace{-2mm} H_\theta(\lambda_i, \lambda_j)  \prod_{i=1}^{N}w(\lambda_i; N), \mbox{ with } \\
& H_\theta(\lambda_i, \lambda_j) = \frac{\Gamma(\lambda_j - \lambda_i + 1)\Gamma(\lambda_j - \lambda_i + \theta)}{\Gamma(\lambda_j - \lambda_i)\Gamma(\lambda_j - \lambda_i +1-\theta)},
\end{split}
\end{equation}
where $\lambda_i = x_i + \theta (i-1)$ and $x_1 \leq x_2 \leq \dots \leq x_N$ are integers. The interest in these discrete models comes from integrable probability; specifically, due to their connection to uniform random tilings, $(z,w)$-measures, Jack measures, etc. 

In the present paper we consider the following two-parameter generalization of discrete $\beta$-ensembles
\begin{equation}\label{PDef_1}
\begin{split}
&\mathbb{P}_N(\ell_1, \dots, \ell_N) \propto  \prod_{1 \leq i < j \leq N} H^{q,v}_{\theta}(\ell_i, \ell_j)\cdot \prod_{i = 1}^N w(\ell_i;N), \mbox{ with } \\
&H^{q,v}_{\theta}(\ell_i, \ell_j) = q^{-2\theta \lambda_j}  \frac{\Gamma_{q}(\lambda_j - \lambda_i +1) \Gamma_{q} (\lambda_j - \lambda_i+\theta)}{\Gamma_{q} (\lambda_j - \lambda_i) \Gamma_{q} (\lambda_j - \lambda_i + 1 - \theta )} \frac{\Gamma_{q} (\lambda_j + \lambda_i + v + 1) \Gamma_{q} (\lambda_j+ \lambda_i + v + \theta)}{\Gamma_{q} (\lambda_j + \lambda_i + v ) \Gamma_{q} (\lambda_j + \lambda_i + v + 1 - \theta)},
\end{split}
\end{equation}
where $\ell_i = q^{-\lambda_i} + u \cdot q^{ \lambda_i}$, $u = q^v$ and $\lambda_i$ are as in the definition of the discrete $\beta$-ensembles, while $q \in (0,1)$ and $ v\in (1,\infty).$ The measures (\ref{eq:distr2}) are recovered from (\ref{PDef_1}) by setting $u \rightarrow 0$ and $q \rightarrow 1$. We interpret the random vector $(\ell_1, \dots, \ell_N)$ as locations of $N$ particles. If $\theta = 1$ then all particles $\ell_i$ live on the same space $ \mathfrak{Z}:= \{q^{-x} + u \cdot q^x: x \in \mathbb{Z}\}$, and we refer to the set $\mathfrak{Z}$ as a {\em quadratic lattice} in the spirit of \cite{NSU} (note that in this case $H_\theta^{q,v}(\ell_i, \ell_j) = (\ell_i - \ell_j)^2$). For general $\theta$ we call the class of measures \eqref{PDef_1} {\em discrete $\beta$-ensembles on shifted quadratic lattices}.

Our study is motivated by random matrix theory on one side, and by integrable models on the other. We first investigate  $\mathbb{P}_N$ for a general choice of weights $w$ in the {\em multu-cut} and {\em fixed filling fractions} regime. We prove that these systems obey a law of large large numbers under a certain scaling as $N$ goes to infinity and also show that their global fluctuations are asymptotically Gaussian with a {\em universal covariance}. The same phenomenon is present in the case of discrete and continuous log-gases.  Subsequently, we apply our general results to a class of tiling models that was introduced in \cite{BGR} and obtain explicit formulas for their limit shape and global fluctuations. The tiling model we investigate corresponds to a special case of (\ref{PDef_1}) when $\theta = 1$, and we remark that for general $\theta > 0$ the interaction term $H^{q,v}_{\theta}(\ell_i, \ell_j)$ can be linked to {\em Macdonald-Koorwinder polynomials} \cite{Koor92} similarly to how $H_\theta(\lambda_i, \lambda_j)$ in (\ref{eq:distr2}) is linked to Jack symmetric polynomials, see also Remark \ref{RemMacKoor}.

\subsubsection{Log-gases}\label{Section1.1.1} The probability measures from  \eqref{eq:distr1} and \eqref{eq:distr2} have been extensively studied in the past, see \cite{Meh,Forr,AGZ,PS} for $\mathbb{P}^c_N$ and \cite{CLP, J3, J4,Fe, BF} for $\mathbb{P}^d_N$ among many others. 

Under weak assumptions on the potential $V(x)$ or weight function $w(x;N)$, continuous and discrete log-gases exhibit a {\em law of large numbers} as $N \rightarrow \infty$. Specifically, if one forms the (random) empirical measures 
\begin{equation*}
\mu_N = \frac{1}{N} \sum_{i = 1}^N \delta \left( x_i/N \right), \quad \text{where } (x_1,\dots,x_N) \mbox{ is } \mathbb{P}^{c, d}_N-\mbox{distributed},
\end{equation*}
then the measures $\mu_N$ converge weakly in probability to a deterministic measure $\mu$, called the {\em equilibrium measure}. In the continuous case with $V(x) = x^2$ this statement goes back to the work of Wigner \cite{Wi}, and is called {\em Wigner's semicircle law}. The analogous statements for generic $V(x)$ were proved much later, see \cite{JL, BeGu, BPS}. In discrete settings similar law of large numbers type results were obtained in \cite{Fe, J4, J3}. In both cases the equilibrium measure $\mu$ is the solution to a suitable variational problem and one establishes the convergence of $\mu_N$ to $\mu$ by proving large deviation estimates. In essence, $\mu$ maximizes the density (\ref{eq:distr1}) or \eqref{eq:distr2} and the large deviation estimates show that $\mu_N$ concentrate around that maximum.

The next order asymptotics asks about the fluctuations of $\mu - \mu_N$ as $N \rightarrow \infty$. One natural way to analyze this difference is to form the pairings with smooth test functions $f$ and consider the asymptotic behavior of the random variables
\begin{equation}\label{gsnpair}
N \left( \int_{\mathbb{R}} f(x) \mu_N(x) - \int_{\mathbb{R}} f(x) \mu(x) \right), \hspace{ 3mm} \mbox{ as $N \rightarrow \infty$.}
\end{equation}
There is an efficient method, which establishes that the limits of (\ref{gsnpair}) are Gaussian in a very general setup and its key ingredient is the so-called {\em loop equations} (also known as Schwinger-Dyson equations), see \cite{JL, BoGu, S, BoGu2, KS} and references therein. These are functional equations for certain observables of the log-gases \eqref{eq:distr1} that are related to the Stieltjes transforms of the empirical measure $\mu_N$ and and their cumulants.  Since their introduction loop equations have become a powerful tool for studying not only global fluctuations but also {\em local universality} for random matrices \cite{BEY, BFG}.

In \cite{BGG} the authors presented an analogue of the above method for discrete $\beta$-ensembles. They introduced  {\em discrete loop equations} and used them to establish that the limits of (\ref{gsnpair}) for the measures in \ref{eq:distr2}) are Gaussian with a covariance that is the same as in the continuous case for a large class potentials. These discrete loop equations originate in the work of Nekrasov \cite{N} and are also called {\em Nekrasov's equations}. The central limit theorem for (\ref{eq:distr2}) had been previously known for various very specific {\em integrable} choices of the potential, see e.g. \cite{BF, P2, BD}. The main contribution of \cite{BGG} is that it establishes {\em general conditions} on the potential $V(x)$ that lead to the asymptotic Gaussianity of (\ref{gsnpair}). Similarly to the continuous case, discrete loop equations have become a valuable tool to study not only global fluctuations \cite{BGG} but also edge universality for discrete $\beta$-ensembles \cite{GH}.

In the present paper we establish the universality type results for the global fluctuations of discrete $\beta$-ensembles on shifted quadratic lattices (\ref{PDef_1}). To obtain the law of large numbers we use a similar combination of large deviation estimates and variational problems that proved to be successful for $\mathbb{P}_N^{c/d}$. In order to study the next order fluctuations we introduce a new version of discrete loop equations for a quadratic lattice, which we also call Nekrasov's equations, and view the latter as one of the main contributions of this paper. We remark that it is hard to guess that there even exists an analogue of the Nekrasov's equation in this setting, since it is a very subtle equation which reflects some specific algebraic structure of the system. Equipped with these new equations, we establish global central limit theorems for log-gases on a quadratic lattice for a multi-cut general potential by adapting the arguments in \cite{BGG}.

Our main motivation for considering the class of measures $\mathbb{P}_N$ comes from an interesting tiling model introduced in \cite{BGR} which we describe next. 
\subsubsection{The $q$-Racah tiling model} \label{Section1.1.2} Consider a hexagon drawn on a regular triangular lattice, whose side lengths are given by integers $a,b,c \geq 1$, see Figure \ref{S1F1}. We are interested in random tilings of such a hexagon by rhombi, also called lozenges (these are obtained by gluing two neighboring triangles together). There are three types of rhombi that arise in such a way, and they are all colored differently in Figure \ref{S1F1}. This model also has a natural $3D$ interpretation as a boxed plane partition or, equivalently, a random stepped surface formed by a stack of boxes. One can assign to every lattice vertex $(i,j)$ inside the hexagon an integer $h(i, j)$, which reflects the height of the $3D$ stack at that point, see an example in Figure \ref{S1F1}. One typically calls $h$ the {\em height function} and formulates results in terms of it. 

\begin{wrapfigure}{r}{0.3\textwidth}
  \begin{center}
    \includegraphics[width=0.28\textwidth]{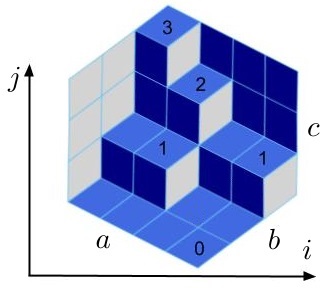}
  \end{center}
  \caption{Tiling of a hexagon and the corresponding height function}
 \label{S1F1}
\end{wrapfigure}
The probability measures on the set of tilings that we consider were introduced in \cite{BGR} and form a $2$-parameter generalization of the uniform distribution. Denoting the two parameters by $q$ and $\kappa$, one defines the weight of a tiling as the product of simple factors $ ( \kappa q^{ j -(c+1)/2 }- q^{-j + (c+1)/2}/\kappa)$ over all horizontal rhombi $\diam,$ where $j$ is the coordinate of the topmost point of the rhombus. The dependence of the factors on the location of the lozenge makes the model {\em inhomogeneous}. Note that the uniform measure on tilings is recovered if one sends $\kappa \rightarrow 0$ and $q \rightarrow 1$. Other interesting cases include $\kappa \rightarrow 0$, then the weight becomes proportional to $q^{-V}$ (here $V$ refers to the number of boxes in the $3D$ interpretation).  In addition, the same way the Hahn orthogonal polynomial ensemble arises in the analysis of uniform lozenge tilings, our measures are related to the $q$-Racah orthogonal polynomials. In this sense, the model goes all the way up to the top of the Askey scheme \cite{KLS}, and we call it the $q$-Racah tiling model. 

We believe that the $q$-Racah tiling model is a source of rich and interesting structures that are worth investigating. The presence of two parameters allows one to consider various limit regimes that lead to quite different behavior of the system as can be seen in Figure \ref{S1F6}. One of the central goals of this paper is to understand the asymptotic behavior of the height function of the $q$-Racah tiling model when the sides of the hexagon become large, and simultaneously $q\rightarrow 1, \kappa\rightarrow \kappa_0$, where $\kappa_0 \in (0,1)$ is fixed, see Figure \ref{S1F2} for a sample tiling in this case. 
\begin{figure}[h]
    \includegraphics[width=0.65\textwidth]{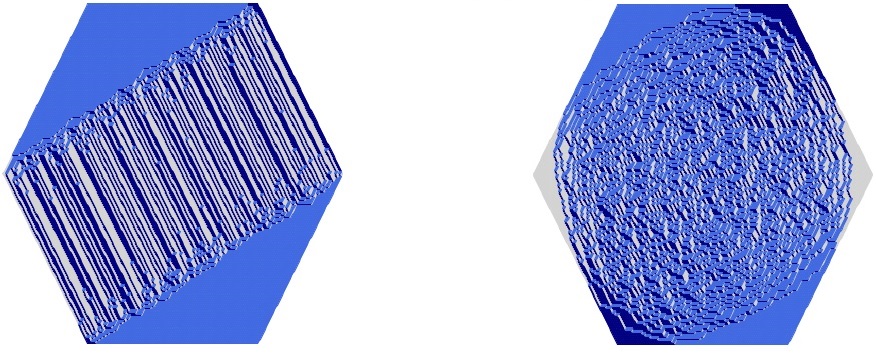}
\caption{ A simulation for $a = 80$, $b = 80$, $c =80$. On the left picture the parameters are $\kappa^2=-1,$ $q=0.8$, and on the right picture the parameters are $\kappa^2=-1,$ $q=0.98.$}
  \label{S1F6}
\end{figure}

It turns out that one can relate one-dimensional sections of the $q$-Racah tiling model to measures from (\ref{PDef_1}) with $\theta = 1$. We will elaborate on this point later in Sections \ref{Section7} and \ref{Section9.2}, but the identification goes as follows. One places a particle in the center of each horizontal lozenge $\diam$ and takes a vertical section of the model; the resulting ``holes'' (positions, where there are no particles) form an $N$-point process. Under a suitable change of variables this point process has the same distribution as (\ref{PDef_1}) for a set of parameters and weight $w$ that depend on the location of the vertical slice. 
Using this identification, our general results for log-gases on (shifted) quadratic lattices imply a law of large numbers and central limit theorem for the height function $h$ of the tiling model.

Informally, our law of large numbers states that there exists deterministic limit shape $\hat{h}$ and the random height functions $h$ concentrate near it with high probability as the parameters of the model scale to their critical values. An important feature of our model is that the limit shape develops {\em frozen facets} where the height function is linear. In addition, the frozen facets are interpolated by a connected disordered {\em liquid region}.  In terms of the tiling a frozen facet corresponds to a region where asymptotically only one type of lozenge is present and in the liquid region one sees lozenges of all three types, see Figure \ref{S1F2}.

\begin{wrapfigure}[15]{r}{0.3\textwidth}
\vspace{-2mm}
    \includegraphics[width=0.28\textwidth]{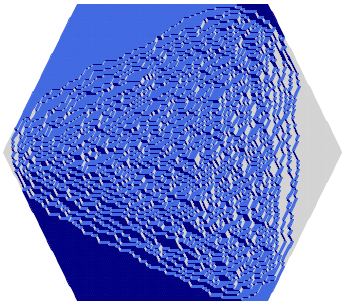}
  \caption{A random tiling of a hexagon of  side lengths $a = 50$, $b = 60$, $c = 40$ for $\kappa^2=0.05$ and $q=0.99$ }
 \label{S1F2}
\vspace{2mm}
\end{wrapfigure}
Similar concentration phenomena for the random height function in the case of the uniform measure and the measure proportional to $q^{-V}$ are well-understood. In particular, in these cases convergence of the random height function to a deterministic function for a large class of domains was established in \cite{JPS,CKP,D,DMB,KO,P1}. The limit shape is given by the unique solution of a suitable variational problem. For the $q$-Racah tiling model we compute the limit shape explicitly introducing a method, which we believe to be novel. This method uses discrete Riemann-Hilbert problems.

The next order asymptotics we obtain show that the one-dimensional fluctuations of the height function around the limit shape are Gaussian with an explicit covariance kernel. An important additional contribution of our work is the introduction of a (rather nontrivial) {\em complex structure} $\Omega$ on the liquid region. The significance of this map is that the fluctuations of $h$ on fixed vertical slices are asymptotically described by the one-dimensional sections of the pullback of the Gaussian free field (GFF for short) on the upper half-plane $\mathbb{H}$ under the map $\Omega$ -- see Theorem \ref{TCLT} for the precise statement. This result admits a natural two-dimensional generalization, which we formulate as Conjecture \ref{GFFConj} in the main text. At this time our methods only provide access to the global fluctuations at fixed vertical sections of the model, and so we cannot establish the full $2D$ result. Nevertheless, we provide some numerical simulations that give evidence for the validity of the conjecture and hope to address it in the future.

The GFF is believed to be a universal scaling limit for various models of random surfaces in $\mathbb{R}^3$. The appearance of the GFF in tiling models with no frozen zones dates back to \cite{K, K1}  and the fluctuations of the liquid region for a random tiling model containing both frozen facets and a liquid region were first studied in \cite{BF}. In case of the uniform measure on domino and lozenge tilings the convergence to the GFF has been established for a wide class of domains in \cite{P2, BufGor, BK}, but there are no results in this direction for more general measures.  One possible reason that explains why the GFF was not recognized in the $q$-Racah tiling model is the rather non-trivial change of coordinates that makes the correct covariance structure appear (see Section \ref{Section8}), and already in the $q^{-V}$ (or $\kappa = 0$) case our result is new. We remark that there is a natural complex coordinate on the liquid region defined by the so-called {\em complex slope}, which in the uniform tiling case is known to be intimately related to the complex structure that gives rise to the GFF. For the $q$-Racah tiling model an expression for the complex slope was obtained in \cite{BGR} and we connect it to our complex structure $\Omega$ through an explicit functional dependence, see Remark \ref{remCS}.

\subsection{Main results}\label{Section1.2}
We present here our main results for the log-gas on a quadratic lattice and forgo stating our results on the $q$-Racah tiling model until the main text -- Section \ref{Section7.2} -- since it requires the introduction of more notation. Moreover, to simplify the discussion in the introduction we will formulate our results for the one-cut case and $\theta=1.$ The general statement of the law of large numbers is given in Theorem \ref{GLLN} and the general statement of the central limit theorem is given in Theorem \ref{CLTfun}.

Let us first explain our regularity assumptions on the parameters and the weight function. We assume we are given parameters $\lq \in (0,1)$, $\lM \geq 1$ and $\lu \in [0, 1)$. In addition, let $q_N \in (0,1)$, $M_N \in \mathbb{N}$ and $u_N \in [0,1)$ be sequences of parameters such that 
\begin{equation}\label{GenParS1}
\mbox{$M_N \geq N-1$ and } \max \left( N^2 \left| q_N - \lq^{1/N} \right|, \left| M_N - N\lM\right|, N|u_N - \lu| \right) \leq A_1, \mbox{ for some $A_1 > 0$.}
\end{equation}

We assume that $w(x;N)$ has the form
$$w(x;N) = \exp\left( - N V_N(x)\right),$$
for a function $V_N$ that is continuous in the intervals $[1 + u_N , q_N^{-M_N} + u_N q_N^{M_N}]$ and such that 
\begin{equation}\label{GenPotS1}
\left| V_N(s) - V(s) \right| \leq A_2 \cdot N^{-1}\log(N), \mbox{ where $V$ is continuous and $|V(s)| \leq A_3$, }
\end{equation}
for some positive constants $A_2, A_3 > 0$. We also require that $V(s)$ is differentiable and for some $A_4 > 0$ there is a bound
\begin{equation}\label{DerPotS1}
\left| V'(s) \right| \leq A_4 \cdot \left[ 1 + \left| \log |s - 1 - \lu | \right|  + | \log |s -  \lq^{-\lM} - \lu \lq^{\lM} ||  \right], \mbox{ for } s \in \left[1 + \lu, \lq^{-\lM} + \lu \lq^{\lM}\right].
\end{equation}
We let $\mathbb{P}_N$ be as in (\ref{PDef_1}) for $q = q_N, q^v =u= u_N, M= M_N, N$ and weight function $w(\cdot) = w(\cdot;N)$.

Our first result is the law of large numbers for the empirical measures $\mu_N$, defined by
\begin{equation*}
\mu_N = \frac{1}{N} \sum_{i = 1}^N \delta \left( \ell_i \right), \quad \text{where } (\ell_1,\dots,\ell_N) \mbox{ is } \mathbb{P}_N-\mbox{distributed}.
\end{equation*}
\begin{thm}\label{GLLN1cut} There is a deterministic, compactly supported and absolutely continuous probability measure $\mu(x) dx$ \footnote{Throughout the paper we denote the density of a measure $\mu$ by $\mu(x)$.}  such that $\mu_N$ concentrate (in probability) near $\mu$.  More precisely, for each Lipschitz function $f(x)$ defined in a real neighborhood of the interval $ [ 1 + \lu, \lq^{-\lM} + \lu \lq^{\lM}]$ and each $\varepsilon > 0$ the random variables
\begin{equation}\label{GLLNeq}
N^{1/2 - \varepsilon} \left|\int_{\mathbb{R}} f(x) \mu_N(dx) -  \int_{\mathbb{R}} f(x)\mu(x)dx\right|
\end{equation}
converge to $0$ in probability and in the sense of moments.
\end{thm}
\begin{rem}
Theorem \ref{GLLN1cut} is a special case of Theorem \ref{GLLN}, where we extend the statement to the multi-cut regime with fixed filling fractions and for general $\theta > 0$. 
\end{rem}

To obtain our central limit theorem we need to impose certain analyticity conditions on the weight $w(x;N)$ that we now detail. We assume that we have an open set $\mathcal{M}  \subset \mathbb{C} \setminus \{0, \pm \sqrt{\lu}\}$, such that for large $N$
$$ \left( \left[q_N^{1} , q_N^{- M_N - 1}\right] \cup \left[u_Nq_N^{M_N + 1} , u_Nq_N^{ - 1}\right] \right) \subset \mathcal{M}.$$
 In addition, we require that for all sufficiently large $N$ there exist analytic functions $\Phi^+_N, \Phi^-_N$ on $\mathcal{M}$ such that for $z \in \mathcal{M}$ and $\sigma_N(z) = z + u_Nz^{-1}$ the following holds
\begin{equation}\label{eqPhiNS1}
\frac{w(\sigma_N(z);N)}{w(\sigma_N(q_Nz);N)}=\frac{ q_N(z^2-u_N) \Phi_N^+(z)}{(q_N^2z^2-u_N) \Phi_N^-(z)}.
\end{equation} 
Moreover, the functions $\Phi_N^{\pm}$ satisfy the following vanishing conditions
$$\Phi_N^+\left(q_N^{-M_N - 1 }\right)=\Phi_N^{-}\left(1\right)= \Phi_N^+\left( u_Nq_N^{ -1}\right)=\Phi_N^{-}\left(u_Nq_N^{M_N} \right)=0.$$
and asymptotic expansion
\begin{equation*}
\begin{split}
&\Phi^{-}_N(z) = \Phi^{-}(z) + \varphi^{-}_N(z) + O \left(N^{-2} \right) \mbox{ and }\Phi^{+}_N(z) = \Phi^{+}(z) + \varphi^{+}_N(z) + O \left(N^{-2} \right),
\end{split}
\end{equation*}
 where $\varphi^{\pm}_N(z) = O(N^{-1})$ and the constants in the big $O$ notation are uniform over $z$ in compact subsets of $\mathcal{M}$. All the aforementioned functions are holomorphic in $\mathcal{M}$.\\

The assumptions in (\ref{eqPhiNS1}) are the analogues of Assumptions 3 and 5 in \cite{BGG}, and similarly to that paper their importance to the analysis comes from the following observation, which is the starting point for our results. We discuss the general $\beta$ setup and the corresponding Nekrasov's equation in Section \ref{Section4}.

\begin{thm}[Nekrasov's equation]\label{NekS1} Suppose that (\ref{eqPhiNS1}) hold and define
\begin{equation} \label{REQS1}
 R_N(z)=\Phi_N^{-}(z)\cdot \E \left[ \prod\limits^N_{i=1} \frac{\sigma_N(q_Nz)-\ell_i}{\sigma_N(z)-\ell_i}\right]+ \Phi_N^{+}(z)\cdot
\E \left[ \prod\limits^N_{i=1} \frac{\sigma_N(z)- \ell_i}{\sigma_N(q_Nz)-\ell_i}\right].
\end{equation} 
Then $ R_N(z)$ is analytic in $\mathcal{M}.$ If $\Phi_N^{\pm}(z)$ are polynomials of degree at most $d$, then so is $ R_N(z)$.
\end{thm}
\begin{rem} If $\mu$ denotes the equilibrium measure from Theorem \ref{GLLN1cut}, and $G_\mu(z) = \int_{\mathbb{R}} \frac{\mu(x)dx}{z - x}$ is its Stieltjes transform then as explained in Section \ref{Section4} one has
$$\lim_{N \rightarrow \infty}  \E \left[ \prod\limits^N_{i=1} \frac{\sigma_N(q_Nz)-\ell_i}{\sigma_N(z)-\ell_i}\right] = \exp\left(\mathfrak{G}(z)\right)\mbox{ with }\mathfrak{G}(z) =  \log (\lq) \cdot (z - \lu z^{-1})  \cdot G_\mu(z + \lu z^{-1}).$$
In this sense, the Nekrasov's equation lead to a functional equation for $\mathfrak{G}(z)$, and our central limit theorem is a consequence of a careful analysis of the lower order terms of the above limit. We remark that in \cite{BGG} the expression that appears in the exponent above is directly the Stieltjes transform $G_\mu(z)$ and not a modified version of it as in our case, which increases the technical difficulty of our arguments. The appearance of $\mathfrak{G}$ is a novel feature that comes from working on a quadratic lattice and we give some explanation of it in Remark \ref{GREM}.
\end{rem}

Our central limit theorem requires that the equilibrium measure $\mu$ satisfies Assumption 5 in Section \ref{Section2.1}, which roughly ensures that $\mu$ has a single {\em band} in $[1 + \lu, \lq^{-\lM} + \lu \lq^{\lM}]$. In our context, a band is a maximal interval $(a,b)$ such that $0 < \mu(x) <  f_{\lq}(\sigma_{\lq}^{-1}(x))^{-1}$, where $\sigma_{\lq}(x) = \lq^{-x} + \lu\tq^x$ and $f_{\lq}(x) =\frac{d}{dx} \sigma_{\lq}\left(\lq^{-x} \right)$ (see also Section \ref{Section4.2}). The parameters $\alpha_1, \beta_1$ that appear in the next Theorem \ref{CLTfun1cut} are then precisely the endpoints of this band.

\begin{thm}\label{CLTfun1cut}
Suppose that (\ref{GenParS1}, \ref{GenPotS1}, \ref{DerPotS1},\ref{eqPhiNS1}) and that (technical) Assumption 5 from Section \ref{Section2.2} hold. For $m\geq 1$ let $f_1, \dots, f_m$ be real analytic functions in a neighborhood of $[1 + \lu, \lq^{-\lM} + \lu \lq^{\lM}]$ and define 
$$\mathcal L_{f_i}=N \int_\mathbb{R} f_j(x) \mu_N(dx) -N \mathbb E _{\mathbb P_N}  \left[ \int_\mathbb{R} f_j(x) \mu_N(dx) \right] \mbox{ for $i = 1, \dots, m$}.$$
Then the random variables $\mathcal{L}_{f_i}$ converge jointly in the sense of moments to an $m$-dimensional centered Gaussian vector $X = (X_1,\dots, X_m)$ with covariance
$$Cov(X_i, X_j) = \frac{1}{(2 \pi \iu)^2}\oint_{\Gamma} \oint_{\Gamma } f_i(s)f_j(t) \mathcal C (s, t) ds dt,\text{ where }$$
\begin{equation}\label{QRCOVS1}
\mathcal{C}(s, t) = -\frac{1}{2(s-t)^2} \left(1 - \frac{(s - \alpha_1)(t- \beta_1) + (t - \alpha_1 )(s- \beta_1)}{2\sqrt{(s -\alpha_1 )(s- \beta_1)}\sqrt{(t-\alpha_1 )(t- \beta_1 )}} \right),  \footnote{Throughout the paper, given $a, b \in \mathbb{R}$ with $a<b$, we write $f(z) = \sqrt{(z - a)(z- b)} $ to mean $$f(z) = \begin{cases} \sqrt{z - a} \sqrt{z-b} &\mbox{ when $z \in \mathbb{C} \setminus (-\infty, b]$ }, \\ -\sqrt{a -z }\sqrt{b - z} &\mbox{ when $z \in (-\infty, a)$ }. \end{cases}$$ Observe that in this way $f$ is holomorphic on $\mathbb{C} \setminus [a,b]$, cf. Theorem 2.5.5 in \cite{SS}.}
\end{equation}
where $\alpha_1, \beta_1 \in [1 + \lu, \lq^{-\lM} + \lu \lq^{\lM}]$ are given in Assumption 5 and $\Gamma$ is a positively oriented contour that encloses the interval $[1 + \lu, \lq^{-\lM} + \lu \lq^{\lM}]$.
\end{thm}

We emphasize that the covariance in (\ref{QRCOVS1}) depends only on $\alpha_1, \beta_1$, and is not sensitive to other features of the equilibrium measure $\mu$. Furthermore, the covariance is the same as for the continuous log-gases, cf. \cite[Theorem 2.4]{JL} and \cite[Chapter 3]{PS}. Thus, the discreteness of the model is invisible on the level of the central limit theorem, which is consistent with what was observed for the discrete $\beta$-ensembles in \cite{BGG}.

\begin{rem}
Theorem \ref{CLTfun1cut} is a special case of Theorem \ref{CLTfun}, where we extend the statement to the multi-cut regime with fixed filling fractions and general $\theta > 0$. 
\end{rem}
\begin{rem} Observe that the covariance $\mathcal{C}(s,t)$ has no singularity when $s = t$, since the RHS of (\ref{QRCOVS1}) has a finite limit when $s$ tends to $t$. 
\end{rem}

\subsection*{Outline} In Section \ref{Section2} we describe the general framework of our study, the scaling regime we consider and the assumptions on the weight $w$. In Section \ref{Section3} we establish a general law of large numbers as Theorem \ref{GLLN}. Nekrasov's equation is discussed in Section \ref{Section4}. Sections \ref{Section5} and \ref{Section6} contain the proof of Theorem \ref{CLTfun} (our general central limit theorem). A detailed description of the $q$-Racah tiling model is given in Section \ref{Section7} and we give the proof of our results about its random height function in Section \ref{Section8}. Section \ref{Section9} provides the verification that the tiling model fits into the general framework of Section \ref{Section2}. Finally, Section \ref{Section10} contains the asymptotic analysis of the Nekrasov's equation for the tiling model using discrete Riemann-Hilbert problems.

\subsection*{Acknowledgements}  The authors are deeply grateful to Alexei Borodin, Vadim Gorin and Alice Guionnet for very helpful discussions. For the second author the financial support was partially available through NSF grant DMS:1704186 and the project started when the second author was still a PhD student at Massachusetts Institute of Technology. We also want to thank the hospitality of PCMI during the summer of 2017 supported by NSF grant DMS:1441467.

\section{General setup}\label{Section2} In this section we describe the general setting of a multi-cut, fixed filling fractions model that we consider and list the specific assumptions we make about it.

\subsection{Definition of the system} \label{Section2.1}
We begin with some necessary notation. Let $q \in (0,1)$, $M\in \mathbb{Z}_{\geq 0}$, $u \in [0,1)$, $\theta > 0$ and $N \in \mathbb{N}$ be such that $M \geq N-1$. For such parameters we set 
\begin{equation}\label{GenState}
\begin{split}
&X_N = \{  (x_1, \dots, x_N) : x_1\leq x_2 \leq \cdots \leq x_N, x_i \in \mathbb{Z}  \mbox{ and }0  \leq x_i  \leq M - N + 1 \}, \\
&\mathbb{W}^\theta_N = \{ (\lambda_1, \dots, \lambda_N):  \lambda_i = x_i + (i- 1)\cdot\theta, \mbox{ with } (x_1, \dots, x_N) \in X_N\}, \\
&\mathcal{L}^\theta_N = \{ (\ell_1, \dots, \ell_N):  \ell_i = q^{-\lambda_i} + u q^{\lambda_i} , \mbox{ with }(\lambda_1, \dots, \lambda_N) \in \mathbb{W}^\theta_N \}.
\end{split}
\end{equation}
We interpret the elements $(\ell_1, \dots, \ell_N)$ in $\mathcal{L}^\theta_N$ as locations of $N$ particles.
If $\theta = 1$ then all particles $\ell_i$ live on the same space $ GCLT{Z}:= \{q^{-x} + u q^x: x \in \mathbb{Z}\}$, and we refer to the set $\mathfrak{Z}$ as a {\em quadratic lattice} in the spirit of \cite{NSU}. On the other hand,  for general $\theta > 0$ the particle $\ell_k$ lives on an appropriately {\em shifted} quadratic lattice $\mathfrak{Z}_k^\theta :=  \{q^{-x -(k-1)\theta} + u q^{x + (k-1)\theta}: x \in \mathbb{Z}\}$. This is similar to the setup in \cite{BGG}. Throughout the text we will frequently switch from $ \ell_i$'s to $\lambda_i$'s or $x_i$'s without mention using 
\begin{equation}\label{Eqcoord}
x_i \in \mathbb{Z}, \hspace{2mm} \lambda_i = x_i + (i-1)\cdot \theta \in (i-1)\cdot\theta +\mathbb{Z}, \hspace{2mm} \ell_i = q^{-\lambda_i} + u q^{\lambda_i} \in \mathfrak{Z}_i^\theta.
\end{equation}
We typically choose the coordinate system that leads to the most transparent formulas or arguments.

Our goal is to define probability measures on a subset of $\mathcal{L}^\theta_N$, where particles are split into $k$ groups of prescribed sizes, living on $k$ disjoint prescribed segments. We start by fixing $k \in \mathbb{N}$, which denotes the number of segments (or groups). For each $N \in \mathbb{N}$ we take $k$ integers $n_1(N), \dots, n_k(N)$, set $N_j = \sum_{i = 1}^j n_i(N)$ with the convention $N_0 = 0$ and assume $N_k = N$. The numbers $n_i(N)$ indicate the number of particles in each segment (or group). In addition, we suppose that we have $2k$ integers $a_i(N), b_i(N)$ such that  $0 \leq a_i(N) \leq b_i(N) - 1 \leq M-N+1$ for $i = 1, \dots, k$ and $b_i(N) \leq a_{i+1}(N)$ for $i = 1, \dots, k-1$. With the above data we define the state space of our $N$-point process as follows.
\begin{defi}\label{newStates}
The state space $\mathfrak{X}^{\theta}_N$ consists of $N$-tuples $\ell = (\ell_1, \dots, \ell_N) \in \mathcal{L}^\theta_N$ such that
$a_j(N) \leq x_i \leq b_j(N) - 1$, see (\ref{Eqcoord}), whenever $N_{j-1}+1 \leq i \leq N_j$ for $i = 1, \dots, N$ and $j = 1, \dots, k$. For future use we also denote $\alpha_i(N) = q^{-a_i(N) - N_{i-1}\cdot \theta} + u q^{a_i(N) + N_{i-1}\cdot \theta}$, $\beta_i(N) = q^{-b_i(N) - (N_{i} - 1)\cdot\theta} + u q^{b_i(N) + (N_i -1)\cdot\theta}$ and $\beta^-_i(N)$ the largest element in $\mathfrak{Z}^\theta_{N_i}$ less than $\beta_i(N)$ for $i = 1, \dots, N$.
\end{defi}
Utilizing  Definition \ref{newStates} we define a probability measure $\mathbb{P}_N$ on $\mathfrak{X}^\theta_N$ through 
\begin{equation}\label{PDef}
\begin{split}
&\mathbb{P}_N( \ell_1, \dots, \ell_N) = \frac{1}{Z_N} \cdot \prod_{1 \leq i < j \leq N} H(\ell_i, \ell_j)\cdot \prod_{i = 1}^N w(\ell_i;N), \mbox{ where } \\
&H(\ell_i, \ell_j) = q^{-2\theta \lambda_j}  \frac{\Gamma_{q}(\lambda_j - \lambda_i +1) \Gamma_{q} (\lambda_j - \lambda_i+\theta)}{\Gamma_{q} (\lambda_j - \lambda_i) \Gamma_{q} (\lambda_j - \lambda_i + 1 - \theta )} \frac{\Gamma_{q} (\lambda_j + \lambda_i + v + 1) \Gamma_{q} (\lambda_j+ \lambda_i + v + \theta)}{\Gamma_{q} (\lambda_j + \lambda_i + v ) \Gamma_{q} (\lambda_j + \lambda_i + v + 1 - \theta)}.
\end{split}
\end{equation}
Here $Z_N$ is a normalization constant (called the {\em partition function}), $v$ is such that $q^{v} = u$, and $w(x;N)$ is a weight function, which is assumed to be {\em positive} for $x \in \cup_{i = 1}^k [\alpha_i(N), \beta^-_i(N)]$. We recall  
 \begin{equation}\label{qGamma}
\Gamma_q(x) = (1 - q)^{1 - x} \frac{(q;q)_\infty}{(q^x;q)_{\infty}} \mbox{ where } (a;q)_\infty= \prod_{k = 0}^\infty(1- aq^k) \mbox{ and it satisfies }\frac{\Gamma_q(x+1)}{\Gamma_q(x)}= \frac{1 - q^x}{1- q}.
\end{equation}

Let us remark on a couple of properties of $\mathbb P_N$. Firstly, the measure $\mathbb P_N$ when $u = 0$ was considered in \cite{BGG}. Specifically, our measure agrees with equation (82) of that paper with $\lambda_i$ replaced with $\ell_i$ and $w(\ell_i;N)$ replaced with $w(\ell_i;N) \cdot q^{-\theta (N-1) \ell_i}$. In addition, from \cite[Theorem 10.2.4]{AAR}
$$\frac{\Gamma_q(x+\alpha)}{\Gamma_q(x)} = (1 - q)^{-\alpha} \frac{(q^x;q)_{\infty}}{(q^{x+\alpha}; q)_\infty} \sim(1 - q)^{-\alpha} ( 1 - y)^\alpha \mbox{ as $q \rightarrow 1^-$ and $q^x \rightarrow y \in [0,1)$}$$
and setting $q^{-\lambda_i} = y_i$ for $i = 1,\dots, N$ we have
\begin{equation}\label{thetaPN2}
\begin{split}
H(\ell_i, \ell_j)  \sim (1-q)^{- 4\cdot \theta}\cdot  y_j^{-2\theta} ( 1 - y_i y_j^{-1})^{2\theta}   ( 1 - uy_i^{-1} y_j^{-1})^{2\theta} =  (1-q)^{- 4\cdot \theta} \cdot (\ell_j- \ell_i)^{2\theta} ,
\end{split}
\end{equation}
which is why we view $\mathbb P_N$ as a discretization of the general $\beta = 2\theta$ log-gas to a quadratic lattice. The latter is particularly obvious when $\theta =1$, since then we have
\begin{equation*}
\begin{split}
\mathbb P_N (\ell_1,\dots, \ell_N) = \frac{ (1-q)^{- 2 \cdot N(N-1)}}{Z_N} \cdot \prod_{1 \leq i < j \leq N} (\ell_j - \ell_i)^2  \prod_{i = 1}^N w(\ell_i; N).
\end{split}
\end{equation*}
The above connection to log-gases motivates our choice to work with the particles $(\ell_1, \dots, \ell_N)$ and not for example $(\lambda_1, \dots, \lambda_N)$, although most results can be formulated in terms of the latter.
\begin{rem}\label{RemMacKoor}
One way to understand the interaction $H(\ell_i, \ell_j)$ in (\ref{PDef}) is that it is an {\em integrable} extension of the interaction $(\ell_j - \ell_i)^2$ to general $\theta > 0$. This should be viewed as an analogue to how (\ref{eq:distr2}) is a general $\theta > 0$ version of $(\lambda_j - \lambda_i)^2$, and the integrability of that extension can be traced to {\em discrete Selberg integrals} and {\em Jack symmetric polynomials}, where analogous expressions appear, see \cite[Section 1]{BGG}. One source of motivation for why $H(\ell_i, \ell_j)$ is the correct generalization of $(\ell_j - \ell_i)^2$ in the setting of a quadratic lattice comes from a connection to Macdonald-Koorwinder polynomials \cite{Koor92} as we detail below.

Following the notation in \cite{Rains05} we let $K_\mu^{(n)}(\cdot;q,t; t_0, t_1, t_2, t_3)$ denote the $\mathsf{BC}_n$-symmetric Koorwinder polynomial in $n$ variables. In addition, if $\mu \subset (m)^n$ we define $\tilde{\mu} \subset (n)^m$ through $\tilde{\mu}_j = n - \mu_{m - j + 1}'$. Taking the product of $K_\mu^{(n)}$ and $K_{\tilde{\mu}}^{(m)}$ at the principal and dual principal specializations (such products appear in the dual Cauchy identity for Koorwinder polynomials \cite{Mim01}) gives
\begin{equation}\label{MKP}
\begin{split}
&K_\mu^{(n)}(t^{n-i}t_0; q, t; t_0, t_1 ,t_2, t_3) \cdot K_{\tilde \mu}^{(m)}(q^{m-j}t_0; t, q; t_0, t_1 ,t_2, t_3) = \\
& C(n,m,\theta) \cdot \prod_{1 \leq i < j \leq n} H(\ell_i, \ell_j) \cdot \prod_{i = 1}^n W(\ell_i), 
\end{split}
\end{equation}
where $H(\ell_i, \ell_j)$ is as in (\ref{PDef}) with $t = q^{\theta}$, $u = q^v = t_0t_1t_2t_3 / q$ and $C(n,m,\theta)$ is a $\mu$-independent constant. As before we have $\ell_i = q^{-\lambda_i} + u q^{\lambda_i}$ and $\lambda_i = \mu_{n-i + 1} + (i-1) \cdot \theta$ (notice that $\lambda$'s are indexed in increasing order, while $\mu$'s are indexed in decreasing order as is typical for partitions). In addition, we have
$$W(\ell_i) = (1 - q^{2\lambda_i} u)\frac{(q^{1 + \lambda_i}, q^{m - \lambda_i + 1 + \theta (n-1)}, q^{\lambda_i + m + 1 + \theta(n-1)}u; q)_\infty }{(q^{\lambda_i}u; q)_\infty} q^{(n-m - 1/2)\lambda_i + \lambda_i^2/2},$$
where $(a_1, \dots, a_r; q)_\infty = \prod_{k = 1}^r (a_k; q)_\infty$. The obvious parallel between (\ref{MKP}) and (\ref{PDef}) is one of the main reasons we view $H(\ell_i, \ell_j)$ as the correct integrable generalization to $\theta > 0$.
\end{rem}

\subsection{Scaling and regularity assumption}\label{Section2.2}

We are interested in obtaining asymptotic statements about $\mathbb{P}_N$ as $N \rightarrow \infty$. This requires that we scale our parameters $\theta, q,u,M$ in some way and also impose some regularity conditions for the interval endpoints $\alpha_i(N), \beta_i(N)$ and the weight functions $w(x;N)$. We list these assumptions below.\\

{\raggedleft \bf Assumption 1.} We assume that we are given parameters $\theta > 0$ , $\lq \in (0,1)$, $\lM \geq 1$, and $\lu \in [0, 1)$. For future reference we denote the set of parameters $ {\lq}, {\lM}, {\lu}$  that satisfy the latter conditions by ${\tt P}$ and view it as a subset of $\mathbb{R}^3$ with the subspace topology. In addition, we assume that we have a sequence of parameters $q_N \in (0,1)$, $M_N \in \mathbb{N}$ and $u_N \in [0,1)$ such that
\begin{equation}\label{GenPar}
\mbox{$M_N \geq N-1$, } \max \left( N^2\cdot \left| q_N - \lq^{1/N} \right|, \left| M_N - N\lM\right|, N\cdot|u_N - \lu| \right) \leq A_1, \mbox{ for some $A_1 > 0$.}
\end{equation}
The measures $\mathbb{P}_N$ will then be as in (\ref{PDef}) for $q = q_N, u = u_N, M= M_N, \theta$ and $N$.\\

{\raggedleft \bf Assumption 2.} We require that for each $i = 1, \dots, k$ as $N\rightarrow \infty$ we have for some $A_2 > 0$
$$|\alpha_i(N) - \hat{a}_i | \leq   A_2 \cdot N^{-1} \log(N), \hspace{5mm} |\beta_i(N)- \hat{b}_i| \leq  A_2\cdot N^{-1} \log(N), \mbox{ where }$$
$$1 + \lu \leq \hat{a}_1 < \hat{b}_1 < \hat{a}_2 < \cdots < \hat{a}_k < \hat{b}_k \leq \lq^{-\lM - \theta + 1} + \lu \lq^{\lM +\theta - 1}. $$
In addition, we assume that $w(x;N)$ in the intervals $[\alpha_i(N), \beta_i^-(N)]$, $i = 1, \dots, k$ has the form
$$w(x;N) = \exp\left( - N V_N(x)\right),$$
for a function $V_N$ that is continuous in the intervals $[\alpha_i(N), \beta_i^-(N)]$ and such that 
\begin{equation}\label{GenPot}
\left| V_N(s) - V(s) \right| \leq A_3 \cdot N^{-1}\log(N), \mbox{ where $V$ is continuous and $|V(s)| \leq A_4$, }
\end{equation}
for some constants $A_3,A_4 > 0$ .We also require that $V(s)$ is differentiable and for some $A_5 > 0$
\begin{equation}\label{DerPot}
\left| V'(s) \right| \leq A_5 \cdot \left[ 1 + \sum_{i = 1}^k\left| \log |s - \hat{a}_i| \right|  + | \log |s - \hat{b}_i||  \right], \mbox{ for } s \in \left[1 + \lu, \lq^{-\lM - \theta + 1} + \lu \lq^{\lM + \theta - 1}\right].
\end{equation}
\begin{rem}
We believe that one can take more general remainders in the above two assumptions, without significantly influencing the arguments in the later parts of the paper. However, we do not pursue this direction due to the lack of natural examples.
\end{rem}

Let us denote $\sigma_{\lq}(x) = \lq^{-x} + \lu\tq^x$ and observe that the latter is a bijective diffeomorhism from $[0, \lM + \theta - 1]$ to $[1 + \lu,  \lq^{-\lM - \theta + 1} + \lu \lq^{\lM +\theta - 1}]$. Let $f_{\lq}(x) =\frac{d}{dx} \sigma_{\lq}\left(\lq^{-x} \right) $ and note that $f_{\lq}$ is positive on the interval $[0, \lM]$. \\

{\raggedleft \bf Assumption 3.} Set $\hat{n}_i(N) = \frac{n_i(N)}{N}$ for $i = 1,\dots, k$. We will often suppress the dependence of $\hat{n}_i$ on $N$ and we assume that for sufficiently large $N$ these sequences satisfy
$$ A_6 < \hat{n}_i < \theta^{-1} \cdot \left[\sigma_{\lq}^{-1}(\hat{b}_i) -  \sigma_{\lq}^{-1}(\hat{a}_i) \right]- A_6,$$
where $A_6$ is some positive constant. In our future results it will be important that the remainders are uniform over $\hat{n}_i$, satisfying the above conditions.

\begin{rem}\label{eqMR}
The above assumptions will be sufficient to obtain our law of large numbers for $\mathbb{P}_N$. We stated the one-cut $\theta = 1$ case of this law in Theorem \ref{GLLN1cut}. In general, if one assumes that  $\hat{n}_i(N) = \nu_i + O(N^{-1} \log (N))$ for some positive constants $\nu_i$ for $i = 1, \dots, k$, then the sequence of empirical measures $\mu_N := \frac{1}{N} \sum_{i = 1}^N \delta(\ell_i)$  converges to a deterministic measure $\mu$, called the {\em equilibrium measure}. The precise statement detailing this convergence is given in Theorem \ref{GLLN}, and the equilibrium measure turns out to be the maximizer of a certain variational problem -- see Lemma \ref{BL1}. It depends on $\lq, \lu, \theta$, the limiting potential $V$, the endpoints $\hat{a}_i, \hat{b}_i$ from Assumption 2 and the limiting filling fractions $\nu_i$ for $i = 1, \dots, k$.
\end{rem}

We next isolate the assumptions we require for establishing our central limit theorem, starting with the analytic properties of the weight $w(x;N)$. \\

{\raggedleft \bf Assumption 4.}  We assume that we have an open set $\mathcal{M}  \subset \mathbb{C} \setminus \{0, \pm \sqrt{\lu}\}$, such that for large $N$
$$\bigcup_{i = 1}^k \left( \left[q_N^{1-a_i(N) - N_{i-1}\cdot \theta} , q_N^{ - b_i(N) - (N_i - 1)\cdot \theta }\right] \cup \left[u_Nq_N^{b_i(N) + (N_i -1) \cdot \theta } , u_Nq_N^{a_i(N) -1 + N_{i-1}\cdot \theta}\right] \right) \subset \mathcal{M}.$$
 In addition, we require for all large $N$ the existence of analytic functions $\Phi^+_N, \Phi^-_N$ on $\mathcal{M}$ such that 
\begin{equation}\label{eqPhiN}
\begin{split}
&\frac{w(\sigma_N(z);N)}{w(\sigma_N(q_Nz);N)}=\frac{ (q_N^{2}z^2-u_Nq_N^\theta )(z^2-u_N) \Phi_N^+(z)}{(q_N^{\theta}z^2-u_N)(q_N^2z^2-u_N) \Phi_N^-(z)} , \mbox{ },
\end{split}
\end{equation} 
 whenever $\sigma_N(z), \sigma(q_Nz) \in \cup_{i = 1}^k [\alpha_i(N), \beta_i^-(N)]$ where $\sigma_N(z) = z + u_Nz^{-1}$.
Moreover, 
\begin{equation*}
\begin{split}
&\Phi^{-}_N(z) = \Phi^{-}(z) + \varphi^{-}_N(z) + O \left(N^{-2} \right) \mbox{ and } \Phi^{+}_N(z) = \Phi^{+}(z) + \varphi^{+}_N(z) + O \left(N^{-2} \right),
\end{split}
\end{equation*}
 where $\varphi^{\pm}_N(z) = O(N^{-1})$ and the constants in the big $O$ notation are uniform over $z$ in compact subsets of $\mathcal{M}$. All aforementioned functions are holomorphic in $\mathcal{M}$.\\

The next assumption we require is about the equilibrium measure $\mu$, which was discussed in Remark \ref{eqMR}. A convenient way to encode $\mu$ is through its Stieltjes transform
\begin{equation}\label{GmuDef}
G_{\mu}(z) := \int_\mathbb{R} \frac{\mu(x)dx}{z - x}.
\end{equation}
The following two functions $R_\mu(z), Q_\mu(z)$ naturally arise from our discrete loop equations (see Section \ref{Section4.2}) and play an important role in our further analysis
\begin{equation}\label{QRmu}
\begin{split}
&R_{\mu}(z) = \Phi^-(z) \cdot e^{ \theta \log (\lq) (z - \lu z^{-1}) G_{\mu} (z + \lu z^{-1}) }+  \Phi^+(z) \cdot e^{- \theta\log (\lq) (z - \lu z^{-1}) G_{\mu} (z + \lu z^{-1}) }\\
&Q_{\mu}(z) = \Phi^-(z)\cdot e^{ \theta \log (\lq) (z - \lu z^{-1}) G_{\mu} (z + \lu z^{-1}) } -  \Phi^+(z) \cdot e^{-\theta\log (\lq) (z - \lu z^{-1}) G_{\mu} (z + \lu z^{-1}) }
\end{split}
\end{equation}
In Section \ref{Section4.2} we show that $R_\mu(z)$ is analytic, while $Q_\mu(z)$ is a branch of a two-valued analytic function, given by the square root of a holomorphic function on $\mathcal{M}$. Our assumption on $\mu$ is expressed through $Q_\mu$ as follows. \\

{\raggedleft \bf Assumption 5.} For each $N$ let $\hat{\mu}_N$ be the equilibrium measure $\mu$ discussed in Remark \ref{eqMR} for the parameters $\lq$, $\lu$, endpoints $\hat{a}_i, \hat{b}_i$ as in Assumptions 1,2 and filling fractions $\nu_i = \hat{n}_i = n_i(N)/N$,  $i = 1, \dots, k$. Observe that $\hat{\mu}_N$ depends on $N$ only through the filling fractions, in particular in the one-cut case it does not depend on $N$.

Let $Q_{\hat{\mu}_N}$ be as in (\ref{QRmu}) for the measure $\hat{\mu}_N$. Then we require that for all large $N$ there exist real numbers $r_i(N), s_i(N)$ and functions $H_N(z)$ on $\mathcal{M}$ such that
\begin{itemize}
\item $\hat{a}_i \leq r_i(N) < s_i(N) \leq \hat{b}_i$, and there are constants  $\hat{a}_i \leq \hat{r}_i  < \hat{s}_i\leq \hat{b}_i$ such that $r_i(N) - \hat{r}_i = O(N^{-1}\log(N)) = s_i(N) - \hat{s}_i $ for $i = 1, \dots, k$.
\item $Q_{\hat{\mu}_N}(z) = H_N(z) \prod_{i = 1}^k \sqrt{(z + \lu z^{-1} - r_i(N))(z + \lu z^{-1} - s_i(N))}$, with $H_N(z) \neq 0$ in $\mathcal{M}$.
\end{itemize}

\begin{rem} Assumption 5 is the analogue of Assumption 4 in \cite[Section 3]{BGG} for our setting and as discussed there it does not describe a general case. In particular, it implies that $\mu(x)$ has a single interval of support inside each interval $[\hat{a}_i, \hat{b}_i]$. To the authors' knowledge there are no simple conditions on the potential $V$ that ensure that $\mu$ has this property. 
\end{rem}

We will further impose a vanishing condition for the functions $\Phi^{\pm}_N$. We believe that it can be relaxed, but introduce it to simplify our arguments in the text.\\

{\raggedleft \bf Assumption 6.} If $a_i(N), b_i(N), N_i$ are as in Section \ref{Section2.1} then for all $i = 1, \dots, k$ we have
$$\Phi_N^+\left(q_N^{-b_i - (N_i - 1)\theta }\right)=\Phi_N^{-}\left(q_N^{-a_i - N_{i-1}\cdot \theta }\right)= \Phi_N^+\left( u_Nq_N^{ a_i -1+ N_{i-1}\cdot \theta}\right)=\Phi_N^{-}\left(u_Nq_N^{b_i -1+ (N_i -1) \cdot \theta } \right)=0.$$

Finally, we state an assumption, under which one can find an explicit formula for the density of $\mu$ in Remark \ref{eqMR} in terms of $R_\mu$ and $\Phi^{\pm}$ as in (\ref{QRmu}) and Assumption 4, see Lemma \ref{Lsupp}.\\

{\raggedleft \bf Assumption 7. } 
We assume that $V(s)$ is real analytic in a real open neighborhood of $\cup_{i = 1}^k [\hat{a}_i, \hat{b}_i]$ and $\Phi^{+}(x), \Phi^{-}(x)  \in \mathbb{R}$ with $\Phi^{+}(x) \cdot \Phi^{-}(x) > 0$  whenever $x + \lu x^{-1} \in \cup_{i = 1}^k ( \hat{a}_i, \hat{b}_i)$.

\section{Law of large numbers} \label{Section3}
In this section we establish the law of large numbers for the empirical measures $\mu_N$ associated to $\mathbb{P}_N$ from Section \ref{Section2}. In Section \ref{Section3.1} we provide a variational formulation of the equilibrium measure $\mu$, which describes the limit of $\mu_N$. The convergence of $\mu_N$ to $\mu$ is detailed in Theorem \ref{GLLN} and we reduce the proof of the latter to a concentration inequality -- see Proposition \ref{BP1}. This inequality is established in Section \ref{Section3.2} using arguments similar to \cite{BGG}, which in turn go back to \cite{BoGu2} and \cite{MMS}.

\subsection{Convergence of empirical measures}\label{Section3.1} We continue with the same notation as in Section \ref{Section2}.
With $\mathbb{P}_N$ as in (\ref{PDef}) we define the empirical measures $\mu_N$  as
\begin{equation*}
\mu_N = \frac{1}{N} \sum_{i = 1}^N \delta \left( \ell_i \right) \mbox{ where } (\ell_1,\dots,\ell_N) \mbox{ is } \mathbb{P}_N-\mbox{distributed}.
\end{equation*}

The measures $\mu_N$ satisfy the following law of large numbers.
\begin{thm}\label{GLLN} Suppose that Assumptions 1, 2 and 3 from Section \ref{Section2.2} hold. In addition, suppose that $|\hat{n}_i - \nu_i | \leq A_7 \cdot N^{-1} \log(N)$ for some positive constants $A_7$ and $\nu_i$, $i = 1, \dots, k$ such that $\sum_i \nu_i = 1$. Then there is a deterministic probability measure $\mu(x) dx$, depending on $\nu_i$ for $1 \leq i \leq  k$, such that $\mu_N$ concentrate (in probability) near $\mu$.  More precisely, for each Lipschitz function $f(x)$ defined in a real neighborhood of the interval $\cup_{i = 1}^k [ \hat{a}_i, \hat{b}_i]$ and each $\varepsilon > 0$ the random variables
\begin{equation}\label{GLLNeq}
N^{1/2 - \varepsilon} \left|\int_{\mathbb{R}} f(x) \mu_N(dx) -  \int_{\mathbb{R}} f(x)\mu(x)dx\right|
\end{equation}
converge to $0$ in probability and in the sense of moments.
\end{thm}
The limiting measure $\mu$ is defined as the maximizer of a certain variational problem, described in the following section.

\subsubsection{Variational problem}\label{Section3.1.1}
 Define the functional $I_V[\rho]$ of a measure $\rho$ supported in $\cup_{i = 1}^k [\hat{a}_i, \hat{b}_i]$  via
\begin{equation}\label{DSW}
I_V[\rho] = \theta \cdot \iint\limits_{ x \neq y}\log \left| x - y\right| d\rho(x) d\rho(y)- \int_{\mathbb{R}}V(x) d\rho(x).
\end{equation}
\begin{lem}\label{BL1}
Let $\Theta$ denote the set of absolutely continuous probability measures $\rho(x)dx$ supported on $\cup_{i = 1}^k [\hat{a}_i, \hat{b}_i]$, whose denisty $\rho(x)$ is between $0$ and $\theta^{-1} \cdot f_{\lq}(\sigma_{\lq}^{-1}(x))^{-1}$ and such that
$$\int_{\hat{a}_i}^{\hat{b}_i} \rho(x) dx = \hat{\nu}_i, \mbox{ with $1 \leq i \leq k$,}$$
where $0 < \hat{\nu}_i < \theta^{-1} \cdot \left[ \sigma_{\lq}^{-1} (\hat{b}_i) -  \sigma_{\lq}^{-1} (\hat{a}_i) \right]$, $i = 1, \dots, k$ are such that  $\sum_{i = 1}^k \hat{\nu}_i = 1$ (recall that $\sigma_{\lq}$ and $f_{\lq}$ were defined in Section \ref{Section2.2}). Then the functional $I_V$ has a unique maximum $\hat{\mu}$ on $\Theta$.
\end{lem}
\begin{proof}
Observe that by our assumption on ${\lu}$ and $\lq$ we know that $\sigma_{\lq}(y)$ is a strictly increasing function, whose derivative $f_{\lq}(y)$ on $[0, \lM]$ lies between $(-\log \lq)(1-{\lu})$ and $(-\log \lq)\lq^{-\lM - \theta + 1}$. 

Let $\Theta'$ be the same as $\Theta$, except that we restrict $0 \leq \rho(x) \leq \frac{1}{\theta \cdot (-\log \lq)\cdot (1-{\lu})}$. From the above argument we conclude that $\Theta$ is a closed convex subset of $\Theta'$. It follows from the proof of Lemma 5.1 in \cite{BGG} that $I_V[\rho]$ is a continuous strictly concave functional on $\Theta'$ and that the latter is compact. It follows that $\Theta$ is convex and compact, and hence $I_V[\rho]$ attains a unique maximum there.
\end{proof}
We call the measure $\hat{\mu}$ from Lemma \ref{BL1} the {\em equilibrium measure}. The significance of $\hat{\mu}$ is that it equals $\mu$ from Theorem \ref{GLLN}, when $\hat{\nu}_i = \nu_i$ for $1 \leq i \leq k$. Proving this fact will be the focus of this and the subsequent sections.\\

For a measure $\rho \in \Theta$ as in Lemma \ref{BL1} define the effective potential $F^V_{\rho}(x)$ through
$$F^V_{\rho}(x) = 2\theta \cdot \int_{\mathbb{R}} \log|x - t| \rho(t)dt - V(x).$$
Applying Theorem 2.1 in \cite{DS} to each interval $[\hat{a}_i, \hat{b}_i]$ we know that there exist real numbers $f_i$ for $1 \leq i \leq k$ such that 
\begin{equation}\label{DSeq}
  \begin{cases}F^V_{\hat{\mu}}(x) - f_i \leq 0 &\mbox{ for a.e. } x\in S_i = \{ \hat{a}_i \leq x \leq \hat{b_i} | 0 \leq \hat{\mu}(x) < \theta^{-1} \cdot f_{\lq}(\sigma_{\lq}^{-1}(x))^{-1}\}, \\ F^V_{\hat{\mu}}(x) - f_i \geq  0 &\mbox{ if } x \in [ \hat{a}_i,\hat{b_i}] \cap \mbox{supp}(\hat{\mu}). \end{cases}
\end{equation}

\subsubsection{Proof of Theorem \ref{GLLN}}\label{Section3.1.2} Our approach to proving Theorem \ref{GLLN} is to reconstruct in our setup the arguments of Section 5 in \cite{BGG} and we begin by introducing some relevant notation. Take any two compactly supported absolutely continuous probability measures with uniformly bounded densities $\nu(x)dx$ and $\rho(x)dx$ and define $\mathcal{D}(\nu(x), \rho(x))$ through
\begin{equation}\label{B2}
\mathcal{D}^2(\nu(x), \rho(x)) = - \int_{\mathbb{R}} \int_{\mathbb{R}} \log|x-y| (\nu(x) - \rho(x))(\nu(y) - \rho(y)) dx dy.
\end{equation}
There is an alternative formula for $\mathcal{D}(\nu(x), \rho(x))$ in terms of Fourier transforms, cf. \cite{BeGu}:
\begin{equation}\label{B3}
\mathcal{D}(\nu(x), \rho(x)) = \sqrt{\int_{0}^\infty \frac{dt}{t} \left| \int_{\mathbb{R}} e^{-{\iu} tx} (\nu(x) - \rho(x)) dx\right|^2 }.
\end{equation}

Fix a parameter $p > 2$ and let $\tilde{\mu}_N$ denote the convolution of the empirical measure $\mu_N$ with the uniform measure on the interval $[0, N^{-p}]$.
With the above notation we can now state the main technical result we require for proving Theorem \ref{GLLN}.
\begin{prop}\label{BP1}
Assume the same notation as in Theorem \ref{GLLN} and let $\hat\mu$ be as in Lemma \ref{BL1} for $\hat{\nu}_i = \nu_i$ for $i = 1, \dots, k$. There exists a constant $C > 0$ such that for all $x > 0$ and $N \geq 2$ 
$$\mathbb{P}_N \left( \mathcal{D}(\tilde{\mu}_N, \hat\mu ) \geq x \right) \leq \exp\left( CN \log(N)^2 - \theta \cdot x^2 N^2\right).$$
The constant $C$ depends on the constants $A_1,\dots, A_7$ in Theorem \ref{GLLN} and Assumptions 1, 2, 3,  as well as $\lq, \lM, \lu, \theta$ and is uniform as the latter vary over compact subsets of ${\tt P}$.
\end{prop}
The proof of Proposition \ref{BP1} is the focus of Section \ref{Section3.2} below. For now we assume its validity and conclude the proof of Theorem \ref{GLLN}. We start by deducing the following corollary.

\begin{cor}\label{BC1}
Assume the same notation as in Proposition \ref{BP1}. For a compactly supported Lipschitz function $g(x)$  define 
$$ \|g \|_{ 1/2} = \left( \int_{-\infty}^\infty |s| \left| \int_{-\infty}^\infty e^{-{\iu}sx} g(x)dx \right|^2 ds \right)^{1/2}, \hspace{3mm} \|g\|_{\mbox{\em \Small Lip}} = \sup_{ x \neq y} \left|\frac{g(x) - g(y)}{x- y} \right|.$$
Fix any $p > 2$. Then for all $a > 0$, $N \geq 2$ and $g$ we have
\begin{equation}\label{B16}
\mathbb{P}_N \left( \left| \int_{\mathbb{R}} \hspace{-1mm} g(x) \mu_N(dx) -\hspace{-1mm}  \int_{\mathbb{R}}  \hspace{-1mm}  g(x) \hat\mu(dx) \right| \geq a\|g  \|_{ 1/2} + \frac{ \|g\|_{\mbox{\em \Small Lip}}}{N^p} \right) \leq \exp\left( CN \log(N)^2 - 2\theta\pi^2   a^2N^2\right), \hspace{-2mm}
\end{equation}
 where the constant $C$ is as in Proposition \ref{BP1}.
\end{cor}
\begin{proof}
From the triangle inequality we have
$$\left| \int_{\mathbb{R}} g(x) \mu_N(dx) - \int_{\mathbb{R}}  g(x) \mu(dx) \right| \leq \left| \int_{\mathbb{R}} g(x) \mu_N(dx) - \int_{\mathbb{R}}  g(x) \tilde{\mu}_{N}(dx)\right| +  \left| \int_{\mathbb{R}}  g(x) [\tilde{\mu}_N(x) - \mu(x)]dx \right| .$$
The first term is bounded by $\frac{ \|g\|_{\mbox{ \Small Lip}}}{N^p}$ and corresponds to such a term in (\ref{B16}). We will thus focus on estimating the second term.

We denote by $\mathcal{F}[\phi](\xi) := \int_{\mathbb{R}} e^{-{ \iu } x \xi} \phi(x) dx$ the Fourier transform of a function $\phi$. Note that $g, \hat\mu(x), \tilde\mu_N(x)$ all belong to $L^2(\mathbb{R}) \cap L^1(\mathbb{R})$ and so we can use Parseval's formula (see e.g. Theorem 7.1.6. in \cite{Hor}) and the Cauchy-Schwarz inequality to get
$$\left| \int_{\mathbb{R}} g(x) \tilde{\mu}_N(x)dx - \int_{\mathbb{R}}  g(x)\mu(x)dx \right| = (2\pi)^{-1}\left| \int_{\mathbb{R}} \left( \sqrt{|\xi|} \mathcal{F}[g](\xi)\right) \cdot \frac{\mathcal{F}[\tilde\mu_N](\xi) - \mathcal{F}[\hat{\mu}](\xi)}{\sqrt{|\xi|}} d \xi\right| \leq$$
$$ (2\pi)^{-1}\|g\|_{1/2} \cdot \sqrt{\int_{\mathbb{R}} \frac{|\mathcal{F}[{\tilde\mu}_N](\xi) - \mathcal{F}[\hat{\mu}](\xi)|^2}{|\xi|} d\xi} = (2\pi)^{-1}\|g\|_{1/2} \cdot \sqrt{2} \mathcal{D}(\tilde{\mu}_N , \mu) .$$
In the last equality above we used (\ref{B3}). What remains is to use Proposition \ref{BP1}.
\end{proof}

\begin{proof}[Proof of Theorem \ref{GLLN}] Suppose that $f$ and $\varepsilon$ are as in the statement of the theorem, $\mu = \hat{\mu}$ from Proposition \ref{BP1} and assume without loss of generality that $\varepsilon \in (0,1/2)$. Fix $\eta > 0$ and let $h$ be a smooth function, whose support is contained in a $\eta$-neighborhood of $\cup_{i = 1}^k [\hat{a}_i, \hat{b}_i]$ and $h(x) = 1$ on a $\eta/2$-neighborhood of $\cup_{i = 1}^k [\hat{a}_i, \hat{b}_i]$. If we set
$$X_N := N^{1/2 - \varepsilon} \left|\int_{\mathbb{R}} f(x) \cdot h(x) \mu_N(dx) -  \int_{\mathbb{R}} f(x)\cdot h(x) \mu(x)dx\right|,$$
then to prove the theorem we need to show that for each $k \geq 1$ we have $\lim_{N \rightarrow \infty} \mathbb{E}\left[ X_N^k \right] = 0$.

It follows from Corollary \ref{BC1} that there exist positive constants $c_1, c_2$ and $C$ such that for all $a > 0$ and $N \geq 2$ we have
$$\mathbb{P}\left(X_N \geq c_1 \cdot a \cdot N^{1/2 - \varepsilon} + c_2 \cdot N^{-3 } \right) \leq \exp( C N \log(N)^2 - 2\theta\pi^2 a^2 N^2).$$
Using the above inequality and setting $a_N = 2c_2 \cdot N^{-\varepsilon}$ we see that for any $k \geq 1$ we have
\begin{equation*}
\begin{split}
&\mathbb{E}\left[ X_N^k \right] = \int_0^\infty x^{k-1} \cdot \mathbb{P}(X_N \geq x) dx =  \int_0^{a_N}  x^{k-1} \cdot \mathbb{P}(X_N \geq x) dx + \int_{a_N}^{\infty}  x^{k-1} \cdot \mathbb{P}(X_N \geq x) dx \leq\\
&\leq  a_N \cdot \max( a_N^{k-1}, 1) + \exp( C N \log(N)^2 ) \cdot \int_{a_N}^{\infty}  x^{k-1} \cdot \exp\left( - \theta\cdot \frac{\pi^2 x^2 N^{1 + 2 \varepsilon}}{2 c_1^2} \right) dx
\end{split}
\end{equation*}
The last inequality implies that $\lim_{N \rightarrow \infty} \mathbb{E}\left[ X_N^k \right] = 0$.
\end{proof}

\subsection{Proof of Proposition \ref{BP1}}\label{Section3.2}
We begin with a technical result about the asymptotics of the ratio of two $q$-Gamma functions.
\begin{lemma}\label{BLemma}
Suppose that $\theta > 0$, $\lq \in (0,1)$ and $q_N \in (0,1), \alpha_N > 0$ are sequences such that $|q_N^N - \lq| \leq A \cdot N^{-1}$ and $\alpha_N \leq A$ for some $A > 0$. Then for any $x \geq 1$ we have
\begin{equation}\label{RatQG}
\frac{\Gamma_{q_N}(x + \alpha_N + \theta) }{\Gamma_{q_N}(x + \alpha_N)} = (1-q_N)^{-\theta} \cdot (1 -q_N^x)^\theta \cdot \exp \left[O\left(  \frac{N^{-1}}{1 - q_N^x}\right) \right],
\end{equation} 
where the constants in the big $O$ notation depend on $\theta, A$ and $\lq$.
\end{lemma}
\begin{proof}
For convenience we drop the dependence on $N$ from the notation. Recall from (\ref{qGamma}) that
$$\frac{\Gamma_q(x + \alpha + \theta) }{\Gamma_q(x + \alpha)} = (1 -q)^{-\theta} \cdot \frac{ (q^{x + \alpha}; q)_\infty}{ (q^{x + \alpha + \theta}; q)_\infty} , \mbox{ where we have } (a;q)_\infty = \prod_{k = 0}^\infty (1 -a q^k).$$
The first term matches the corresponding one in (\ref{RatQG}) and we focus on the second term.

We first observe that if $q,B \in (0,1)$ and $s \in [0,1]$ then
\begin{equation}\label{Beq}
\frac{(B;q)_\infty}{(B q^s; q)_\infty} \leq (1 - B)^s,
\end{equation}
The latter is equivalent to showing that 
$$f(s) := s \log(1 -B) + \sum_{k = 0}^\infty \log( 1 - B q^{s+k}) - \sum_{k = 0}^\infty \log( 1 - B q^k) \geq 0 \mbox{ on [0,1]},$$
which is immediate from the observations: $f(0) = f(1) = 0$ and $f''(s) < 0$ for $s \in (0,1)$.

We next note that
$$ \frac{ (q^{x + \alpha}; q)_\infty}{ (q^{x + \alpha + \theta}; q)_\infty} =  \prod_{i = 1}^{ \lfloor \theta  \rfloor} ( 1 - q^{x+ \alpha + i -1}) \cdot   \frac{ (q^{x + \alpha +  \lfloor \theta  \rfloor}; q)_\infty}{ (q^{x + \alpha + \theta}; q)_\infty} = \prod_{i = 1}^{ \lfloor \theta  \rfloor + 1} ( 1 - q^{x+ \alpha + i -1}) \cdot   \frac{ (q^{x + \alpha +  \lfloor \theta  \rfloor + 1}; q)_\infty}{ (q^{x + \alpha + \theta}; q)_\infty}.$$
Combining the latter with (\ref{Beq}) we conclude that 
\begin{equation*}
(1 - q^{x}) ^\theta \leq \frac{ (q^{x + \alpha}; q)_\infty}{ (q^{x + \alpha + \theta}; q)_\infty} \leq (1 - q^{x + \alpha + \theta})^\theta \leq (1 - q^{x}) ^\theta \cdot \exp\left[ \frac{\theta(1 - q^{\alpha +\theta}) }{1 - q^x} \right],
\end{equation*}
where in the last inequality we used that $q^x \leq 1$ and the trivial inequality $(1 + a)^\theta \leq e^{ a \theta}$, for $a > 0$. The latter tower of inequalities implies our desired estimate.
\end{proof}

In the remainder of this section we present the proof of Proposition \ref{BP1}, using appropriately adapted arguments from \cite{BGG}. For clarity we split it into several steps and we outline them here. In Step 1 we relate the formula for $\mathbb{P}_N$ to the value of the functional $I_V$ from Section \ref{Section3.1.1} at the empirical measure $\mu_N$. In Steps 2, 3 and 4 we obtain a lower bound for the partition function ${Z}_N$. In the fifth step we replace the empirical measure $\mu_N$ with its convolution with the uniform measure on $[0, N^{-p}]$ with $p > 2$, and reduce the statement of the proposition to establishing a certain upper bound on $I_V[\mu_N] - I_V[\tilde{\mu_N}].$ In Steps 6,7 and 8 we establish the desired upper bound by employing the variational characterization of $\hat{\mu}$ from Section \ref{Section3.1.1}. \\

{\raggedleft {\bf Step 1.}} We recall for the reader's convenience equation (\ref{PDef}) below
\begin{equation}\label{BD1}
\begin{split}
&\mathbb{P}_N( \ell_1, \dots, \ell_N) = \frac{1}{Z_N} \cdot \prod_{1 \leq i < j \leq N} H(\ell_i, \ell_j)\cdot \prod_{i = 1}^N w(\ell_i;N), \mbox{ where } \\
&H(\ell_i, \ell_j) = q^{-2\theta \lambda_j}  \frac{\Gamma_{q}(\lambda_j - \lambda_i +1) \Gamma_{q} (\lambda_j - \lambda_i+\theta)}{\Gamma_{q} (\lambda_j - \lambda_i) \Gamma_{q} (\lambda_j - \lambda_i + 1 - \theta )} \frac{\Gamma_{q} (\lambda_j + \lambda_i + v + 1) \Gamma_{q} (\lambda_j+ \lambda_i + v + \theta)}{\Gamma_{q} (\lambda_j + \lambda_i + v ) \Gamma_{q} (\lambda_j + \lambda_i + v + 1 - \theta)}.
\end{split}
\end{equation}
where we drop the dependence on $N$ from the notation. The goal of this section is to show
\begin{equation}\label{B5}
\mathbb{P}_N(\ell_1,\dots,\ell_N) = {Z}_N^{-1} \cdot (1 - q)^{-2\theta N(N-1)} \cdot \exp \left( N^2 I_V\left[\mu_N \right] + O(N\log(N) ) \right).
\end{equation}
Notice that the definition of $I_V$ in (\ref{DSW}) makes sense for discrete (atomic) measures -- here it is important that we integrate over $x\neq y$ as otherwise the integral would be infinite, and so the RHS of (\ref{B5}) is well-defined and finite.

From Assumption 2 in Section \ref{Section2.2}, we know that  $w(\ell_i;N) = \exp( - N V(\ell_i) + O(\log N))$, and to conclude the proof of (\ref{B5}) what remains is to show that
\begin{equation}\label{B51}
 \prod_{1 \leq i < j \leq N} H(\ell_i, \ell_j) = (1 - q)^{-2\theta N(N-1)} \prod_{1 \leq i < j \leq N} (\ell_j - \ell_i)^{2\theta} \cdot \exp \left[ O(N \log(N)) \right].
\end{equation}
From Lemma \ref{BLemma} we know that
\begin{equation}\label{B52}
 \prod_{1 \leq i < j \leq N} H(\ell_i, \ell_j) = (1 - q)^{-2\theta N(N-1)} \prod_{1 \leq i < j \leq N} (\ell_j - \ell_i)^{2\theta} \cdot \exp \left[ O\left(  \frac{N^{-1}}{1 - q^{\theta (j-i)}}\right)\right],
\end{equation}
where we used that $\lambda_j - \lambda_i \geq \theta \cdot (j-i)$ by assumption. On the other hand,
\begin{equation}\label{B53}
\sum_{1 \leq i < j \leq N}  \frac{N^{-1}}{1 - q^{\theta (j-i)}} \leq \frac{1}{1 - q^\theta} \cdot \sum_{i = 1}^N \frac{1 - q^\theta}{1 - q^{i \theta}} \leq \frac{1}{1 - q^\theta} \cdot  \sum_{i = 1}^N \frac{1}{i \cdot c_0} \leq C \cdot N \log(N),
\end{equation}
where $c_0 > 0$ is a universal lower bound of $q^N$. Equations (\ref{B52}) and (\ref{B53}) imply (\ref{B51}).\\

{\raggedleft {\bf Step 2.}} The goal of this and the next two steps is to obtain the following lower bound 
\begin{equation}\label{B6}
Z_N \geq  (1 - q)^{-2\theta N(N-1)} \cdot  \exp\left( N^2  I_V\left[\hat\mu\right] + O(N\log(N)^2 ) \right).
\end{equation}
In this step we construct a particular element $\hat{\ell} = (\hat{\ell}_1, \dots, \hat{\ell}_N) \in \mathfrak{X}^\theta_N$ that depends on $\hat\mu$, and then  in view of (\ref{B5}) we have the immediate lower bound
\begin{equation}\label{B61}
Z_N \geq (1 - q)^{-2\theta N(N-1)} \cdot \exp\left( N^2  I_V\left[\mbox{mes}\left[\hat{\ell}_1, \cdots, \hat{\ell}_N\right] \right] + O(N\log(N) ) \right),
\end{equation}
where  $ \mbox{mes}\left[\hat{\ell}_1, \cdots, \hat{\ell}_N \right] = \frac{1}{N} \sum_{i = 1}^N \delta ( \hat{\ell}_i)$. 

Let $y_i$, $i = 1,\dots,N$ be quantiles of $ \sigma_{\lq}^{-1}\circ \hat{\mu}$ defined through 
$$\int_0^{y_i} \left[ \sigma_{\lq}^{-1}\circ \hat{\mu} \right](x)dx = \frac{i - 1/2}{N}, \hspace{3mm} 1 \leq i \leq N.$$
Since $\left[ \sigma_{\lq}^{-1}\circ \hat{\mu} \right](x) \leq \theta^{-1}$ we have $\theta (y_{i+1} - y_i) \geq N^{-1}$. Arguing as in the proof of \cite[Proposition 5.6]{BGG} we can find an element $\hat{\lambda} = (\hat{\lambda}_1, \dots, \hat{\lambda}_N) \in \mathbb{W}^{\theta}_N$ such that:
\begin{enumerate}
\item if $\hat{\lambda}_i = \hat{x}_i + (i-1) \cdot \theta$, then $a_i(N) \leq \hat{x}_i \leq b_i(N)$;
\item there is a constant $U$ (independent of $N$) such that $| N y_i - \hat{\lambda}_i| \leq U$ for all $1 \leq i \leq N$ except for $O(\log (N))$ ones.
\end{enumerate}
We then define $\hat{\ell}$ through $\hat{\ell}_i = q^{-\hat{\lambda}_i} + u q^{\hat{\lambda}_i}$ and note that the first condition above ensures $\hat{\ell} \in \mathfrak{X}^\theta_N$.\\

{\raggedleft {\bf Step 3.}} The goal of this step is to show that 
\begin{equation}\label{B71}
N^2  I_V\left[\mbox{mes}\left[\hat{\ell}_1, \cdots, \hat{\ell}_N\right] \right] = N^2  I_V\left[ \hat{\mu} \right]+ O(N\log(N)^2 ).
\end{equation}
Clearly, (\ref{B71}) and (\ref{B61}) give (\ref{B6}).

Setting $\sigma_{N}(x) = q^{-x} + u q^x$ we see that to show (\ref{B71}) it suffices to have the following equalities
\begin{equation}\label{B8}
\sum_{1\leq  i < j \leq N} \hspace{-2mm} \log \left( \sigma_N(\hat{\lambda}_j) - \sigma_N(\hat{\lambda}_i) \right) = N^2 \iint\limits_{s < t}  \log(t - s )\hat \mu(t) \hat\mu(s) dt ds + O( N \log (N)^2),
\end{equation}
\vspace{-5mm}
\begin{equation}\label{B9}
N \sum_{i = 1}^N V( \sigma_N(\hat{\lambda}_i) ) = N^2 \int_{\mathbb{R}} V(t) \hat\mu(t) dt + O( N \log (N)).
\end{equation}
We defer the proof of (\ref{B8}) to Step 4 and focus on showing (\ref{B9}). 

Set $z_i = \sigma_{\lq}(y_i)$ for $i = 1,\dots, N$ and observe that $\int_{z_{i}}^{z_{i+1}} \hat\mu(t)dt = N^{-1}$ for $i = 1, \dots,  N-1$. Then
\begin{equation}\label{B101}
 \sum_{i = 1}^N V\left(\sigma_N(\hat{\lambda}_i) \right) =  N \sum_{i = 1}^{N-1} \int_{z_i}^{z_{i+1}} V\left(\sigma_N(\hat{\lambda}_i) \right)\hat\mu(t)dt + O(1).
\end{equation}

Let $I$ be the set of indices such that $\sigma_N(\hat{\lambda}_i), z_i, z_{i+1}$ for $i \in I$ are all inside $\cup_{j = 1}^k [\hat{a}_j, \hat{b}_j]$ and at least $N^{-1}$ away from the complement of this set, and such that $| N y_i - \hat{\lambda}_i| \leq U$ from Step 2 holds. From Assumption 2  on $\hat{a}_j, \hat{b}_j$ we conclude that $N - |I|= O(\log N)$. Note that for $i \in I$ 
\begin{itemize}
\item  $z_i = \sigma_N(\hat{\lambda}_i) + O(N^{-1})$;
\item $V(z_i) - V\left(\sigma_N(\hat{\lambda}_i) \right) = (z_i- \sigma_N(\hat{\lambda}_i))V'(s) = O \left( N^{-1} \log (N)\right)$, where we used the mean value theorem and $V'(s) = O(\log (N))$ from Assumption 2.
\end{itemize}
In view of the above (\ref{B101}) implies
\begin{equation}\label{B111}
 \sum_{i = 1}^N V\left(\sigma_N(\hat{\lambda}_i) \right)  = N \sum_{i  \in I} \int_{z_i}^{z_{i+1}} V(z_i)\hat\mu(t)dt + O(\log (N)).
\end{equation}

A second application of the mean value theorem leads to
\begin{equation*}
\sum_{i = 1}^N V\left(\sigma_N(\hat{\lambda}_i) \right) =  N \sum_{i  \in I} \int_{z_i}^{z_{i+1}}\left[ V(t) + \left(t - z_i \right) V'(\kappa(t)) \right]\hat\mu(t)dt + O(\log (N)).
\end{equation*}
where $\kappa(t)$ is a point inside $\cup_{j = 1}^k [\hat{a}_j, \hat{b}_j]$ at least $N^{-1}$ from the complement of this set. Arguing as before that $V'(\kappa(t)) = O(\log (N))$, we see that
\begin{equation}\label{B121}
\begin{split}
&\sum_{i = 1}^N V\left(\sigma_N(\hat{\lambda}_i) \right)  =  N \sum_{i  \in I} \int_{z_i}^{z_{i+1}}\left[ V(t) + \left(z_{i + 1} - z_i \right) O(\log(N)) \right]\hat\mu(t)dt + O(\log (N)) =\\
& N \sum_{i  \in I} \int_{z_i}^{z_{i+1}} V(t) \hat\mu(t)dt + N  \frac{O(\log(N))}{N} \sum_{i = 2}^{N-2} \left(z_{i+1} - z_i \right) =  N \int_{\mathbb{R}} V(t) \hat\mu(t)dt + O(\log(N)).
\end{split}
\end{equation}
The second equality above follows from the definition of $z_i$ as quantiles of $\hat\mu$, and the last one follows from the fact that $\hat\mu\left( \cup_{i \in I} [z_i, z_{i+1}]\right) = |I|/N = 1 - O(N^{-1} \log(N))$ and $V = O(1)$ on the support of $\hat \mu$. Clearly, (\ref{B121}) implies (\ref{B9}).\\

{\raggedleft {\bf Step 4.}} In this step we prove (\ref{B8}) and start by showing that
\begin{equation}\label{B7}
\sum_{1\leq  i < j \leq N}  \log \left( \sigma_{N}(\hat{\lambda}_j) - \sigma_{N}(\hat{\lambda}_i)  \right) \leq N^2 \iint\limits_{s < t}  \log(t - s )\hat \mu(t) \hat\mu(s) dt ds + O( N \log (N)^2).
\end{equation}

If $I$ is as in Step 3 then we observe that 
\begin{equation}\label{B701}
\sum_{1\leq  i < j \leq N}  \log \left( \sigma_{N}(\hat{\lambda}_j) - \sigma_{N}(\hat{\lambda}_i)  \right)  =\sum_{\substack{  i  < j; i,j \in I} }\log \left( \sigma_{N}(\hat{\lambda}_j) - \sigma_{N}(\hat{\lambda}_i)  \right)  + O( N \log (N)^2).
\end{equation}
Indeed, the two sums differ by $O(N \log (N))$ summands, each of order $O(\log(N))$.

As discussed in Step 3 we have that $ \sigma_{N}(\hat{\lambda}_i) = z_i + O(N^{-1})$ for $i \in I$. It follows, that we can find a positive constant $C$ such that
\begin{equation}\label{B702}
\begin{split}
&\sum_{\substack{  i  < j; i,j \in I} } \log \left( \sigma_{N}(\hat{\lambda}_j) - \sigma_{N}(\hat{\lambda}_i)  \right) \leq \sum_{\substack{  i < j;  i,j \in I} } \log \left(z_j- z_i + CN^{-1} \right) \leq\\
& N^2 \sum_{\substack{  i < j;  i,j \in I} } \int_{z_{i-1}}^{z_i} \int_{z_{j}}^{z_{j+1}} \log \left(t- s + CN^{-1} \right)\hat \mu(t) \hat\mu(s) dt ds =\\
& N^2 \iint\limits_{s < t}  \log \left(t- s + CN^{-1} \right)\hat \mu(t) \hat\mu(s) dt ds + O(N \log(N)^2).
\end{split}
\end{equation}
In going from the first to the second line we used that $z_i$ are quantiles of $\hat{\mu}$ and the monotonicity of $\log$. In going from the second to the third, we note that the set difference over which the two integrals are taken has $\hat \mu \times \hat\mu$ measure $O(N^{-1} \log N)$ and the integrand is $O(\log (N))$ there.

 We see that to conclude (\ref{B7}) it suffices to show that 
$$ \sum_{j = 1}^k \hspace{2mm}\iint\limits_{\hat{a}_j \leq v < w \leq \hat{b}_j}   \log \left( \frac{w - v  + CN^{-1}}{w - v}\right) dw dv =  O\left(\frac{\log(N)}{N}\right).$$
The above is now immediate from the observation that for $c \geq 0$ we have
$$\int_{a}^{b} \int_{v }^{b}   \log \left(w - v  + c\right) dw dv = \int _{a}^{b}[ (b - v + c)( \log( b - v + c) - 1) - c\log(c) + c ] dv= $$
$$ c [ 1- \log(c)] r  + \frac{c^2 [ 1 - 2 \log(c)] + (r+c)^2  [2 \log(r + c) - 1]}{4} - b r + r^2/2 + cr,$$
where $r = b - a$. The above identities show (\ref{B7}) and the reverse inequality can be established in an analogous way, which proves (\ref{B8}).\\

{\raggedleft {\bf Step 5.} } In this step we show that we can replace $\mu_N$ in (\ref{B5}) with its convolution with the uniform measure on $[0, N^{-p}]$ with $p > 2$, denoted by $\tilde{\mu}_N$. For that we extend $V$ to $\cup_{j = 1}^k [\hat{a}_j,  \hat{b}_j + N^{-p}]$ by setting $V(x + \hat{b}_j) = V(\hat{b}_j)$ for $x \in [0, N^{-p}]$ and take two independent random variables $u, \tilde{u}$ uniformly distributed on $[0, N^{-p}]$. Then we have
\begin{equation}
\begin{split}
&I_V[ \tilde{\mu}_N] = \mathbb{E}_{u, \tilde{u}}  \iint  \log\left| x + u-  y - \tilde{u}  \right| {\mu}_N(dx) {\mu}_N(dy)-\mathbb{E}_{u} \int V(x+u){\mu}_N(dx)  = \\
& I_V[ \mu_N] + \frac{1}{N}\cdot \mathbb{E}_{u, \tilde{u}} \int \log |  u-   \tilde{u}| \mu_N (dx) - \mathbb{E}_u \int [V(x+u) - V(x)]{\mu}_N(dx)  \\
& + \mathbb{E}_{u, \tilde{u}} \iint_{x \neq y} \log \left|  \frac{ x + u- y- \tilde{u}}{x-  y}\right| \mu_N (dx)\mu_N (dy)  = I_V[\mu_N] + O(N^{-1} \log(N)),
\end{split}
\end{equation}
where the last equality follows from the conditions on $V$ from Assumption 2. The above shows that we can replace $\mu_N$ with $\tilde{\mu}_N$ in (\ref{B5}) without affecting the statement. Combining the latter with the lower bound of ${Z}_N$ from (\ref{B6}), we conclude that there exists a constant $C' > 0$ such that 
\begin{equation}\label{B10}
\mathbb{P}_N(\ell_1, \dots, \ell_N) \leq \exp\left( C'N \log(N)^2 \right) \cdot \exp \left( N^2 (I_V(\tilde{\mu}_N) - I_V(\hat\mu))\right).
\end{equation}

We next claim that we have the following inequality
\begin{equation}\label{B11}
I_V[\tilde{\mu}_N] - I_V[\hat \mu]  \leq - \theta \cdot \mathcal{D}^2(\tilde{\mu}_N , \hat\mu) + O(N^{-1} \log(N)).
\end{equation}
We defer the proof of the above to the next step. In what follows we assume its validity and finish the proof of the proposition. It follows from (\ref{B10}) and (\ref{B11}) that for some $C'' > 0$ we have
$$\mathbb{P}_N(\ell_1, \dots, \ell_N) \leq \exp( C'' N \log(N)^2) \exp \left( -\theta \cdot N^2 \mathcal{D}^2(\tilde{\mu}_N , \hat\mu ) \right).$$
Notice that the number of $N$-tuples $ \ell_1 < \cdots < \ell_N $ in $\mathfrak{X}^\theta_N$ is at most $\binom{M_N}{N}$. Since  $M_N = O(N)$ we see that for some $C > 0$ and all $x > 0$ we have
$$\mathbb{P}_N\left(\mathcal{D}(\tilde{\mu}_N , \hat \mu)  \geq x \right) \leq \binom{M_N}{N} \exp \left(C''N\log(N)^2 - \theta \cdot N^2x^2\right)  \leq \exp\left( CN \log(N)^2 - \theta \cdot x^2 N^2\right).$$
This is the desired estimate.\\

{\raggedleft {\bf Step 6.}} In this and the next two steps we establish (\ref{B11}). By definition of $\mathcal{D}$ we have
\begin{equation}\label{B12}
\begin{split}
&I_V[\tilde{\mu}_N] - I_V[\hat \mu] =- \theta \cdot \mathcal{D}^2 (\tilde{\mu}_N, \hat\mu ) + \int_{\mathbb{R}}F^V_{\hat\mu}(x) \left( \tilde\mu_N(x) - \hat\mu(x)\right)dx = \\
& -\theta \cdot \mathcal{D}^2 (\tilde{\mu}_N, \hat \mu) + \sum_{j = 1}^k \int_{\hat{a}_j}^{\hat{b}_j} [F^V_{\hat\mu}(x)- f_i] \left( \tilde\mu_N(x) - \hat\mu(x)\right)dx + O(N^{-1} \log(N)),
\end{split}
\end{equation}
where we recall that $F^V_{\hat\mu}$ and $f_i$ were defined in Section \ref{Section3.1.1}. The extra $O(N^{-1} \log(N))$ comes from two sources. Firstly, there is additional mass of $ \tilde\mu_N$ that lies outside of $\cup_{j = 1}^k [\hat{a}_j, \hat{b}_j]$ and we are excluding. The second source comes from the fact that the mass of $\mu_N$ and $\hat\mu$ on each $[\hat{a}_j, \hat{b}_j]$ are not exactly the same (thus the integral over the constant $f_i$ is not zero). The first issue is resolved by Assumption 2 on the endpoints $\alpha_i(N), \beta_i(N)$, which estimates the missing weight as $O(N^{-1} \log(N))$. The second issue is resolved by our assumption that $|\hat{n}_i - \nu_i| =O(N^{-1} \log(N))$.

We recall from Section \ref{Section3.1.1} that $ S_j = \{ \hat{a}_j \leq x \leq \hat{b_j} | 0 \leq \hat{\mu}(x) < \theta^{-1} \cdot f_{\lq}(\sigma_{\lq}^{-1}(x))^{-1}\}$ and also set $S_j' = \{\hat{a}_j \leq x \leq \hat{b}_j| \hat{\mu}(x) = 0 \}$ and $S_j'' = \{\hat{a}_j \leq x \leq \hat{b}_j |  \hat\mu(x) = \theta^{-1} \cdot f_{\lq}(\sigma_{\lq}^{-1}(x))^{-1} \}$. In view of (\ref{B12}) it suffices to show for each $j = 1,\dots, k$ that 
\begin{equation}\label{B13}
\begin{split}
&\int_{S_j \setminus S_j'}[F^V_{\hat\mu}(x)- f_j] \left( \tilde\mu_N(x) - \hat\mu(x)\right)dx = 0, \hspace{2mm} \int_{S_j'}[F^V_{\hat\mu}(x)- f_j] \left( \tilde\mu_N(x) - \hat\mu(x)\right)dx \leq 0, \\ &\int_{S_j''}[F_{\hat\mu}(x)- f_j] \left( \tilde\mu_N(x) - \hat\mu(x)\right)dx \leq 0 +  O(N^{-1}\log(N)).
\end{split}
\end{equation}

In what follows we fix $j \in \{1, \dots, k\}$ and show (\ref{B13}), dropping the dependence on $j$ from all the notaiton. Let $L$ be the subset of points $[\hat{a}, \hat{b}]$ for which the Lebesgue differentiation theorem for $\hat\mu$ holds. From (\ref{DSeq}) we know that a.e. on $(S \setminus S')\cap L$ the function $F^V_{\hat\mu}(x) - f$ vanishes, this proves the first equality in (\ref{B13}), since $L$ is of full Lebesgue measure. We next observe that a.e. on $S'$ we have $ F^V_{\hat\mu}(x)- f \leq 0$ and $\hat\mu(x) = 0$ --- this proves the second inequality in (\ref{B13}). 

Let us denote by $R = \{ x \in [\hat{a},\hat{b}]: F^V_{\hat\mu}(x) > f \}$ and observe that
$$\int_{S'' }[F^V_{\hat\mu}(x)- f] \left( \tilde\mu_N(x) - \hat\mu(x)\right)dx = \int_{R\cap S''}[F^V_{\hat\mu}(x)- f] \left( \tilde\mu_N(x) -  \theta^{-1}\cdot f_{\lq}(\sigma_{\lq}^{-1}(x))^{-1}\right)dx.$$
To see the latter we first observe that on $S''$ we have $\hat\mu(x) = \theta^{-1}\cdot f_{\lq}(x)^{-1}$. In addition, we know that a.e. point in $S''$ belongs to the support of $\hat\mu$, and so by (\ref{DSeq}) a.e. on $S''$ we have that $F^V_{\hat\mu}(x) -f \geq 0$. Finally, we can remove the points of equality as they do not contribute to the integral. Next,
$$ \int_{R\cap S''}[F^V_{\hat\mu}(x)- f] \left( \tilde\mu_N(x) - \theta^{-1} \cdot f_{\lq}(\sigma_{\lq}^{-1}(x))^{-1}\right)dx =  \int_{R}[F^V_{\hat\mu}(x)- f] \left( \tilde\mu_N(x) - \theta^{-1} \cdot f_{\lq}(\sigma_{\lq}^{-1}(x))^{-1}\right)dx.$$
The above follows from the fact that $R \cap (S'')^c$ has zero Lebesgue measure, which we know from (\ref{DSeq}). We have thus reduced the proof of the proposition to establishing
\begin{equation}\label{B14}
\int_{R}[F^V_{\hat\mu}(x)- f] \left( \tilde\mu_N(x) - \theta^{-1} \cdot f_{\lq}(\sigma_{\lq}^{-1}(x))^{-1}\right)dx \leq 0 + O(N^{-1}\log(N)).
\end{equation}

{\raggedleft {\bf Step 7.}} In this and the next step we establish (\ref{B14}). We start by noting that if $\kappa =\lq^{-\lM - \theta }$ then because $\hat\mu(x)$ is bounded we have
\begin{equation}\label{BSup}
\sup_{x,y \in \mathbb{R}: |x- y| \leq \kappa N^{-1}} \left|F^V_{\hat\mu}(x) -F^V_{\hat\mu}(y) \right| =  O\left( N^{-1}\log(N)\right).
\end{equation}
In particular, the above implies that $F^V_{\hat\mu}$ is continuous and so $R$ is an open set. We denote $\sigma_{q_N^N}(y) = q_N^{-Ny} + u_N q_N^{Ny}$ and perform the change of variables $x =\sigma_{q_N^N} (y)$ to rewrite the LHS in (\ref{B14}) as
$$\int_{\sigma_{q_N^N}^{-1}(R)}\left[F^V_{\hat\mu}\left(\sigma_{q_N^N}(y) \right)- f\right] \left(\tilde{\mu}_N(\sigma_{q_N^N}(y)) -  \theta^{-1} \cdot f_{\lq}\left(\sigma_{\lq}^{-1}(\sigma_{q_N^N}(y))\right)^{-1}\right)\sigma'_{q_N^N}(y) dy.$$

We observe that by Assumption 1 we have
$$f_{\lq}\left(\sigma_{\lq}^{-1}(\sigma_{q_N^N}(y))\right)^{-1} \cdot \sigma'_{q_N^N}(y)  = 1 + O(N^{-1}).$$ 
 If we set $O = \sigma_{q_N^N}^{-1}(R)$ and $\tilde{\rho}_N(y) = \tilde{\mu}_N(\sigma_{q_N^N}(y)) \cdot \sigma'_{q_N^N}(y)$ we may rewrite the LHS in (\ref{B14}) as
$$\int_{O}\left[F^V_{\hat\mu}\left(\sigma_{q_N^N}(y) \right)- f\right] \left(\tilde{\rho}_N(y) - \theta^{-1}\right)dy + O(N^{-1}).$$

Since $\sigma_{q_N^N}(y)$ is an invertible diffeomorphism we know that $O$ is an open subset of $\left[\sigma_{q_N^N}^{-1}(\hat{a}), \sigma_{q_N^N}^{-1}(\hat{b})\right]$. In particular, we can find a collection of disjoint open intervals $(s_i, t_i), i \in J$ with $J$ countable such that $O = \cup_{i \in J} (s_i,t_i)$, upto the endpoints $\sigma_{q_N^N}^{-1}(\hat{a}), \sigma_{q_N^N}^{-1}(\hat{b})$.

Since the sum of the lengths of these intervals is at most $\lM - 1 + \theta$ we have finitely many such that $t_i - s_i > \theta \cdot N^{-1}$. Let us further subdivide such segments into segments of length exactly $\theta \cdot N^{-1}$, which are contained in $(t_i,s_i)$ as well as edge segments $(s_i,x], [y, t_i)$ with length at most $\theta \cdot N^{-1}$. In this way we obtain a finite collection $\{K_i:= [r_i, r_i + \theta /N]\}_{i \in J_1}$ and a countable collection $\{ [c_i, d_i ]\}_{i \in J_2}$ of intervals such that 
\begin{equation}\label{B15}
\begin{split}
\int_{O}[F^V_{\hat\mu}\left(\sigma_{q_N^N}(y) \right)- f] \left(\tilde{\rho}_N(y) -\theta^{-1}\right)dy = &\sum_{i \in J_1}  \int_{K_i} \left[F^V_{\hat\mu}\left(\sigma_{q_N^N}(y) \right)- f\right] \left(\tilde{\rho}_N(y) -\theta^{-1}\right)dy + \\
& + \sum_{i \in J_2}   \int_{c_i}^{d_i}\left[F^V_{\hat\mu}\left(\sigma_{q_N^N}(y) \right)- f\right] \left(\tilde{\rho}_N(y) -\theta^{-1}\right)dy
\end{split}
\end{equation}
and also $d_i - c_i \leq \theta / N$, and at least one of the points $c_i,d_i$ is a boundary point of $O$. Our goal for the remainder is to show that the sums over $J_1$ and $J_2$ are both dominated by $0 + O(N^{-1}\log(N))$. This would conclude the proof of (\ref{B14}).\\

Notice that by the continuity of $F^V_{\hat\mu}\left(\sigma_{q_N^N}(y)\right)$ and the definition of $O $, we know that on boundary points of this set we have that $F^V_{\hat\mu}\left(\sigma_{q_N^N}(y)\right)= f$. In particular, for the sum  over $J_2$ in (\ref{B15})
$$\left|F^V_{\hat\mu}\left(\sigma_{q_N^N}(y)\right)- f\right| \leq \sup_{\substack{0 \leq x,y \leq \lM + \theta \\ |x- y| \leq  N^{-1}}} \left|F^V_{\hat\mu}\left(\sigma_{q_N^N}(y)\right) - F^V_{\hat\mu}\left(\sigma_{q_N^N}(y)\right) \right| .$$
Since $\sigma'_{q_N^N}(x) \leq 2\kappa$ on $[0,\lM +\theta]$ for large $N$, we conclude from (\ref{BSup}) that
\begin{equation}\label{BSup2}
\sup_{\substack{0 \leq x,y \leq \lM +\theta \\ |x- y| \leq  N^{-1}}} \left|F^V_{\hat\mu}\left(\sigma_{q_N^N}(x)\right) - F^V_{\hat\mu}\left(\sigma_{q_N^N}(y)\right) \right| \leq  \sup_{\substack{x,y \in \mathbb{R} \\ |x- y| \leq  2\kappa N^{-1}}} \hspace{-3mm}  \left|F^V_{\hat\mu}(x) - F^V_{\hat\mu}(y)\right| = O(N^{-1} \log(N)).
\end{equation}
We conclude that the sum over $J_2$ in (\ref{B15}) is bounded in absolute value by
$$O(N^{-1}\log(N)) \cdot \sum_{i \in J_2}   \int_{c_i}^{d_i} \left(\tilde{\rho} _N(y) + \theta^{-1}\right)dy = O(N^{-1}\log(N)) .$$
We are left with estimating the sum over $J_1$, which we do in the next step.\\

{\raggedleft {\bf Step 8.}}  To conclude the proof what remains is to show
\begin{equation}\label{B16}
 \sum_{m \in J_1}  \int_{r_m}^{r_m + \theta/N}\left[F^V_{\hat\mu}\left(\sigma_{q_N^N}(y) \right)- f\right] \left(\tilde{\rho}_N(y) -\theta^{-1}\right)dy  \leq 0 + O(N^{-1}\log(N)).
\end{equation}

We first recall that by definition $\tilde{\mu}_N(x) = \sum_{i = 1}^N N^{p - 1} \cdot {\bf 1}[ \ell_i , \ell_i + N^{-p}](x),$
where $\ell = (\ell_1, \dots, \ell_N) \in \mathfrak{X}^{\theta}_N$ and ${\bf 1}K$ stands for the indicator of the set $K$. In particular, 
$$\tilde{\rho}_N(y) =  \sum_{i = 1}^N N^{p - 1}\cdot  {\bf 1} A_i (y) \cdot \sigma'_{q_N^N}(y), \mbox{ where } A_i = \left[ \lambda_i/N , \sigma_{q_N^N}^{-1}(\ell_i + N^{-p}) \right]$$
and $\lambda_i$ are such that $\ell_i = q_N^{-\lambda_i} + u_N q_N^{\lambda_i}$. 

Since $\lambda_{i+1} - \lambda_i \geq \theta$, we know that each interval $K_m=  [r_m, r_m + \theta/ N]$ intersects at most two of the intervals $A_i$. If it intersects at most one we know that 
$$ \tilde\rho_N(K_m) \leq \int_{A_i} \tilde{\rho}_N(y) dy = 1/N.$$
 If it intersects two then they must be $A_i$ and $A_{i+1}$ for some $i$ such that $\lambda_{i+1} - \lambda_i = \theta$. Let us note that $\sigma'_{q_N^N}(y) = \sigma'_{q_N^N}(\lambda_i/N) + O(N^{-1})$ whenever $y \in [\lambda_i/N, (\lambda_i + 2\theta)/N]$. In addition, we have 
\begin{equation*}
\begin{split}
\sigma_{q_N^N}^{-1}(\ell_i + N^{-p}) &= \lambda_i/N + N^{-p} \cdot \left[\frac{d}{dx}\sigma_{q_N^N}^{-1}\right](\ell_i) + O(N^{-2p}) \mbox{ and }\\
\sigma_{q_N^N}^{-1}(\ell_{i+1} + N^{-p}) &= \lambda_{i+1}/N + N^{-p} \cdot \left[\frac{d}{dx}\sigma_{q_N^N}^{-1}\right](\ell_i) + O(N^{-p - 1}).
\end{split}
\end{equation*}
Combining the above estimates, and setting $\gamma = N^{-p} \cdot \left[\frac{d}{dx}\sigma_{q_N^N}^{-1}\right](\ell_i) $, we see that
\begin{equation}\label{B18}
\begin{split}
&\tilde\rho_N(K_m) = \int\limits_{K_m} N^{p - 1} \left[  {\bf 1} A_i (y) +   {\bf 1} A_{i+1} (y)  \right]\cdot \sigma'_{q_N^N}(y)dy = \\
&\int_{K_m} N^{p - 1} \left[  {\bf 1}[\lambda_i/N, \lambda_i/N + \gamma ] (y) +   {\bf 1} [\lambda_{i+1}/N, \lambda_{i+1}/N + \gamma]  (y)  \right] \cdot \sigma'_{q_N^N}(\lambda_i/N)dy +O(N^{-2}).
\end{split}
\end{equation}
The key observation is that the integral in the second line of (\ref{B18}) is precisely $\sigma'_{q_N^N}(\lambda_i/N) \cdot N^{p - 1} \cdot \gamma = N^{-1}$. Consequently, we obtain the estimate 
$$ \tilde\rho_N(K_m) \leq N^{-1} + O(N^{-2}) \mbox{ for each } m \in J_1.$$

Combining the latter with the fact that $F^V_{\hat\mu}\left(\sigma_{q_N^N}(r_m)\right) \geq  f$ we see that
\begin{equation*}
\begin{split}
&\int_{K_m}[F^V_{\hat\mu}\left(\sigma_{q_N^N}(y)\right)- f] \left(\tilde{\rho}_N(y) -\theta^{-1}\right)dy \leq \int_{K_m}\left[F^V_{\hat\mu}\left(\sigma_{q_N^N}(y)\right) - F^V_{\hat\mu}\left(\sigma_{q_N^N}(r_m)\right)  \right] \left(\tilde{\rho}_N(y) -\theta^{-1}\right)dy \\
& + O(N^{-2})\leq  \sup_{\substack{0 \leq x,y \leq \lM + \theta \\ |x- y| \leq  N^{-1}}}  2N^{-1}\left|F^V_{\hat\mu}\left(\sigma_{q_N^N}(x)\right) -F^V_{\hat\mu}\left(\sigma_{q_N^N}(y)\right)\right|  + O(N^{-2}) =  O( N^{-2} \log(N)) ,
\end{split}
\end{equation*}
where the last equality follows from (\ref{BSup2}). Combining the above estimates over $m \in J_1$ we get
$$ \sum_{m \in J_1}  \int_{K_m}\left[F^V_{\hat\mu}\left(\sigma_{q_N^N}(y) \right)- f\right] \left(\tilde{\rho}_N(y) -\theta^{-1}\right)dy  \leq 0 +  O( N^{-2} \log(N)) \cdot |J_1|.$$
The above implies (\ref{B16}) since $(\theta/N)\cdot  |J_1|\leq \lM + \theta - 1$. 

\section{Nekrasov's equations}\label{Section4}
In this section we present the main algebraic component in our arguments, which we call the {\em Nekrasov's equations} -- Theorem \ref{NekGen}. In Section \ref{Section4.2} we study the asymptotics of this equation and explain how it gives rise to a functional equation for the equilibrium measure from Section \ref{Section3}.

\subsection{Formulation}\label{Section4.1}

As explained earlier the measure $\mathbb{P}_N$ in (\ref{PDef}) can be understood as a discretization of the continuous $\beta$ log-gas to shifted quadratic lattices. In \cite{BGG} the authors consider a different discretization (called {\em discrete $\beta$-ensembles}) where the particles occupy (appropriately shifted) integer lattices. They manage to obtain results about the global fluctuations of these particle systems and their analysis is based on appropriate discrete versions of the Schwinger-Dyson equations, which they also call the Nekrasov's equations. More recently, in \cite{GH} the same Nekrasov's equations were used to prove rigidity and edge universality for the models in \cite{BGG}.

Motivated by the success of the Nekrasov's equations for the discrete $\beta$-ensembles, we develop appropriate $q$-analogues that are applicable for the measures (\ref{PDef}). The key result is given below and can be understood as a version of the Nekrasov's equation for shifted quadratic lattices.

\begin{thm}\label{NekGen} Let $\mathbb{P}_N$  be a probability distribution on $\mathfrak{X}^\theta_N$ as in (\ref{PDef}). Let $\mathcal{M} \subset \mathbb{C}$ be open and
$$\bigcup_{i = 1}^k \left( \left[q^{1-a_i - N_{i-1}\cdot \theta} , q^{ - b_i - (N_i - 1)\cdot \theta }\right] \cup \left[uq^{b_i + (N_i -1) \cdot \theta } , uq^{a_i -1 + N_{i-1}\cdot \theta}\right] \right) \subset \mathcal{M}.$$
 Suppose there exist two functions $\Phi^+(z)$ and $\Phi^-(z)$ that are analytic in $\mathcal{M}$ and such that
\begin{equation}\label{eq:phiV2}
\begin{split}
&\frac{w(\sigma(z);N)}{w(\sigma(qz);N)}=\frac{ (q^{2}z^2-uq^\theta )(z^2-u) \Phi^+(z)}{(q^{\theta}z^2-u)(q^2z^2-u) \Phi^-(z)} , \mbox{ whenever $\sigma(z), \sigma(qz) \in \cup_{i = 1}^k [\alpha_i, \beta_i^-]$ },
\end{split}
\end{equation} 
where $\sigma(z) = z + u z^{-1}$. We also assume that $\Phi^{\pm}$ satisfy for each $i \in \{1,\dots, k\}$
\begin{equation}\label{eq:phiV22}
\Phi^+\left(q^{-b_i - (N_i - 1)\theta }\right)=\Phi^{-}\left(q^{-a_i - N_{i-1}\cdot \theta }\right)= \Phi^+\left( uq^{ a_i -1+ N_{i-1}\cdot \theta}\right)=\Phi^{-}\left(uq^{b_i -1+ (N_i -1) \cdot \theta } \right)=0.
\end{equation}
If we define
\begin{equation} \label{REQV2}
\tilde{R}(z)=\Phi^{-}(z)\cdot \E\left[ \prod\limits^N_{i=1} \frac{\sigma(q^{\theta}z)-\ell_i}{\sigma(z)-\ell_i}\right]+ \Phi^{+}(z)\cdot \E \left[ \prod\limits^N_{i=1} \frac{\sigma(q^{1-\theta}z)- \ell_i}{\sigma(qz)-\ell_i}\right]
\end{equation} 
then $ \tilde{R}(z)$ is analytic in the same complex neighborhood $\mathcal{M}.$ Moreover, if $\Phi^{\pm}(z)$ are polynomials of degree at most $d$, then so is $ \tilde{R}(z)$.
\end{thm}
\begin{proof}
As usual, see (\ref{Eqcoord}) we set $q^{-\lambda_i} + u q^{\lambda_i} = \ell_i$ for $i = 1,\dots, N$. Then we have
\begin{equation}\label{NekPeqV1}
\begin{split}
&\frac{\sigma(q^{\theta} z)- \ell_i}{\sigma(z)-\ell_i}=\frac{1}{q^{\theta}(q^{-\lambda_i}-uq^{\lambda_i})}\cdot (q^{\theta}z-q^{-\lambda_i})(q^\theta z-uq^{\lambda_i})\left(\frac{1}{z-q^{-\lambda_i}}-\frac{1}{z-uq^{\lambda_i}} \right)\\
&\frac {\sigma(q^{1-\theta}z)-\ell_i}{\sigma(qz)-\ell_i}=\frac{q^{\theta}}{(q^{-\lambda_i}-uq^{\lambda_i})}\cdot (q^{1-\theta}z-q^{-\lambda_i})(q^{1-\theta}z-uq^{\lambda_i})\left(\frac{1}{q z-q^{-\lambda_i}}-\frac{1}{q z-uq^{\lambda_i}}
\right).
\end{split}
\end{equation}
From the above we see that the possible singularities of $ \tilde{R}(z)$ in $\mathcal{M}$ are simple poles at points $q^{-m}$  and  $u q^{m}$ whenever $\lambda_i = m$ in the first line of (\ref{NekPeqV1}) and $\lambda_i = m-1$ in the second line of (\ref{NekPeqV1}). 

We will separately compute the residue contribution coming from each $i = 1, \dots, N$, which we fix for the remainder. We also let $j \in \{1, \dots, k\}$ be the unique index such that $N_{j - 1} + 1 \leq i \leq N_{j}$. By definition, we know that $\lambda_i$ varies in the set $\{C, C+1 ,\dots, D\}$, where $C = (i-1)\cdot \theta + a_j$ and $D = (i-1)\cdot \theta + b_j - 1$. If $m$ lies in $\{C, C+1 ,\dots, D\}$, we see that the residue at $q^{-m}$ is given by
\begin{equation}\label{residueAB1}
\sum\limits_{\ell \in \mathfrak{X}^\theta_N| \lambda_i=m} A(\ell, i,m) + \sum\limits_{\ell \in\mathfrak{X}^\theta_N| \lambda_i=m - 1} B(\ell,i,m) ,
\end{equation}
where
\begin{equation*}\label{FormAB1}
\begin{split}
&A(\ell,i,m)  =  \frac{\Phi^{-}(q^{-m}) (q^{\theta}q^{-m}-q^{-m})(q^\theta q^{-m}-uq^{m}) }{q^{\theta}(q^{-m}-u q^{m})} \left[\mathbb P_N(\ell) \prod \limits_{j\neq i}\frac{\sigma( q^{\theta-m})-\ell_j}{\sigma(q^{-m})-\ell_j} \right]\\
&B(\ell,i,m) =\frac{\Phi^{+}(q^{-m}) q^{\theta-1} (q^{1-\theta}q^{-m}-q^{-m+1})(q^{1-\theta}q^{-m}-uq^{m-1})}{(q^{-m+1}-uq^{m-1})}\cdot\left[ \mathbb P_N(\ell)\prod \limits_{j\neq i }\frac{\sigma(q^{1-\theta-m})-\ell_j}{\sigma( q^{1-m})-\ell_j}\right].
\end{split}
\end{equation*}

Let us fix $\ell_1, \dots, \ell_{i-1}$ and $\ell_{i +1}, \dots, \ell_N$ and set $\ell^+= (\ell_1,\dots,\ell_{i-1}, q^{-m} + uq^{m} , \ell_{i+1},\dots, \ell_N)$, $\ell^- =  (\ell_1,\dots,\ell_{i-1}, q^{-m+1} + uq^{m-1} , \ell_{i+1},\dots, \ell_N)$ -- notice that $\ell^+, \ell^-$ are not necessarily in $\mathfrak{X}^\theta_N$. We claim that $A(\ell^+,i,m) + B(\ell^-,i,m) = 0$, where we set $A$ and $B$ to be zero if the argument is not in $\mathfrak{X}^\theta_N$. If true, we would obtain that the sum in (\ref{residueAB1}) is zero and so $R$ is analytic near $q^{-m}$. The latter statement is clear if both $\ell^{\pm} \not\in \mathfrak{X}^\theta_N$, hence we assume at least one of them belongs to the state space.

If $m = a_j + (i-1)\cdot \theta$ then $B(\ell^-,i,m)  = 0$ since $\lambda_i = a_j - 1 + (i-1)\cdot \theta$ (and so  $\ell^- \not \in \mathfrak{X}^\theta_N$). In addition, $A(\ell^+,i,m) = 0$, since either $i = N_{j-1} + 1$ and then $\Phi^-(q^{-m}) = \Phi^-(q^{- a_j - N_{j-1}\theta}) = 0$ or $i \geq N_{j-1} + 2$ and then $\lambda_{i - 1} = a_j + (i-2)\cdot \theta$ so that the factor $\left(\sigma(q^{\theta-m}) - \ell_{i - 1}\right)$ vanishes. Similarly, we have $B(\ell^-,i,m)  = 0 = A(\ell^+,i,m) $ if $m = b_j +  (i-1)\cdot \theta$ . If $\lambda_{i-1} = m-\theta$, we know that $B(\ell^-,i,m)  = 0$, since $\ell^- \not \in \mathfrak{X}^\theta_N$, but also $A(\ell^+,i,m) = 0$ as it has the factor $\left(\sigma( q^{\theta-m})-\ell_{i-1}\right)$. Similarly, we have  $B(\ell^-,i,m)  = 0 = A(\ell^+,i,m) $ if $\lambda_{i+1} = m + \theta$. We may thus assume that $\lambda_{i-1} \leq m - 1 -\theta < m + 1 + \theta \leq \lambda_{i+1}$, $a_j + 1 \leq m - (i-1)\theta \leq b_j- 1$ and that $\ell^{\pm}\in \mathfrak{X}^\theta_N$.

We next observe that
\begin{align*}
\frac{(q^{\theta}q^{-m}-q^{-m})(q^\theta q^{-m}-uq^{m}) }{q^{\theta}(q^{-m}-u q^{m})}\cdot &
\frac{(q^{-m+1}-uq^{m-1}) }{q^{\theta-1} (q^{1-\theta}q^{-m}-q^{-m+1})(q^{1-\theta}q^{-m}-uq^{m-1})}=\\
& =-\frac{(q^{\theta}z^2-u)(q^2z^2-u)}{(q^{2}z^2-u q^{\theta})(z^2-u)}\Big |_{z=q^{-m}}.
\end{align*}
 Therefore, from the definition of $\Phi^+$ and $\Phi^{-}$ we get
\begin{equation}\label{ABQuot}
\frac{A(\ell^{+},i,m)}{B(\ell^-,i,m)}=-\frac{\mathbb{P}_N(\ell^{+})w(\sigma(q^{-m +1});N)}{\mathbb{P}_N(\ell^{-})w(\sigma(q^{-m});N)} \prod \limits_{j\neq i}\left[\frac{\sigma( q^{\theta-m})-\ell_j}{\sigma(q^{-m})-\ell_j} \cdot \frac{\sigma( q^{1-m})-\ell_j}{ \sigma(q^{1-\theta-m})-\ell_j}\right].
\end{equation}
Our goal for the remainder is to show that the right side in (\ref{ABQuot}) is equal to $-1$.\\

In view of (\ref{PDef}) we have that $\frac{\mathbb{P}_N(\ell^{+})w(\sigma(q^{-m +1});N)}{\mathbb{P}_N(\ell^{-})w(\sigma(q^{-m});N)}$ equals
\begin{equation}\label{bigprod}
\begin{split}
&\prod_{1 \leq l < i}  q^{-2\theta m}  \frac{\Gamma_{q}(m - \lambda_l +1) \Gamma_{q} (m - \lambda_l+\theta)}{\Gamma_{q} (m - \lambda_l) \Gamma_{q} (m - \lambda_l + 1 - \theta )}\cdot  q^{2 \theta (m-1)} \frac{\Gamma_{q} (m - \lambda_l- 1) \Gamma_{q} (m - \lambda_l - \theta )}{\Gamma_{q}(m - \lambda_l) \Gamma_{q} (m - \lambda_l+\theta - 1)} \times \\
&\prod_{1 \leq l < i}  \frac{\Gamma_{q} ( m + \lambda_l + v + 1) \Gamma_{q} ( m + \lambda_l + v + \theta)}{\Gamma_{q} (m+ \lambda_l+ v ) \Gamma_{q} ( m + \lambda_l + v+ 1 - \theta)} 
\cdot  \frac{\Gamma_{q} (m + \lambda_l- 1+ v ) \Gamma_{q} ( m + \lambda_l+ v - \theta)} {\Gamma_{q} ( m  + \lambda_l+ v ) \Gamma_{q} ( m + \lambda_l+ v + \theta - 1)} \times  \\
& \prod_{ i < j \leq N}  q^{-2\theta \lambda_j} \frac{\Gamma_{q}(\lambda_j - m +1) \Gamma_{q} (\lambda_j -m+\theta)}{\Gamma_{q} (\lambda_j - m) \Gamma_{q} (\lambda_j - m + 1 - \theta )} \cdot  q^{2\theta \lambda_j}  \frac{\Gamma_{q} (\lambda_j - m + 1) \Gamma_{q} (\lambda_j - m + 2 - \theta )}{\Gamma_{q}(\lambda_j - m+2) \Gamma_{q} (\lambda_j -m+1 + \theta)}\\
& \prod_{ i < j \leq N}  \frac{\Gamma_{q} (\lambda_j + m + v + 1) \Gamma_{q} (\lambda_j + m +v + \theta)}{\Gamma_{q} (\lambda_j +m+ v) \Gamma_{q} (\lambda_j + m+ v + 1 - \theta)} 
\cdot  \frac{\Gamma_{q} (\lambda_j +m - 1+ v) \Gamma_{q} (\lambda_j + m+ v - \theta)} {\Gamma_{q} (\lambda_j + m + v) \Gamma_{q} (\lambda_j + m + v + \theta - 1)}  .
\end{split}
\end{equation}
Using (\ref{qGamma}) and that $q^v = u$ we can rewrite (\ref{bigprod}) as
\begin{equation}\label{bigprod2}
\begin{split}
&\prod_{1 \leq l < i}  q^{-2\theta}  \frac{(1 - q^{m - \lambda_l})(1 - q^{m - \lambda_l+\theta - 1})}{(1 - q^{m - \lambda_l - 1})(1 - q^{m - \lambda_l -\theta})} \times  \frac{(1 - uq^{m + \lambda_l})(1 - uq^{m+ \lambda_l+\theta - 1})}{(1 - uq^{m + \lambda_l - 1})(1 - uq^{m + \lambda_l -\theta})} \\
& \prod_{ i < j \leq N}  \frac{(1 - q^{\lambda_j - m})(1 - q^{\lambda_j - m + 1 - \theta})} {(1 - q^{\lambda_j - m +1})(1 - q^{ \lambda_j - m+\theta})} \times \frac{(1 - uq^{\lambda_j + m})(1 - uq^{\lambda_j + m + 1 - \theta})} {(1 - uq^{\lambda_j + m +1})(1 - uq^{ \lambda_j + m+\theta})} .
\end{split}
\end{equation}
We next observe that for $l < i$ we have
$$q^{-2\theta}  \frac{(1 - q^{m - \lambda_l})(1 - q^{m - \lambda_l+\theta - 1})}{(1 - q^{m - \lambda_l - 1})(1 - q^{m - \lambda_l -\theta})} \times  \frac{(1 - uq^{m + \lambda_l})(1 - uq^{m+ \lambda_l+\theta - 1})}{(1 - uq^{m + \lambda_l - 1})(1 - uq^{m + \lambda_l -\theta})} = $$
$$  \frac{(q^{-m} - q^{ - \lambda_l})(q^{1-m -\theta} - q^{ - \lambda_l })}{(q^{1-m} - q^{ - \lambda_l })(q^{\theta-m} - q^{ - \lambda_l})} \times  \frac{(1 - uq^{m + \lambda_l})(1 - uq^{m+ \lambda_l+\theta - 1})}{(1 - uq^{m + \lambda_l- 1})(1 - uq^{m + \lambda_l -\theta})} = $$
$$\frac{[\sigma(q^{-m})-\sigma(q^{-\lambda_l})] \cdot [ \sigma(q^{1-\theta-m})-\sigma(q^{-\lambda_l})]}{[\sigma( q^{1-m})-\sigma(q^{-\lambda_l})] \cdot [\sigma( q^{\theta-m})-\sigma(q^{-\lambda_l})]  },$$
where we used that $(z + u z^{-1} - t - ut^{-1}) = (z- t)(1 - uz^{-1}t^{-1})$.

One similarly establishes that for $i < j$ we have
$$ \frac{(1 - q^{\lambda_j - m})(1 - q^{\lambda_j - m + 1 - \theta})} {(1 - q^{\lambda_j - m +1})(1 - q^{ \lambda_j - m+\theta})} \times \frac{(1 - uq^{\lambda_j + m})(1 - uq^{\lambda_j + m + 1 - \theta})} {(1 - uq^{\lambda_j + m +1})(1 - uq^{ \lambda_j + m+\theta})} = $$
$$\frac{(q^{-\lambda_j} - q^{- m})(q^{-\lambda_j} - q^{- m + 1 - \theta})} {(q^{-\lambda_j} - q^{ - m +1})(q^{-\lambda_j} - q^{  - m+\theta})} \times \frac{(1 - uq^{\lambda_j + m})(1 - uq^{\lambda_j + m + 1 - \theta})} {(1 - uq^{\lambda_j + m +1})(1 - uq^{ \lambda_j + m+\theta})} = $$
$$\frac{[\sigma(q^{-\lambda_j}) - \sigma(q^{-m})] \cdot [\sigma(q^{-\lambda_j})- \sigma(q^{1-\theta-m})]}{[\sigma(q^{-\lambda_j}) - \sigma( q^{1-m})] \cdot [\sigma(q^{-\lambda_j}) -\sigma( q^{\theta-m})]  }.$$
The last two calculations together with (\ref{bigprod2}) show the right side of (\ref{ABQuot}) is equal to $-1$ as desired. This proves that $\tilde{R}$ is analytic near $q^{-m}$. One can use analogous arguments to show that $R$ is also analytic near the points $u q^{m}$ and so on all of $\mathcal{M}$. 

Notice that if $\Phi^{\pm}(z)$ are polynomials of degree at most $d$ then $\tilde R(z)$ is entire from the first part of the theorem, which grows as $O(|z|^d)$ as $|z|\rightarrow
\infty$. By Liuoville's theorem $\tilde R(z)$ is a polynomial of degree at most $d$.
\end{proof}

\begin{rem}
Theorem \ref{NekGen} also holds if $u = 0$, where (\ref{eq:phiV2}) is replaced with $\frac{w(z;N)}{w(qz;N)}=q^{\theta}\cdot \frac{ \Phi^+(z)}{\Phi^-(z)}$ and the second two equalities in  (\ref{eq:phiV22}) are removed . In this case (\ref{NekPeqV1}) only produces possible poles at $z = q^{-m}$. From here the proof proceeds in the same way and can be found as \cite[Theorem 4.2]{BGG}.
\end{rem}

\subsection{Asymptotics of the Nekrasov's equations}\label{Section4.2} In this section we derive some properties of the equilibrium measure $\mu$ and $R_\mu, Q_\mu$ from (\ref{QRmu}) using the asymptotics of the Nekrasov's equation (\ref{REQV2}) as $N \rightarrow \infty$ under Assumptions 1-4 and 6-7. We assume the same notation as in Section \ref{Section2.2}.
\begin{lem}\label{AnalRQ}
Suppose that Assumptions 1-4 and 6 from Section \ref{Section2.2} hold. Then the functions $R_\mu$ and $Q_\mu^2$ from (\ref{QRmu}) are analytic on $\mathcal{M}$. If $\Phi^{\pm}_N$ are polynomials of degree at most $d$ then so is $R_\mu$ and $Q_\mu^2$ is a polynomial of degree at most $2d$. If Assumption 7 also holds then $R_\mu$ and $Q_\mu^2$ are real analytic on $\mathcal{M} \cap \mathbb{R}$.
\end{lem}
\begin{proof}
We observe that by Assumptions 4. and 6. the Nekrasov's equation (\ref{REQV2}) holds and so
\begin{equation} \label{analRQ1}
\tilde{R}_N(z):=\Phi_N^{-}(z)\cdot \E\left[ \prod\limits^N_{i=1} \frac{\sigma_N(q_N^{\theta}z)-\ell_i}{\sigma_N(z)-\ell_i}\right]+ \Phi_N^{+}(z)\cdot \E \left[ \prod\limits^N_{i=1} \frac{\sigma(q_N^{1-\theta}z)- \ell_i}{\sigma(q_Nz)-\ell_i}\right]
\end{equation} 
defines an analytic function on $\mathcal{M}$. For $\mu_N$ as in Section \ref{Section3.1} define
$$\mathfrak{G}_N^{d}(z):= N \log q_N \cdot (z - u_Nz^{-1})  \cdot \int_{\mathbb{R}} \frac{\mu_N(dx)}{z + u_Nz^{-1} - x}.$$

One readily observes that
\begin{equation*} 
\begin{split}
&\prod\limits^N_{i=1}\frac{\sigma_N(q_N^{\theta} z)-\ell_i}{\sigma_N(z)- \ell_i}=\exp \left[\theta  \mathfrak{G}_N^d(z)+ O(N^{-1}) \right], \hspace{2mm} \prod\limits^N_{i=1} \frac{\sigma_N(q_N^{1-\theta}z)- \ell_i}{\sigma_N(q_Nz)-\ell_i} = \exp \left[-\theta\mathfrak{G}_N^d(z)+ O(N^{-1})\right].
\end{split}
\end{equation*}
where the constants in the big $O$ notation are uniform as $z$ varies over compact subsets of $\mathcal{M} \setminus \left\{ \cup_{i = 1}^k [\hat{a}_i, \hat{b}_i] \right\}$. In addition, by Theorem \ref{GLLN} we know that $\mathfrak{G}_N^{d}(z)$ converges in probability to $G_{\mu} (z + \lu z^{-1}) $. An application of the Bounded convergence theorem and Assumption 4 implies that 
\begin{equation}\label{RLimNek}
\lim_{N \rightarrow \infty}\tilde{R}_N(z) = R_\mu(z),
\end{equation}
where the convergence is uniform over  compact subsets of $\mathcal{M} \setminus \left\{ \cup_{i = 1}^k [\hat{a}_i, \hat{b}_i] \right\}$. Since $\tilde{R}_N(z)$ are analytic in $\mathcal{M}$ we conclude the same is true for $R_\mu(z)$. Next, since $Q^2_\mu(z) = R_\mu^2(z) - 4 \Phi^+(z) \Phi^-(z)$, we conclude from Assumption 4 that $Q^2_\mu(z)$ is also analytic in $\mathcal{M}$. The real analyticity of $R_\mu$ and $Q_\mu^2$ is a consequence of the one assumed for $\Phi^{\pm}$ in Assumption 7.

If $\Phi^{\pm}_N$ are polynomials of degree at most $d$ then by Theorem \ref{NekGen} we know that so is $\tilde{R}_N(z)$. The uniform convergence of $\tilde{R}_N(z)$ over compact sets in $\mathcal{M} \setminus \left\{ \cup_{i = 1}^k [\hat{a}_i, \hat{b}_i] \right\}$ is equivalent to the convergence of the coefficients of the polynomials, and so $R_\mu(z)$ is a polynomial of degree at most $d$. Finally, the same argument shows $\Phi^{\pm}(z)$ are polynomials of degree at most $d$ and $Q^2_\mu(z) = R_\mu^2(z) - 4 \Phi^+(z) \Phi^-(z)$ is a polynomial of degree at most $2d$
\end{proof}

Our next goal is to give a formula for the equilibrium measure $\mu$ in Theorem \ref{GLLN} in terms of the functions $R_\mu$ and $\Phi^{\pm}$ but we first introduce some notation that will be useful. From Assumption 7 we know that $V$ is real analytic in an open neighborhood of $\cup_{i = 1}^k[\hat{a}_i, \hat{b}_i]$ and from \cite{Kui} we conclude that $\mu$ has a continuous density on each interval $[\hat{a}_i, \hat{b}_i]$. Borrowing terminology from \cite{BKMM}, each of the intervals $[\hat{a}_i, \hat{b}_i]$ is split into three types of regions:
\begin{enumerate}
\item Maximum (with respect to inclusion) closed intervals where $\mu(x) = 0$ are called {\em voids}.
\item Maximal open intervals where $0 < \mu(x) < \theta^{-1} \cdot f_{\lq}(\sigma_{\lq}^{-1}(x))^{-1}$  are called {\em bands} (recall $f_{\lq}$ was defined in Section \ref{Section2.2}).
\item Maximal closed  intervals where $\mu(x) = \theta^{-1} \cdot f_{\lq}(\sigma_{\lq}^{-1}(x))^{-1}$ are called {\em saturated regions}.
\end{enumerate}

\begin{lem}\label{Lsupp} Suppose that Assumptions 1-4 and 6-7 from Section \ref{Section2.2} hold. Then $\mu$ has density
\begin{equation}\label{eqMForm}
 \mu(y_0 + \lu y_0^{-1}) =  \frac{1}{\theta \log (\lq)\pi (y_0 - \lu y_0^{-1}) } \cdot \mbox{arccos} \left( \frac{R_\mu(y_0)}{2 \sqrt{\Phi^-(y_0)  \Phi^+(y_0)}}\right)\footnote{Throughout the paper we denote by $\mbox{arccos}(x)$ the function, which is $\pi$ on $(-\infty, -1]$, $0$ on $[1 ,\infty)$, and the usual arccosine function on $(-1,1)$.},
\end{equation}
for $y_0 \geq 1$ such that $y_0 + \lu y_0 \in \cup_{i = 1}^k [\hat{a}_i, \hat{b}_i]$ and $0$ otherwise. 
\end{lem}
\begin{proof}
We will assume that $\lu > 0$, the case $\lu = 0$ is simpler and can be handled similarly. As discussed earlier, Assumption 7 implies $\mu(x)$ is continuous on each interval $[\hat{a}_i, \hat{b}_i]$. By assumption there are unique $\hat{d}_i > \hat{c}_i \geq 1$ such that $\sigma(\hat{c}_i) = \hat{a}_i$ and $\sigma(\hat{d}_i) = \hat{b}_i$ for $i = 1, \dots, k$ where $\sigma(z) = z + \lu z^{-1}$. Consequently, $\sigma^{-1} ( [\hat{a}_i, \hat{b}_i]) = [\hat{c}_i, \hat{d}_i] \cup [u \hat{d}_i^{-1}, u \hat{c}_i^{-1}]$ for $i = 1, \dots, k$ and all $2k$ of the latter intervals are disjoint. Let
$$\psi(y) := \begin{cases} \mu(\sigma(y)) &\mbox{ if $y \in \cup_{i = 1}^k \left\{ [\hat{c}_i, \hat{d}_i] \cup [u \hat{d}_i^{-1}, u \hat{c}_i^{-1}] \right\},$} \\ 0 &\mbox{ else}.\end{cases}$$
It follows from (\ref{GmuDef}) that
\begin{equation*}
\begin{split}
(z - \lu z^{-1})G_{\mu}(z + \lu z^{-1}) =   \sum_{i = 1}^k\int_{\hat{c}_i}^{\hat{d}_i} \psi(y)(1 - \lu y^{-2}) \left[ \frac{z}{z-y} -\frac{\lu z^{-1}}{\lu z^{-1}-y} \right]dy.
\end{split}
\end{equation*}
 Using that 
$$\frac{z}{z - y} = \frac{y}{z- y} + 1, \hspace{2mm} \psi(y) = \psi(\lu y^{-1}) \mbox{ and }\sum_{i = 1}^k \int_{\lu \hat{c}_i^{-1}}^{\lu \hat{d}_i^{-1}} \psi(y) (1 - \lu y^{-2}) dy  = 1  \mbox{ we get }$$
\begin{equation*}
\begin{split}
&\theta \log (\lq) (z - \lu z^{-1})G_{\mu}(z + \lu z^{-1})  = \theta \log (\lq) +F(z) \mbox{, where } F(z):= \int_{\mathbb{R}} \frac{\theta \log (\lq) \psi(y)|y - \lu y^{-1}|}{z-y}dy.
\end{split}
\end{equation*}
 Using \cite[Theorem 2.1]{Hille} and \cite[Chapter 5, Theorem 93]{Tit} we conclude that $F(x + \iu y)$ defines a regular function for $y > 0$ and 
\begin{equation}\label{tit}
\begin{split}
&\lim_{ \varepsilon \rightarrow 0^+} F(x + \iu \varepsilon) = f(x) - \iu g(x) \mbox{ for a.e. $x \in \mathbb{R}$, where $f,g \in L^2(\mathbb{R})$ are given by}\\
&g(x) =  \theta \log (\lq)\pi \cdot \psi(x)|x - \lu x^{-1}| \mbox{ and } f(x) = - P\int_{\mathbb{R}} \frac{g(t)}{t - x} dx,
\end{split}
\end{equation}
and $P$ means that we take the integral in the principal value sense. Since $g$ is continuous on $\cup_{i = 1}^k [\hat{c}_i, \hat{d}_i]$, we can apply \cite[Chapter 4]{Mush} and conclude that $f$ is continuous on  $\cup_{i = 1}^k (\hat{c}_i, \hat{d}_i)$.\\

Let us take $z = y_0 + \iu \varepsilon$ with $y_0 \in \cup_{i = 1}^k (\hat{c}_i, \hat{d}_i)$  and let $\varepsilon$ converge to $0^+$ in  (\ref{QRmu}). This gives
$$R_{\mu}(y_0) = \Phi^-(y_0) e^{ \theta \log (\lq) }\cdot \exp\left(f(y_0) - \iu g(y_0) \right)  +  \Phi^+(y_0) \cdot e^{-\theta\log (\lq) } \exp\left(- f(y_0) + \iu g(y_0)\right) .$$
The above defines a quadratic equation for $\exp\left(f(y_0) - \iu g(y_0) \right) $ and we conclude that 
\begin{equation}\label{solQE}
\{ \exp\left(f(y_0) \pm \iu g(y_0) \right) \} = \left\{ \frac{ R_\mu(y_0) \pm \sqrt{  R^2_\mu(y_0) - 4\Phi^-(y_0) \Phi^+(y_0) }}{2 \Phi^-(y_0)  e^{ \theta \log (\lq)}} \right\},
\end{equation}
where the square root is with respect to the principal branch and assumed in $\mathbb{H}$ for negative values.

Suppose that $y_0 \in (\hat{c}_i, \hat{d}_i)$ and $\frac{R_\mu(y_0) }{2\sqrt{\Phi^-(y_0) \Phi^+(y_0) }} \in (-1, 1)$. Then the numbers in (\ref{solQE}) are complex conjugates with non-zero imaginary part, and we have
$$\exp\left(f(y_0) +\iu g(y_0) \right)  = \frac{ R_\mu(y_0) + \iu \sqrt{ - R^2_\mu(y_0) + 4\Phi^-(y_0) \Phi^+(y_0)  }}{2 \Phi^-(y_0)  e^{ \theta \log (\lq)}}, \mbox{ since both lie in $\mathbb{H}$.}$$
Taking the argument on both sides of the above equation we see that 
\begin{equation}\label{bandeq}
g(y_0) =  \mbox{arccos} \left( \frac{R_\mu(y_0)}{2 \sqrt{\Phi^-(y_0)  \Phi^+(y_0)}}\right) \in (0,\pi).
\end{equation}
The above computation also shows that
\begin{equation}\label{bandequiv}
\mbox{ $y_0 + \lu y_0^{-1}$ belongs to a band of $\mu$ in $[\hat{a}_i, \hat{b}_i]$  if and only if $\frac{R_\mu(y_0) }{2\sqrt{\Phi^-(y_0) \Phi^+(y_0) }} \in (-1, 1)$.}
\end{equation}

If $y_0 \in (\hat{c}_i, \hat{d}_i)$ and $\frac{R_\mu(y_0) }{2\sqrt{\Phi^-(y_0) \Phi^+(y_0) }} \geq 1$ then the numbers in (\ref{solQE}) are real and so $g(y_0) = 0$ or $\pi$, i.e. $y_0 + \lu y_0^{-1}$ belongs to a void or saturated region in $[\hat{a}_i, \hat{b}_i]$ for the measure $\mu$, which we denote by $[s,t]$. Notice that $[s,t] \neq [\hat{a}_i, \hat{b}_i]$ by our assumption on the filling fractions $\nu_i$. This implies that there is a band of $\mu$ in $[\hat{a}_i, \hat{b}_i]$ either ending at $s + \lu s^{-1}$ or starting from $t + \lu t^{-1}$. By continuity of $g$ and (\ref{bandequiv}) we see that $g(s) = 0$ or $g (t ) = 0$, which implies that $g(y_0) = 0$. A similar argument shows that $g(y_0) = \pi$ if $y_0 \in (\hat{c}_i, \hat{d}_i)$ and $\frac{R_\mu(y_0) }{2\sqrt{\Phi^-(y_0) \Phi^+(y_0) }} \leq -1$. Combining the above statements with the definition of $g$ concludes the proof of the lemma.
\end{proof}

We end the section by making a remark about the function $\mathfrak{G}_N^{d}(z)$ from the proof of Lemma \ref{AnalRQ}, whose exponent is the observable we obtain from the Nekrasov's equation.
\begin{rem}\label{GREM} Let us assume for simplicity that $k = 1$ and set 
$$\psi_N = \frac{1}{2N} \sum_{i = 1}^N \left( \delta( q^{-\lambda_i}) + \delta( u q^{\lambda_i}) \right),$$ 
where $\lambda_i$ is such that $q^{-\lambda_i} + u q^{\lambda_i} = \ell_i$ and we have dropped the dependence of $q$ and $u$ on $N$ to ease notation.
Then we have that 
$$(z - uz^{-1}) \int_{\mathbb{R}} \frac{\mu_N(dx)}{ z + uz^{-1} - x} = \int_{\mathbb{R}} \frac{ (z - uz^{-1}) \psi_N(dy)}{ z + uz^{-1} - y - u y^{-1}} =  \int_{\mathbb{R}} \psi_N(dy)\left[\frac{z}{z-y} -\frac{uz^{-1}}{uz^{-1}-y} \right].$$
 Using that $\frac{z}{z - y} = \frac{y}{z- y} + 1$ we get 
$$ \int_{\mathbb{R}} \psi_N(dy)\left[\frac{z}{z-y} -\frac{uz^{-1}}{uz^{-1}-y} \right] =  \int_{\mathbb{R}} \psi_N(dy)\left[\frac{y}{z-y} -\frac{y}{uz^{-1}-y} \right]  =  $$
$$ \int_{\mathbb{R}} \psi_N(dy)\left[\frac{y}{z-y} -\frac{uy^{-1}}{uz^{-1}-uy^{-1}} \right] = 1 + \int_{\mathbb{R}}\frac{2y \cdot \psi_N(dy)}{z-y}.$$
The above computation shows that, upto a constant and negligible error, the observable $\mathfrak{G}^d_N(z)$ we obtain from the Nekrasov's equation is the Stieltjes transform of the (deformed) empirical measure
$$\frac{1}{N} \sum_{i = 1}^N  q^{-\lambda_i}\cdot \delta( q^{-\lambda_i})  + \frac{1}{N} \sum_{i = 1}^N   u q^{\lambda_i} \cdot \delta( u q^{\lambda_i}).$$ 
In \cite{BGG} the Nekrasov's equation produced the exponent of the usual Stieltjes transform for the underlying particle system as an observable and the vanishing conditions on $\Phi^{\pm}$ 
the authors assumed, correspond to boundary conditions for that system. In our case, we see that in a sense we have two copies of particles sitting at $q^{-\lambda_i}$ and $uq^{\lambda_i}$ and the vanishing assumptions in Theorem \ref{NekGen} play the role of boundary conditions for each copy. The authors are not aware of such a phenomenon ocurring in other systems and would like to have a better conceptual understanding for its appearance. 
\end{rem}

\section{Central limit theorem: Part I} \label{Section5} 
Our goal in this and the next section is to study using Nekrasov's equation the fluctuations of the empirical measures $\mu_N$, for which we proved the law of large numbers in Section \ref{Section3}. In Section \ref{Section5.1} we introduce a $2m$-parameter deformation of the measures $\mathbb{P}_N$ and describe a certain map $\Upsilon_v$. In Section \ref{Section5.2} we state the  main technical result of the section -- Theorem \ref{MainTechT}, and deduce some corollaries from it. In Section \ref{Section5.3} we explain how to employ Nekrasov's equation for the deformed measures.  In Section \ref{Section5.4} we give  the proof of Theorem \ref{MainTechT} modulo a certain asymptotic statement in equation (\ref{S8}), whose proof is the focus of Section \ref{Section6}. 

\subsection{Deformed measure}\label{Section5.1}
We adopt the same notation as in Section \ref{Section2.2} and assume that Assumptions 1-6 hold. Introduce the usual random empirical measures $\mu_N$ on $\mathbb{R}$ through
\begin{equation}\label{M1}
\mu_N = \frac{1}{N} \sum_{i = 1}^N \delta \left(\ell_i\right), \mbox{ where $(\ell_1,\dots,\ell_N)$ is $\mathbb{P}_N$ - distributed}.
\end{equation}
We also define the (continuous, deterministic) probability measures
\begin{equation}\label{M2}
\begin{split}
\hat\mu_N \mbox{ as in Assumption 5.}
\end{split}
\end{equation}
 It follows from Corollary \ref{BC1} that $\mu_N - \hat\mu_N$ converges weakly in probability to $0$. Our goal is to understand the fluctuations of $\mu_N - \hat\mu_N$.

Let us introduce the Stieltjes transforms of $\mu_N$ and $\hat\mu_N$ through
\begin{equation}\label{RegG}
 G^d_N(z)=\int\limits_{\mathbb{R}}\frac{\mu_N(dx)}{z-x} \text{ and } G^c_N(z)=\int\limits_{\mathbb{R}}\frac{\hat\mu_N(dx)}{z-x}. 
\end{equation}
Observe that the above formulas make sense whenever $z$ does not lie in the support of the measures, and they define holomorphic functions there. Our study of  $\mu_N - \hat\mu_N$ goes through understanding $G_N^d(z ) - G_N^c(z )$ as $N\rightarrow \infty$. For that we introduce a deformed version of $\mathbb{P}_N$ following an approach that is similar to the one in \cite{BGG}.\\

Take $2m$ parameters $\t = (t_1,\dots,t_m)$, $\v = (v_1,\dots,v_m)$ such that $v_a + t_a - y \neq 0$ for all $a = 1,\dots,m$ and all $y \in \cup_{j = 1}^k [\hat{a}_j, \hat{b}_j]$, and let the deformed distribution $\mathbb P^{\t, \v}_N$  be defined through
\begin{equation} \label{eq:distrgen_deformed}
\mathbb P^{\t, \v}_N(\ell_1,\dots ,\ell_N)=Z(\t, \v)^{-1}\prod_{1\leq i<j \leq N}(\ell_i-\ell_j)^2 
\prod_{i=1}^{N}\left[w(\ell_i;N)\prod^k_{a=1}  \left(  \hspace{-1mm} 1+ \frac{t_a}{v_a-\ell_i} \right) \right].
\end{equation}
 If $m = 0$ we have $\mathbb P_N^{\t, \v} = \mathbb{P}_N$ is the undeformed measure. In general, $\mathbb P_N^{\t, \v}$ may be a complex-valued measure but we always choose the normalization constant $Z(\t, \v)$ so that $\sum_{\ell \in \mathfrak{X}} \mathbb P^{\t, \v}_N(\ell) = 1$. In addition, we require that the numbers $t_a$ are sufficiently close to zero so that $Z(\t, \v) \neq 0$. 

Let us denote 
\begin{equation}\label{DeltaGDef}
\Delta G_N(z) = N(G^d_N(z) - G^c_N(z)), \mbox{ where $(\ell_1,\dots,\ell_N)$ is $\mathbb{P}^{\t, \v}_N$ - distributed}. 
\end{equation}
Abusing notation we will suppress the dependence on $\t,\v$ of $\mu_N, G_N^d$ and $\Delta G_N$ when we replace $ \mathbb P_N$ with $ \mathbb P^{\t, \v}_N$ in (\ref{M1}) and use the same letters. It will be clear from context, which formula we mean. 

The definition of the deformed measure $\mathbb P^{\t, \v}$ is motivated by the following observation. 
\begin{lem}\label{LemCum1} Let $\xi$ be a bounded random variable. For any $m\geq 1$ the $m$th mixed derivative 
\begin{equation}\label{eq:derivative_k}
\frac{\partial^m}{\partial t_1 \cdots \partial t_m}\mathbb
E_{\mathbb P^{\t, \v}_N}\left[\xi\right]\bigg\rvert_{t_a=0, 1\leq a\leq m}
\end{equation}
is the joint cumulant of the $m + 1$ random variables $\xi , NG^d_N(v_1),\dots ,NG^d_N(v_m)$ with respect to $\mathbb P_N$.
\end{lem}
\begin{rem}
The above result is analogous to Lemma 2.4 in \cite{BGG}, which in turn is based on earlier related work in random matrix theory \cite{Mi,Ey}. We present a proof below for the sake of completeness.
\end{rem}
\begin{proof}
One way to define the joint cumulant of $m+1$ bounded random variables $\xi_0, \dots, \xi_m$ is through
$$ \frac{\partial^{m+1}}{\partial t_0 \partial t_1 \cdots \partial t_m}  \log \left(\mathbb{E} \exp \left( \sum_{ i = 0}^m t_i \xi_i\right) \right) \Bigg\rvert_{t_i=0, 0\leq i\leq m}. $$
Performing the differentiation with respect to $t_0$ we can rewrite the above as
$$ \frac{\partial^{m}}{ \partial t_1 \cdots \partial t_m} \frac{\mathbb{E}  \left[\xi_0\exp \left( \sum_{ i = 1}^m t_i \xi_i\right) \right] }{\mathbb{E}  \left[\exp \left( \sum_{ i = 1}^m t_i \xi_i\right) \right]} \Bigg \rvert_{t_i=0, 1\leq i\leq m}. $$
Setting $\xi_0 = \xi$ and $\xi_i = NG^d_N(v_i)$ for $i = 1, \dots, m$ and observing that 
$$\exp \left( t NG^d_N(z)\right) = \prod_{i = 1}^N \left(1+ \frac{t}{z-\ell_i} \right) + O(t^2) \mbox{ as $t \rightarrow 0$, }$$
we obtain the desired statement.
\end{proof}

In the remainder of this section we introduce some notation from the theory of hyperelliptic integrals. We will require the latter to formulate our main result in the next section. 

Fix $k$ simple positively oriented contours $\gamma_1, \dots, \gamma_k$ such that each $\gamma_j$ encloses the segment $[\hat{a}_j, \hat{b}_j]$ (and thus also $[r_j, s_j]$ from Assumption 5) for $j = 1, \dots, k$. We assume that $\gamma_j$ are pairwise disjoint and do not enclose each other.

Let $P(z) = p_0 + p_1 z + \dots + p_{k-2} z^{k-2}$ be a polynomial of degree at most $k -2$, and define 
\begin{equation}\label{Omega}
\Omega: P(z) \rightarrow \left( \frac{1}{2\pi \iu} \oint_{\gamma_1} \frac{P(z)dz}{\prod_{j = 1}^k \sqrt{(z - r_j)(z - s_j)}}, \cdots, \frac{1}{2\pi \iu} \oint_{\gamma_k} \frac{P(z)dz}{\prod_{j = 1}^k \sqrt{(z - r_j)(z - s_j)}} \right).
\end{equation}
Notice that the sum of the integrals in (\ref{Omega}) equals (minus) the residue of $ \frac{P(z)dz}{\prod_{j = 1}^k \sqrt{(z - r_j)(z - s_j)}}$ at infinity, which is zero. Therefore, $\Omega$ defines a linear map between $(k-1)$-dimensional vector spaces. The map $\Omega$ is rather complicated, but it is known to be an isomorphism of vector spaces for $k \geq 2$ (see \cite[Section 2.1]{Dub}).

Using $\Omega$ we can now define a different map $\Upsilon_z$ as follows. The map $\Upsilon_z$ is defined in terms of the $k$ contours $\gamma_j$ and the points $r_j, s_j$ for $j = 1, \dots, k$. It is a linear map on the space of continuous functions $f(z)$ on $\gamma = \cup_{j = 1}^k \gamma_j$, whose integral over $\gamma$ is zero and is given by
\begin{equation}\label{Upsilon}
\Upsilon_z[f] = f(z) + \frac{P(z)}{\prod_{j = 1}^k \sqrt{(z - r_j)(z - s_j)}},
\end{equation}
where $P(z)$ is the unique polynomial of degree at most $k-2$ such that for each $j = 1, \dots, k$ we have
$$\oint_{\gamma_j} \Upsilon_z[f] dz = 0.$$
The polynomial $P(z)$ can be evaluated in terms of the map $\Omega$ via
\begin{equation}\label{Omegatop}
P = \Omega^{-1} \left(- \frac{1}{2\pi \iu} \oint_{\gamma_1} f(z) dz, \cdots, - \frac{1}{2\pi \iu} \oint_{\gamma_k} f(z) dz \right).
\end{equation}
We emphasize that the map $f \rightarrow \Upsilon_z[f]$ does not depend on $\t,\v$. 

We will require several properties of $\Upsilon_z$, which can easily be deduced from the above definitions. We summarize them in the following proposition without proof.
\begin{proposition}\label{PropUpsilon} The function $\Upsilon_z$ satisfies the following properties:
\begin{enumerate}
\item it is Lipschitz continuous in the uniform norm on the contours $\gamma_j$, $j = 1,\dots, k$;
\item if $\tilde{P}(z)$ is a polynomial of degree at most $k-2$ then $\Upsilon_z\left[ \frac{\tilde{P}(z)}{\prod_{j = 1}^k \sqrt{(z - r_j)(z - s_j)}}\right] = 0$;
\item if $\Upsilon^N_z$ is defined in terms of $r_j(N), s_j(N)$ and $r_j(N) - r_j = O(N^{-1}\log(N)) = s_j (N) - s_j $ for $j = 1, \dots, k$ then for any $f$ we have $\Upsilon^N_z[f] - \Upsilon_z[f] = O(N^{-1} \log(N))$.
\end{enumerate}
\end{proposition}

\subsection{Main result} \label{Section5.2} At this time we isolate the main technical result we prove about $\Delta G_N(z)$ and deduce a couple of easy corollaries from it. We continue with the notation from the previous section.

\begin{thm}\label{MainTechT} Suppose Assumptions 1-6 in Section \ref{Section2.2} hold. Let $U : = \mathbb{C} \setminus \{ \cup_{j = 1}^k [\hat{a}_j, \hat{b}_j] \cup \{ \pm 2\sqrt{\lu} \} \}$ and $\Gamma = \cup_{j = 1}^k \gamma_j \subset U$, where each $\gamma_j$  is a positively oriented contour that encloses the segment $[\hat{a}_j, \hat{b}_j]$ for $j = 1, \dots, k$,  $\gamma_j$ are pairwise disjoint and do not enclose each other. We set $U^u$ to be the single unbounded component of $U \setminus \Gamma$.

Fix $m \in \mathbb{N}$ and $v_0,\dots,v_m \in U^u$. For $m \geq 2$ we have
\begin{equation}\label{MT0}
 \frac{\partial^m}{\partial t_1\cdots \partial t_m} \mathbb E_{\mathbb P^{\t, \v}_N}\left[\Delta G_N(v_0) \right]\bigg\rvert_{t_a=0, 1\leq a\leq m}= O(N^{-1}\log(N)),
\end{equation}
while for $m = 1$
 \begin{equation}\label{MT1}
\begin{split}
&\frac{\partial}{\partial t_1}  \Et \left[\Delta {G}_N(v_0)\right]  \bigg\rvert_{t_1=0} = O(N^{-1} \log (N)) + \\
&\Upsilon_{v_0} \left[ \frac{1}{4\theta \pi \iu \cdot  \prod_{j = 1}^k \sqrt{\left(v_0-   \hat{r}_j\right)\left(v_0  -  \hat{s}_j\right)}} \oint_{\Gamma} \frac{  \prod_{j = 1}^k \sqrt{\left(z-   \hat{r}_j\right)\left(z  -  \hat{s}_j\right)} }{(z - v_1 )^2(z - v_0  )} dz.\right]
\end{split}
\end{equation}
In the above $\Upsilon_{v_0}$ is as in (\ref{Upsilon}) for the contours $\gamma_j$ and the points $\hat{r}_j, \hat{s}_j$ for $j = 1, \dots, k$ as in Assumption 5. 
Finally, the constants in the big $O$ notation are uniform as $v_0, v_1, \dots, v_k$ vary over compact subsets of $U^u$.
\end{thm}
\begin{rem} We will prove Theorem \ref{MainTechT} for the case $\lu > 0$. The case $\lu = 0$ can be handled with minor modifications of the argument. 
\end{rem}

\begin{thm} \label{TGField}
Assume the same notation as in Theorem \ref{MainTechT}. As $N \rightarrow \infty$, the random field $N(G^d_N(z)-\mathbb E_{\mathbb P_N} \left[G^d_N(z)\right]), z \in U$, converges  (in the sense of joint moments, uniformly in $z$ in compact subsets of $U$) to a centered
complex Gaussian random field with second moment
\begin{equation}\label{eq:GField}
\lim_{N\rightarrow \infty} N^2 \left(\mathbb E_{\mathbb P_N} \left[G^d_N(z_1) G^d_N(z_2)\right]-\mathbb E_{\mathbb P_N} \left[G^d_N(z_1)\right] \mathbb E_{\mathbb P_N} \left[G^d_N(z_2)\right] \right)=: \mathcal C_\theta(z_1, z_2), \mbox{ where }
\end{equation}

\begin{equation}\label{eq:var}
\begin{split}
 \mathcal C_\theta(z_1, z_2) = \theta^{-1} \cdot \Upsilon_{z_2} \Bigg[&-\frac{1}{2(z_1-z_2)^2} +\prod_{j = 1}^k\frac{\sqrt{(z_1 - \hat{r}_j)(z_1 - \hat{s}_j)}}{\sqrt{(z_2 - \hat{r}_j)(z_2 - \hat{s}_j)}} \times \\
&\Bigg( \frac{1}{(z_1 - z_2)^2} - \frac{1}{2 (z_1 - z_2)} \sum_{j = 1}^k \left( \frac{1}{z_1 - \hat{r}_j} +  \frac{1}{z_1 - \hat{s}_j}\right)\Bigg) \Bigg].
\end{split}
\end{equation}
\end{thm}
\begin{rem}
Since $\overline{G_N(z)} = G_N(\overline{z})$, we can use (\ref{eq:var}) to completely characterize the asymptotic covariance of the (recentered) random field $G_N(z)$.
\end{rem}
\begin{rem}\label{meq1} When $m = 1$ the covariance $\mathcal{C}_\theta(z_1, z_2)$ can be written down explicitly as
$$\mathcal{C}_\theta(z_1, z_2) = -\frac{\theta^{-1}}{2(z_1-z_2)^2} \left(1 - \frac{(z_1 - \hat{r}_1)(z_2- \hat{s}_1) + (z_2 - \hat{r}_1 )(z_1- \hat{s}_1)}{2\sqrt{(z_1 -\hat{r}_1 )(z_1- \hat{s}_1)}\sqrt{(z_2 -\hat{r}_1 )(z_2- \hat{s}_1 )}} \right).$$
When $m = 2$ one can also find an explicit form for $\mathcal{C}_\theta(z_1, z_2)$, involving the values of complete elliptic integrals, but we do not pursue it here, cf. \cite{BDE}.
\end{rem}
\begin{rem}
In the continuous log-gas setting the covariance has the same form as (\ref{eq:var}), cf. \cite{J4, S, BoGu2}. A similar result also holds for the discrete $\beta$-ensembles in \cite{BGG}.
\end{rem}

\begin{proof} Fix $v_0, \dots, v_m \in U$ and $\Gamma$ as in Theorem \ref{MainTechT} so that $v_0, \dots, v_m \in U^u$. Setting $\xi = \Delta G_N(v_0)$  in Lemma \ref{LemCum1} we know that the joint cumulant of $ \Delta G_N(v_0), NG_N^d(v_1) , \dots, NG_N^d(v_m)$ is given by
$$\frac{\partial^m}{\partial t_1\partial t_2\cdots \partial t_m}\mathbb E_{\mathbb P^{\t, \v}_N}\left[ \Delta G_N(v_0)\right]\bigg\rvert_{t_a=0, 1\leq a\leq m}.$$
Since cumulants remain unchanged under constant shifts, we see that the above formula is also the joint cumulant of $N(G^d_N(v_i)-\mathbb E_{\mathbb P_N} \left[G^d_N(v_i)\right])$ for $i = 0, \dots, m$. From Theorem \ref{MainTechT} we see that as $N \rightarrow \infty$ all $3$rd and higher order cumulants vanish, which proves the asymptotic Gaussianity of the field $N(G^d_N(z)-\mathbb E_{\mathbb P_N} \left[G^d_N(z)\right])$.

As $N(G^d_N(z)-\mathbb E_{\mathbb P_N} \left[G^d_N(z)\right])$ are centered for each $N$ so is the limiting field. From (\ref{MT1}) we also have the following formula for the limiting covariance (which is the second joint cumulant) of $N(G^d_N(z_1)-\mathbb E_{\mathbb P_N} \left[G^d_N(z_1)\right])$ and $N(G^d_N(z_2)-\mathbb E_{\mathbb P_N} \left[G^d_N(z_2)\right])$ for $z_1, z_2 \in U^u$
$$ \Upsilon_{z_2} \left[ \frac{1}{4 \theta \pi \iu \cdot  \prod_{j = 1}^k \sqrt{\left(z_2 -   \hat{r}_j\right)\left(z_2  -  \hat{s}_j\right)}} \oint_{\Gamma} \frac{  \prod_{j = 1}^k \sqrt{\left(z-   \hat{r}_j\right)\left(z  -  \hat{s}_j\right)} }{(z - z_1 )^2(z - z_2  )} dz\right].$$
 Evaluating and adding the (minus) residues at $z = z_1$ and $z = z_2$ we obtain (\ref{eq:var}).
\end{proof}

\begin{thm}\label{CLTfun}
Assume the same notation as in Theorem \ref{MainTechT}. For $m\geq 1$ let $f_1, \dots, f_m$ be real analytic functions in $U$ and define 
$$\mathcal L_{f_i}=N \int_\mathbb{R} f_j(x) \mu_N(dx) -N \mathbb E _{\mathbb P_N}  \left[ \int_\mathbb{R} f_j(x) \mu_N(dx) \right] \mbox{ for $i = 1, \dots, m$}.$$
Then the random variables $\mathcal{L}_{f_i}$ converge jointly in the sense of moments to an $m$-dimensional centered Gaussian vector $X = (X_1,\dots, X_m)$ with covariance
$$Cov(X_i, X_j) = \frac{1}{(2 \pi \iu)^2}\oint_{\Gamma} \oint_{\Gamma } f_i(s)f_j(t) \mathcal C_\theta (s, t) ds dt \mbox{ with $C_\theta(s,t)$ as in (\ref{eq:var}).}$$
\end{thm}
\begin{proof}
Observe that when $f$ is real analytic in $U$ we have for all large $N$
$$\mathcal{L}_f = \frac{N}{2\pi \iu} \oint_\Gamma f(z) (G_N(z) - \E\left[ G_N(z) \right] )dz,$$
where $\Gamma$ is as in Theorem \ref{MainTechT}. Therefore, for any joint moment of $\mathcal{L}_{f_i}$ we have
\begin{equation}\label{QRCLTeq1}
 \E \left[ \mathcal{L}_{ f_{i_1} }\hspace{-2mm}\cdots \mathcal{L}_{ f_{i_k} }\right] = \frac{1}{(2\pi \iu)^k} \oint_\Gamma \cdots \oint_\Gamma \hspace{-1mm}  \E   \hspace{-1.5mm}\left[\prod_{m = 1}^k  \hspace{-1.5mm} N(G_N(z_m) - \E\left[ G_N(z_m) \right] )\right] \hspace{-1mm}  \prod_{m = 1}^k  \hspace{-1mm} f_{i_m}(z_m)   dz_m. 
\end{equation}
Since cumulants of centered random variables are linear combinations of moments and vice versa, we conclude that all third and higher order cumulants of $\mathcal{L}_{f_i}$ vanish as $N \rightarrow \infty$ (here we used Theorem \ref{TGField}, which implies the third and higher order joint cumulants of $ N(G_N(z_i) - \E\left[ G_N(z_i) \right] )$ vanish uniformly when $z_i \in \Gamma$). This proves the Gaussianity of the limiting vector $X$. Since $\mathcal{L}_{f_i}$ are centered for each $N$ the same is true for $X$. To get $Cov(X_i,X_j)$ we can set $k = 2$, $i_1 = i$ and $i_2 = j$ in (\ref{QRCLTeq1}) and send $N \rightarrow \infty$. In view of (\ref{eq:GField}) we conclude that
$$Cov(X_i, X_j) = \frac{1}{(2 \pi \iu)^2}\oint_{\Gamma} \oint_{\Gamma } f_i(s)f_j(t) \mathcal C_\theta (s, t) ds dt, \mbox{ where $\mathcal{C}_\theta(s,t)$ is as in (\ref{eq:var}). }$$

\end{proof}

\subsection{Application of Nekrasov's equation}\label{Section5.3}
 In this section we begin the proof of Theorem \ref{MainTechT} emphasizing the contribution of the Nekrasov's equation. In what follows we use the same notation as in the previous section and Section \ref{Section2.2}, dropping the dependence on $N$ from parameters.

The first key observation we make is that $\mathbb P^{\t, \v}_N$ satisfies Nekrasov's equation with
\begin{equation}\label{PhiN}
\begin{split}
&\Phi_N^{+, \t, \v}(z)=\Phi^+_N(z) \prod^m_{a=1}\left[(v_a+t_a-\sigma_N(z))(v_a-\sigma_N(qz))\right] ,\\
&\Phi_N^{-, \t, \v}(z)=\Phi^-_N(z)\prod^m_{a=1}\left[(v_a+t_a-\sigma_N(qz))(v_a-\sigma_N(z))\right], \mbox{ where }\\
\end{split}
\end{equation}
$\Phi^{\pm}_N$ are as in Assumption 4 and we recall that $\sigma_N(z) = z + uz^{-1}$. Notice that $\Phi_N^{\pm, \t, \v}(z)$ are also analytic in $\mathcal{M}$. Denoting the RHS of the Nekrasov's equation for the measure $\mathbb P^{\t, \v}_N$ by $R_N^{\t, \v}(z)$ we see from Theorem \ref{NekGen} that $R_N^{\t, \v}(z)$ is analytic in $\mathcal{M}$. 

Using $q = \tq^{1/N} + O(N^{-2})$ and $u = \lu + O(N^{-1})$ we obtain asymptotic expansions for $\Phi_N^{\pm, \t, \v}$
\begin{equation}\label{S1}
\begin{split}
&\Phi_N^{\pm, \t, \v}(z)=[ \Phi^\pm(z)  +\Ps_N^\pm(z)] \prod^m_{a=1}\left[(v_a+t_a- \sigma(z))(v_a-\sigma(z) )\right] + \R^{\pm, \t, \v}_{1,N}(z), \\
& \Ps_N^+(z)  := \varphi^{+}_N(z) -   \Phi^+(z) \cdot \sum_{a = 1}^m \left[\frac{(\log \lq/N) \left(z -\lu z^{-1} \right) - (\lu - u)z^{-1}}{v_a - \sigma(z)} - \frac{ (\lu - u)z^{-1}}{v_a + t_a - \sigma(z)} \right] ,\\ 
& \Ps_N^-(z)  :=  \varphi^{-}_N(z) -   \Phi^-(z) \cdot \sum_{a = 1}^m \left[\frac{(\log \lq/N) \left(z -\lu z^{-1} \right) - (\lu - u)z^{-1}}{v_a + t_a - \sigma(z)} - \frac{ (\lu - u)z^{-1}}{v_a  - \sigma(z)} \right],
\end{split}
\end{equation}
where $\sigma(z) = z + \lu z^{-1}$and $\R^{\pm, \t, \v}_{1,N}(z) = O(N^{-2})$ uniformly over compact subsets of $\mathcal{M}$.

The second important observation is that we have the following asymptotic expansion
\begin{equation} \label{S2}
\begin{split}
&\prod\limits^N_{i=1}\frac{\sigma_N(q^{\theta} z)-\ell_i}{\sigma_N(z)- \ell_i}=\exp \left[\theta  \mathfrak{G}_N^c(z)+ \theta  \Delta \mathfrak{G}_N(z)+W^-_N(z)+\R^{-}_{2,N}(z) \right],\\
&\prod\limits^N_{i=1} \frac{\sigma_N(q^{1-\theta}z)- \ell_i}{\sigma_N(qz)-\ell_i} = \exp \left[-\theta\mathfrak{G}_N^c(z)- \theta\Delta \mathfrak{G}_N(z)+W^+_N(z)+\R^{+}_{2,N}(z) \right],
\end{split}
\end{equation}
\begin {equation}\label{eq:gothic_G}
\begin{split}
\mbox{ where } &\mathfrak{G}_N^{d/ c}(z)= N \log q \cdot (z - \lu z^{-1})  \cdot G^{d/c}_N(z + \lu  z^{-1}),\\
&\Delta \mathfrak{G}_N(z) = \mathfrak{G}^d_N(z) - \mathfrak{G}^c_N(z) \mbox{ with $G_N^c$ and $G_N^d$ as in (\ref{RegG});}
\end{split}
\end{equation}
\begin {equation}\label{eq:Ws}
\begin{split}
&W^-_N(z) =  \theta \log q   \left[ z (\theta/2) \cdot  \partial_z \mathfrak{G}^d_N(z) + (u - \lu) N \cdot \left[ \partial_z G^c_N(z + \lu z^{-1}) -   z^{-1}G_N^c (z + \lu z^{-1}) \right] \right],\\
&W^+_N(z) =  \theta \log q   \left[ z (\theta/2 - 1)   \partial_z \mathfrak{G}^d_N(z)  - (u - \lu) N  \left[ \partial_z G^c_N(z + \lu z^{-1}) -   z^{-1}G_N^c (z + \lu z^{-1}) \right] \right].
\end{split}
\end{equation}
The remainders $\R^{\pm}_{2,N}(z) = O(N^{-2})$ are uniform in $z$ on compact subsets of $\mathcal{O}$, which is the inverse image of $U$ under the map $z \rightarrow z + \lu z^{-1}$. Explicitly, if $\hat{c}_j, \hat{d_j}$ are the points in $[1, \lq^{-\lM}]$ with 
$\hat{c}_j+ \lu \hat{c}_j^{-1} = \hat{a}_j \mbox{ and }\hat{d}_j+ \lu \hat{d}^{-1}_j = \hat{b}_j $ then $\mathcal{O} := \mathbb C\setminus \left\{  \{0,  \sqrt{\lu}, -\sqrt{\lu} \} \cup  \cup_{j = 1}^k [\hat{c}_j, \hat{d}_j] \cup \cup_{j = 1}^k [\lu \hat{d}^{-1}_j, \lu \hat{c}^{-1}_j] \right\}$.

The third observation we require comes from Assumption 5 and Lemma \ref{AnalRQ}, which imply:
\begin{equation}\label{RFE}
\begin{split}
&\mbox{ If }R_N(z) := \Phi^-(z)  e^{\theta \mathfrak{G}^c_N(z)} +  \Phi^+(z)  e^{- \theta\mathfrak{G}^c_N(z)},\mbox{ \hspace{1mm} }Q_N(z) := \Phi^-(z)  e^{\theta\mathfrak{G}^c_N(z)} -  \Phi^+(z)  e^{-\theta\mathfrak{G}^c_N(z)}, \\
&\mbox{ then $R_N$ is analytic in $\mathcal{M}$ and } Q_N(z) = H_N(z) \cdot \prod_{j = 1}^k \sqrt{\left(\sigma(z) -   r_j \right)\left(\sigma(z)  - s_j \right)}, 
\end{split}
\end{equation}
where $H_N$ is analytic and non-vanishing in $\mathcal{M}$.\\

We detail the consequence of the above three observation. Nekrasov's equation for $\mathbb P^{\t, \v}_N$ reads
$$R_N^{\t, \v}(z) =\Phi_N^{-, \t, \v}(z) \cdot \Et  \left[\prod\limits^N_{i=1} \frac{\sigma_N(q^{\theta}z)-\ell_i}{\sigma_N(z)-\ell_i} \right] + \Phi_N^{+, \t, \v}(z) \cdot\Et \left[\prod\limits^N_{i=1} \frac{\sigma_N(q^{1-\theta}z)- \ell_i}{\sigma_N(qz)-\ell_i}\right].$$
Combining the above with (\ref{S1}), (\ref{S2}) and (\ref{RFE}) we conclude that
\begin{equation}\label{S3}
\begin{split}
&R_N^{\t, \v}(z) =\R^{\t, \v}_{N}(z)+\prod^m_{a=1}[(v_a+t_a-\sigma(z))(v_a-\sigma(z))]\times {\Big [} \Et \left[\theta\Delta \mathfrak{G}_N(z)\right] \cdot  Q_N(z) \\
&+ \Ps^-_N(z)e^{\theta\mathfrak{G}^c_N(z)}  +  \Ps^+_N(z)e^{-\theta\mathfrak{G}^c_N(z)}+  \tilde{W}_N(z) + R_N(z)\Big{]}, \mbox{ where }
\end{split}
\end{equation}
\begin{equation}\label{XDef2}
\begin{split}
& \tilde{W}_N(z)  = \Phi^-(z)e^{\theta\mathfrak{G}^c_N(z)}\tilde{W}^-_N(z) + \Phi^+(z)e^{-\theta\mathfrak{G}^c_N(z)}\tilde{W}^+_N(z) \mbox{, with }\\
&\tilde{W}^-_N(z) =  \theta \log q   \left[ z (\theta/2)   \partial_z \mathfrak{G}^c_N(z) + (u - \lu) N \cdot \left[ \partial_z G^c_N(z + \lu z^{-1}) -   z^{-1}G_N^c (z + \lu z^{-1}) \right]\right], \\
&\tilde{W}^+_N(z) =  \theta \log q   \left[ z (\theta/2 - 1)   \partial_z \mathfrak{G}^c_N(z)  - (u - \lu) N  \left[ \partial_z G^c_N(z + \lu z^{-1}) -   z^{-1}G_N^c (z + \lu z^{-1}) \right] \right],\\
& \R^{\t, \v}_{N}(z) = \R^{+, \t, \v}_{1,N}(z) \Et \left[ \prod\limits^N_{i=1}\frac{\sigma(q^{\theta} z)-\ell_i}{\sigma(z)- \ell_i}\right] +  \R^{-, \t, \v}_{1,N}(z) \Et \left[\frac{\sigma_N(q^{1-\theta}z)- \ell_i}{\sigma_N(qz)-\ell_i} \right] + \\
&\prod^m_{a=1}[(v_a+t_a-\sigma(z))(v_a-\sigma(z))] \times \left[ A_N + B_N + C_N + D_N \right] \mbox{ and }
\end{split}
\end{equation}
\begin{equation*}
\begin{split}
A_N = & \hspace{2mm} \Phi^-(z) \cdot \Et \left[ \prod\limits^N_{i=1}\frac{\sigma_N(q^{\theta} z)-\ell_i}{\sigma_N(z)- \ell_i} - e^{\theta\Delta \mathfrak{G}_N(z) + \theta\mathfrak{G}^c_N(z)} - e^{\theta\mathfrak{G}^d_N(z)}\cdot \tilde{W}^-_N(z)\right] + \\
& \hspace{2mm} \Phi^+(z) \cdot \Et \left[\frac{\sigma_N(q^{1-\theta}z)- \ell_i}{\sigma_N(qz)-\ell_i} - e^{-\theta\Delta \mathfrak{G}_N(z) - \theta\mathfrak{G}^c_N(z)} -e^{-\theta\mathfrak{G}^d_N(z)} \cdot  \tilde{W}^+_N(z)\right], 
\end{split}
\end{equation*}
\begin{equation*}
\begin{split}
 B_N = &  \hspace{2mm} \Phi^-(z) \tilde{W}^-_N(z)  \cdot \Et \left[  e^{\theta\mathfrak{G}^d_N(z)}  -  e^{\theta\mathfrak{G}^c_N(z)} \right]  \Phi^+(z)  \tilde{W}^+_N(z)  \cdot \Et \left[  e^{-\theta\mathfrak{G}^d_N(z)}- e^{-\theta\mathfrak{G}^c_N(z)} \right],
\end{split}
\end{equation*}
\begin{equation*}
\begin{split}
C_N =  &  \hspace{2mm}\Phi^-(z)\cdot e^{ \theta\mathfrak{G}^c_N(z)} \cdot \Et \left[ e^{\theta\Delta \mathfrak{G}_N(z)}  - \theta\Delta \mathfrak{G}_N(z) - 1 \right] + \\
& \Phi^+(z)\cdot e^{ -\theta\mathfrak{G}^c_N(z)} \cdot \Et \left[e^{-\theta\Delta \mathfrak{G}_N(z)} +  \theta\Delta \mathfrak{G}_N(z) - 1 \right],
\end{split}
\end{equation*}
\begin{equation*}
\begin{split}
D_N =  &  \hspace{2mm} \Psi^-_N(z) \cdot\Et \left[\frac{\sigma_N(q^{\theta} z)-\ell_i}{\sigma_N(z)- \ell_i}  - e^{\theta\mathfrak{G}^c_N(z)} \right] + \Psi^+_N(z) \cdot\Et \left[\prod\limits^N_{i=1}\frac{\sigma_N(q^{1-\theta}z)- \ell_i}{\sigma_N(qz)-\ell_i} - e^{-\theta\mathfrak{G}^c_N(z)} \right] .
\end{split}
\end{equation*}

Let $\gamma_1, \dots, \gamma_k$ be as in the statement of the theorem and $v_0, \dots, v_m$ lie outside of $\Gamma = \cup_{j = 1}^k \gamma_j$. We let $\gamma'_j, \gamma''_j$ for $j = 1,\dots, k$ be a positively oriented contours such that for each $j = 1, \dots, k$
\begin{itemize}
\item $\gamma'_j$  encloses the interval $[\hat{c}_j, \hat{d}_j] $ and excludes the points $\pm \sqrt{\lu }$ and $0$,
\item $\sigma(\gamma'_j)$ is contained in the bounded component of $\mathbb{C} \setminus \gamma_j$ ,
\item $\gamma''_j:= \omega(\gamma'_j)$, where $\omega(z) = \frac{\lu}{z}$,
\item $\{\gamma'_j, \gamma''_j \}_{j = 1}^k$ are all disjoint and contained in $\mathcal{M}$.
\end{itemize}
The existence of such contours is ensured by our assumptions on $\gamma_j$. Observe that by construction $\gamma''_j$ is a positively oriented contour that also excludes the points $\pm \sqrt{\lu}$ and $0$, and encloses the interval $[\lu \hat{d}^{-1}_j, \lu \hat{c}^{-1}_j]$. For convenience we let $\Gamma_1 = \cup_{j = 1}^k \gamma'_j$ and $\Gamma_2 = \cup_{j = 1}^k \gamma''_j$.

 We divide both sides of \eqref{S3} by 
$$ 2 \pi \iu \cdot z\cdot (v_0  -\sigma(z)) \cdot \prod^m_{a=1}[(v_a+t_a-\sigma(z))(v_a-\sigma(z))] \cdot H_N(z)$$
and integrate over $\Gamma' : = \Gamma_1 \cup \Gamma_2$. Note that $R_N^{\t, \v}(z)$ and $ R_N(z)$ are both holomorphic inside the contours $\Gamma_1, \Gamma_2$ and so the integrals of the corresponding terms vanish. From the rest we get
\begin{equation}\label{S5}
\begin{split}
&\frac{1}{2\pi \iu} \oint_{\Gamma'}\frac{H_N(z)^{-1} z^{-1} Q_N(z)}{v_0 -\sigma(z)} \Et \left[\theta \Delta \mathfrak{G}_N(z)\right] dz =  \\
&\frac{-1}{2\pi \iu} \oint_{\Gamma'} \frac{H_N(z)^{-1} z^{-1}dz}{v_0 -\sigma(z)}\left[  \frac{ \R^{\t, \v}_{N}(z)}{\prod^m_{a=1}[(v_a+t_a-\sigma(z))(v_a-\sigma(z))] } \right]  \\
&+ \frac{-1}{2\pi \iu}  \oint_{\Gamma'} \frac{H_N(z)^{-1} z^{-1} dz}{ v_0 -\sigma(z) }\left[ \Ps^-_N(z)e^{\theta \mathfrak{G}^c_N(z)} +  \Ps^+_N(z)e^{-\theta\mathfrak{G}^c_N(z)}  + \tilde{W}_N(z) \right] 
\end{split}
\end{equation}
Equation (\ref{S5}) can be viewed as the main output of our application of the Nekrasov's equation. In the following section we use it to deduce Theorem \ref{MainTechT}. 

\subsection{Concluding the proof of Theorem \ref{MainTechT}}\label{Section5.4}
In this section we present the remainder of the proof of Theorem \ref{MainTechT}. Our arguments below will require a certain asymptotic bound -- see (\ref{S8}), which will be established in Section \ref{Section6}. For clarity we split the proof into several steps.\\

{\raggedleft {\bf Step 1.}} Our goal in this step is to rewrite (\ref{S5}) into a form that is more useful for our analysis.

Using the formula for $Q_N$ from (\ref{RFE}) and that $\Delta \mathfrak{G}_N(z) = \log q \cdot (z - \lu z^{-1}) \Delta G_N(\sigma(z))$ from (\ref{eq:gothic_G})  we see that the RHS of (\ref{S5}) equals
$$\frac{\log q }{2\pi \iu} \oint_{\Gamma'}\frac{\prod_{j = 1}^k\sqrt{\left(\sigma(z) -   r_j\right)\left(\sigma(z)  -  s_j\right)}}{v_0  -\sigma(z) } \cdot (1 -  \lu z^{-2})  \cdot \Et \left[\theta\Delta {G}_N(\sigma(z))\right] dz $$
We perform a change of variables $\sigma(z) = w$ to rewrite the RHS of (\ref{S5}) as
$$2 \cdot \frac{\log q }{2\pi \iu} \oint_{\Gamma}\frac{\prod_{j = 1}^k\sqrt{\left(w -   r_j\right)\left(w  -  s_j\right)}}{v_0  -w }  \cdot \Et \left[\theta\Delta {G}_N(w)\right] dw, $$
where we used that $\sigma(\gamma_j') = \sigma(\gamma_j'')$ are contained $\gamma_j$ and we can deform the image to the latter without affecting the value of the integral by Cauchy's theorem.

Note that $\Et \left[\Delta {G}_N(w)\right]$ is analytic outside of the contour of integration and decays like $1/w^2$ when $|w| \rightarrow \infty$. Therefore, we can compute the integral as (minus) the residues at $w = v_0$ and $z = \infty$. The residue at $v_0$ is given by
$$ -2\cdot \log q  \cdot \prod_{j = 1}^k\sqrt{\left(v_0-   r_j\right)\left(v_0  -  s_j\right)} \cdot\Et \left[\theta\Delta {G}_N(v_0)\right],$$
while the residue at $\infty$ is a polynomial $P^{\t,\v}_N(v_0)$ of degree at most $k-2$ in $v_0$, whose coefficients are rational functions in $\t, \v$. Substituting the above in (\ref{S5}) we obtain the formula
\begin{equation}\label{S7}
\begin{split}
&\Et \left[\theta \Delta {G}_N(v_0)\right] =  \frac{NP^{\t,\v}_N(v_0) }{D_N(v_0)}+  \frac{N}{2\pi \iu} \oint_{\Gamma'} \frac{-D_N(v_0)^{-1}dz}{zH_N(z)(v_0 -\sigma(z))}\left[   \frac{ \R^{\t, \v}_{N}(z)}{W^{\t, \v}_m( z) } \right] + \\
&\frac{N}{2\pi \iu} \oint_{\Gamma'} \frac{-D_N(v_0)^{-1}dz}{z H_N(z)(v_0 -\sigma(z))}\left[ \Ps^-_N(z)e^{\theta \mathfrak{G}^c_N(z)}  +  \Ps^+_N(z)e^{-\theta \mathfrak{G}^c_N(z)} + \tilde{W}_N(z) \right],
\end{split}
\end{equation}
where 
\begin{equation}\label{S75}
\begin{split}
&D_N(v)= 2N  \log q   \prod_{j = 1}^k \sqrt{\left(v -   r_j\right)\left(v -  s_j\right)} \mbox{ and }  W^{\t, \v}_m( z) = \prod^m_{a=1}[(v_a+t_a-\sigma(z))(v_a-\sigma(z))].
\end{split}
\end{equation}

{\raggedleft {\bf Step 2.}} In this step we isolate an asymptotic estimate that we require to finish the proof.

\begin{fact}\label{fact_6} For each $m \geq 0$ we have
\begin{equation}\label{S8}
\frac{\partial^m}{\partial t_1 \cdots \partial t_m}   \frac{1}{2\pi \iu} \oint_{\Gamma'} \frac{D_N(v_0)^{-1}dz}{zH_N(z)(v_0- \sigma(z))}\left[ \frac{ \R^{\t, \v}_{N}(z)}{W^{\t, \v}_m( z) } \right] \Bigg\rvert_{t_a=0, 1 \leq a \leq m} = O(N^{-2}),
\end{equation}
where the constant in the big $O$ notation is uniform as $v_0, v_1,\dots,v_k$ vary over compacts in $U^u$.
\end{fact}
The proof of Fact \ref{fact_6} will be presented in Section \ref{Section6}. In the remainder of the section we assume its validity and finish the proof of Theorem \ref{MainTechT}.\\

{\raggedleft {\bf Step 3.}} Let us fix $m \geq 1$, differentiate both sides of (\ref{S7}) with respect to $t_1, \dots, t_m$ and set $t_a = 0$ for $a = 1, \dots, m$. Using (\ref{S8}) we get
\begin{equation}\label{S9}
\begin{split}
&\partial_{t_1}\cdots \partial_{t_m}  \Et \left[ \theta\Delta {G}_N(v_0)\right]  \Bigg\rvert_{t_a=0, 1 \leq a \leq m} \hspace{-15mm}=   \partial_{t_1}\cdots \partial_{t_m}     \Bigg[\frac{NP^{\t,\v}_N(v) }{D_N(v_0)}+ \frac{1}{2\pi \iu} \oint_{\Gamma'} \frac{-D_N(v_0)^{-1}dz}{zH_N(z) (v_0-\sigma(z))}\\
&\times \left[ N\Ps^-_N(z)e^{\theta \mathfrak{G}_N^c(z)}  +  N\Ps^+_N(z)e^{-\theta \mathfrak{G}_N^c(z)}+ \tilde{W}_N(z) \right]\Bigg] \Bigg\rvert_{t_a=0, 1 \leq a \leq m}   + O(N^{-1}).
\end{split}
\end{equation}

The only functions in (\ref{S9}), which depend on $\t$ are $P^{\t,\v}_N(v), \Ps^{\pm}_N(z)$, see (\ref{S1}). Since any mixed partial derivatives of $\Ps_N^{\pm}(z)$ vanish, we conclude that for $m \geq 2$ we have
\begin{equation}\label{S91}
\partial_{t_1}\cdots \partial_{t_m} \mathbb E_{\mathbb P^{\t, \v}_N}\left[\theta \Delta G_N(v_0) \right]\bigg\rvert_{t_a=0,
1\leq a\leq m}=  \partial_{t_1}\cdots \partial_{t_m}   \Bigg[\frac{NP^{\t,\v}_N(v_0) }{D_N(v_0)}\Bigg] \Bigg\rvert_{t_a=0, 1 \leq a \leq m}   \hspace{-7mm} +  O(N^{-1}).
\end{equation}

We may now apply $\Upsilon^N_{v_0}$ from (\ref{Upsilon}) for the contours $\gamma_1, \dots, \gamma_k$ and the points $r_j, s_j$ for $j = 1, \dots, k$ to both sides of the above equation. Indeed, we notice that the integral of $G_N^d$ around $\gamma_i$ is deterministic and equals $ n_i(N)/ N$. On the other hand, the integral of $G_N^c$ around $\gamma_i$ equals the total mass of $\hat\mu_N(x)$ inside $\gamma_i$, which is $ n_i(N)/N$ by assumption. We conclude that the integral of $\Delta G_N(v) $ around each loop $\gamma_i$ vanishes. The integral over the first term on the right side vanishes by Property (2) in Proposition \ref{PropUpsilon}. By linearity, we see that the integral over the term represented by $O(N^{-1})$ over $\Gamma$ must also vanish. Applying $\Upsilon^N_{v_0}$ and using Property (1) in Proposition \ref{PropUpsilon} we get
$$\partial_{t_1}\cdots \partial_{t_m} \mathbb E_{\mathbb P^{\t, \v}_N}\left[\Delta G_N(v_0) \right]\bigg\rvert_{t_a=0,
1\leq a\leq m}=   O(N^{-1}),$$
which proves the case $m \geq 2$.

If $m = 1$
\begin{equation}\label{SQ1}
\begin{split}
&\partial_{t_1}  \Et \left[\theta \Delta {G}_N(v_0)\right]  \Bigg\rvert_{t_1=0} = O(N^{-1}) + \partial_{t_1}\frac{NP^{\t,\v}_N(v_0) }{D_N(v_0)}  \Bigg\rvert_{t_1=0}+ \frac{-D_N(v_0)^{-1}}{2\pi \iu} \oint_{\Gamma'} f(z) dz,\\
&\mbox{ where }f(z) :=   \frac{\Phi^-(z)e^{\theta\mathfrak{G}_N^c(z)} \left[(1 - \lu z^{-2}) \log \lq   - (\lu - u) z^{-2} \right] - \Phi^+(z) e^{-\theta\mathfrak{G}_N^c(z)} (\lu - u) z^{-2}}{H_N(z)(v_0-\sigma(z))(v_1 - \sigma(z))^2}.
\end{split}
\end{equation}
Applying (\ref{RFE}) we obtain
$$\frac{-D_N(v_0)^{-1}}{2\pi \iu} \oint_{\Gamma'} f(z) dz =  \frac{-D_N(v_0)^{-1}}{2\pi \iu} \oint_{\Gamma'}  \frac{ 2^{-1} \log \lq [R_N(z) + Q_N(z)](1 - \lu z^{-2}) -  (\lu - u) z^{-2} R_N(z)}{H_N(z)(v_0-\sigma(z))(v_1 - \sigma(z))^2}dz .$$
Notice that the terms with $R_N(z)$ integrate to $0$ by analyticity, and so we may remove them. Substituting $Q_N(z)$ from (\ref{RFE}) and $D_N(v_0)$ from (\ref{S75}) we get
\begin{equation*}
\begin{split}
\frac{-D_N(v_0)^{-1}}{2\pi \iu} \oint_{\Gamma'} f(z) dz =  \frac{-1}{8\pi \iu\cdot  \prod_{j = 1}^k \sqrt{\left(v_0 -   r_j\right)\left(v_0  -  s_j\right)}} \times \\
\oint_{\Gamma'}  \frac{ \prod_{j = 1}^k \sqrt{\left(\sigma(z) -   r_j \right)\left(\sigma(z)  -  s_j \right)}}{(v_0-\sigma(z))(v_1 - \sigma(z))^2} \cdot  (1 - \lu z^{-2})dz .
\end{split}
\end{equation*}
We perform the change of variables $w = \sigma(z)$ and deform the resulting contours to $\Gamma$ to obtain
\begin{equation}\label{S92}
\begin{split}
&\partial_{t_1}  \Et \left[\theta \Delta {G}_N(v_0)\right]  \Bigg\rvert_{t_1=0} = O(N^{-1}) + \frac{1}{4\pi \iu \cdot  \prod_{j = 1}^k \sqrt{\left(v_0 -   r_j\right)\left(v_0  -  s_j\right)}}  \times \\
& \oint_{\Gamma}  \frac{ \prod_{j = 1}^k \sqrt{\left(w -   r_j \right)\left(w  -  s_j\right)}}{(w - v_0)(v_1 - w)^2}dw +  \partial_{t_1} \frac{NP^{\t,\v}_N(v_0) }{D_N(v_0)}  \Bigg\rvert_{t_1=0}.
\end{split}
\end{equation}

As before we can apply $\Upsilon^N_{v_0}$ to both sides of the above equation. The only difference with respect to the $m \geq 2$ case is the second term on the right side. Notice that it is analytic in the unbounded component of $\Gamma$ and decays like $|v_0|^{-k - 1}$ as $|v_0| \rightarrow \infty$. Consequently, there is no residue at infinity and the integral over $\Gamma$ is zero. Arguing as in the case $m\geq 2$ we get
\begin{equation*}\label{SQ2}
\begin{split}
&\partial_{t_1} \Et \left[\theta \Delta {G}_N(v_0)\right]  \Bigg\rvert_{t_1=0} \hspace{-4mm} = O(N^{-1}) + \\
& \Upsilon^N_{v_0} \left[\frac{1}{4\pi \iu \cdot  \prod_{j = 1}^k \sqrt{\left(v_0 -   r_j\right)\left(v_0  -  s_j\right)}} \oint_{\Gamma}  \frac{ \prod_{j = 1}^k \sqrt{\left(w-   r_j\right)\left(w  -  s_j\right)}}{(w - v_0)(v_1 - w)^2} dw \right].
\end{split}
\end{equation*}
Finally, we can replace $r_j, s_j$ with $\hat{r}_j, \hat{s}_j$ and $\Upsilon_{v_0}^N$ with $\Upsilon_{v_0}$, which produces an error $O(N^{-1}\log (N))$ by Assumption 5 and Property (3) in Proposition \ref{PropUpsilon}.

\section{Central limit theorem: Part II}\label{Section6}
In this section we will prove (\ref{S8}), which is the missing ingredient necessary to complete the proof of Theorem \ref{MainTechT}. In what follows we will continue to use the same notation as in Section \ref{Section5}. Before we go into the main argument we introduce a bit of additional notation and isolate a basic result, which will be used several times throughout.

If $X_1,\dots,X_n$ are bounded random variables, we denote by $M_c( X_1, \dots, X_n)$ their joint cumulant. From Lemma \ref{LemCum1} we know that for any bounded random variable $\xi$ we have
\begin{equation}\label{CumEq1}
\frac{\partial^n \Et [ \xi] }{\partial t_1 \cdots \partial t_n}  \Bigg\rvert_{t_a=0, 1 \leq a \leq n} \hspace{-5mm}= M_c(\xi, NG_N(v_1),\dots,NG_N(v_n)) = M_c(\xi, \Delta G_N(v_1),\dots,\Delta G_N(v_n)),
\end{equation}
where the second equality follows from the fact that cumulants are unchanged under shifts. To ease notation later in the text we set for a subset $A = \{a_1, \dots, a_k\} \subset \{1, \cdots, n\}$ 
$$M_N(\xi | v_a, A) := M_c(\xi, \Delta G_N(v_{a_1}), \dots, \Delta G_N(v_{a_k})).$$

\subsection{Estimating the remainders}\label{Section6.1}
In this section we reduce (\ref{S8}) to the following statement, whose proof is given in Section \ref{Section6.2}.
\begin{prop}\label{moments} Assume the notation from Theorem \ref{MainTechT}. If $l \geq 1$ and $v_1,\dots,v_l \in U$ then
\begin{equation}
\E \left[ \prod_{a = 1}^l  \left| \Delta G_N(v_a) \right| \right] = O(1),
\end{equation}
where the constant in the big $O$ notation is uniform as $v_1,\dots,v_l$ vary in compact subsets of $U$.
\end{prop}

We assume the validity of Proposition \ref{moments} and proceed with the proof of (\ref{S8}). Our goal is to prove that for $m \geq 0$ we have 
\begin{equation}\label{RemKILL}
 \frac{\partial^m}{\partial t_1 \cdots \partial t_m}  \left[ \frac{ N^2 \cdot \R^{\t, \v}_{N}(z)}{W^{\t, \v}_m( z) }\right]  \Bigg\rvert_{t_a=0, 1 \leq a \leq m} = O\left(1\right)
\end{equation}
uniformly as $z$ and $v_1, \dots, v_m$ vary over $\Gamma'$ and compacts in $U^u$ respectively.  This implies (\ref{S8}).

In view of (\ref{XDef2}) we have
\begin{equation}\label{DefY}
\begin{split}
 \frac{ N^2 \cdot \R^{\t, \v}_{N}(z)}{W^{\t, \v}_m( z) } = \hspace{2mm}&  \Et \left[ \xi_N(z) \cdot \Delta G_N(\sigma(z))^2 \right] + \Et\left[ \xi'_N(z) \cdot \partial_z  \Delta G_N(\sigma(z)) \right]  + \\
&  \Et \left[  \mathfrak{c}_N(z; \t,\v)  \right] +\Et \left[  \mathfrak{c}'_N(z; \t,\v) \cdot  \Delta G_N(\sigma(z)) \right].
\end{split}
\end{equation} 
In (\ref{DefY}) $\xi_N(z), \xi'_N(z)$ are random analytic functions in $z$, which do not depend on $\t,\v$ and that are $O(1)$ uniformly over compacts in $\mathcal{O}\cap \mathcal{M}$ and $N$, recall that $\mathcal{O}$ is the inverse image of $U$ under the map $z \rightarrow z + \lu z^{-1}$. In addition, $\mathfrak{c}_N(z; \t,\v),\mathfrak{c}'_N(z; \t,\v)  $ are  linear combinations of random analytic function in $z$, independent  of $\t,\v$,  that are also $O(1)$ uniformly over compacts in $\mathcal{O} \cap \mathcal{M}$. The coefficients of these linear combination are infinitely differentiable functions in $t_i$, whose derivatives, evaluated at $t_1 = \dots = t_m = 0$, are all uniformly bounded as $v_1,\dots,v_m$ vary over compacts in $U$, $z$ varies over compacts in $\mathcal{O} \cap \mathcal{M}$ and $\left| \sigma(z)- v_i\right|$ for $i = 1,\dots, m$ are bounded away from $0$.

We now fix $m \geq 1$ and differentiate both sides of (\ref{DefY}) with respect to $t_1,\dots,t_m$ and set $t_1= \dots = t_m = 0$ (the case $m = 0$ will be treated separately). For the terms involving random variables we use (\ref{CumEq1}) to rewrite the result as a cumulant. Observe that we need to apply Leibniz rule when we differentiate  $\Et \left[\mathfrak{c}_N(z; \t,\v) \right] $  or $\Et \left[\mathfrak{c}'_N(z; \t,\v) \right]$; therefore, we will obtain a sum depending on how many times we differentiated one of the coefficients in $\mathfrak{c}_N(z; \t,\v)$ or $\mathfrak{c}'_N(z; \t,\v)$ and how many times the expectations $\Et $. We obtain the following result
\begin{equation}\label{StartHARD}
\begin{split}
&\frac{\partial^m}{\partial t_1 \cdots \partial t_m}  \left[ \frac{ N^2\cdot \R^{\t, \v}_{N}(z)}{W^{\t, \v}_m( z) } \right]  \Bigg\rvert_{t_a=0, 1 \leq a \leq m} = M_N\left( \xi_N(z)  \Delta G_N(\sigma(z))^2| v_a, \{1, \dots,m\}\right)   + \\
&M_N\left(\xi'_N(z)  \partial_z  \Delta G_N(\sigma(z)) |v_a, \{1, \dots, m\} \right) +  \sum_{A \subset \{1,\dots,m\}}  \big[ M_N\left(  \partial_A \mathfrak{c}_N(z; {\bf 0},\v)|v_a, A^c \right)+ \\
&M_N(  \partial_A\mathfrak{c}'_N(z; {\bf 0},\v) \cdot \Delta G_N(\sigma(z)) |v_a,  A^c) \big]. 
\end{split}
\end{equation}
Using that cumulants are linear combinations of moments and Proposition \ref{moments},  we conclude that each term in (\ref{StartHARD}) is $O\left(1\right)$ uniformly as $v_1,\dots,v_m$ vary over compacts in $U$, $z$ varies over compacts in $\mathcal{O} \cap \mathcal{M}$, $\left| \sigma(z)- v_i\right|$ for $i = 1,\dots, m$ are bounded away from $0$ and $N \rightarrow \infty$. One might be cautious about the term involving $\partial_z  \Delta G_N(\sigma(z))$; however, by Cauchy's Theorem the uniform moment bound we have for $\Delta G_N(\sigma(z))$ implies one for its derivative.

Since $\Gamma' \subset \mathcal{O}\cap \mathcal{M}$ we conclude (\ref{RemKILL}) for the case $m \geq 1$. If $m = 0$, then (\ref{DefY}) reads
\begin{equation}\label{Spec0}
\begin{split}
 N^2\cdot \R_{N}(z) = \hspace{2mm} &\E \left[ \xi_N(z) \cdot \Delta G_N(\sigma(z))^2 \right] + \E\left[ \xi'_N(z)\cdot \partial_z  \Delta G_N(\sigma(z)) \right] + \\
&\E \left[  \mathfrak{c}_N(z)  \right]   + \E \left[  \mathfrak{c}'_N(z)  \Delta G_N(\sigma(z)) \right] .
\end{split}
\end{equation}
Combining that $\xi_N(z), \xi'_N(z),  \mathfrak{c}_N(z)$ and $\mathfrak{c}'_N(z)$ are uniformly bounded over compacts in $\mathcal{O} \cap \mathcal{M}$ with Proposition \ref{moments} we conclude that (\ref{Spec0}) is $O(1)$ as $N \rightarrow \infty$. This proves  (\ref{RemKILL}) for all $m \geq 0$.

\subsection{Self-improving estimates and the proof of Proposition \ref{moments}}\label{Section6.2}
In this section we prove Proposition \ref{moments}. For clarity we split the proof into several steps.\\

{\raggedleft {\bf Step 1.}} In the first step we derive a weak a priori estimate on $\E \left[ \prod_{a = 1}^m  \left| \Delta G_N(v_a) \right| \right]$, which will be iteratively improved in the steps below until we reach the desired estimate of the proposition. More precisely, we show that for each $n \in \mathbb{N}$, compact subset $K \subset U$ and $v_1,\dots,v_n \in K$ we have
\begin{equation}\label{WAE}
\E \left[  \prod_{i = 1}^n \left|\Delta G_N(v_i) \right| \right] =O \left( N^{n/2 + 1/2} \right),
\end{equation}
where the constant in the big $O$ notation depends on $K$ and $n$.

Recall from Section \ref{Section5.1} that $\Delta G_N(v) =N\left(G^d_N(v) - G^c_N(v)\right)$, where
$$ G^d_N(v)=\int\limits_{\mathbb{R}}\frac{\mu_N(dx)}{v-x} \text{ and } G^c_N(v)=\int\limits_{\mathbb{R}}\frac{\hat\mu_N(dx)}{v-x}.$$
Using H{\"o}lder's inequality, we can reduce (\ref{WAE}) to showing that for all $v \in K$ we have
\begin{equation}\label{WAE2}
\E \left[  N^n \left|G^d_N(v) - G^c_N(v)\right|^n \right] =O \left( N^{n/2 + 1/2} \right)
\end{equation}
Fix $\eta > 0$ small enough so that the $\eta$ neighborhood of $S = \cup_{j = 1}^k [ \hat{a}_j, \hat{b}_j]$ is disjoint from $K$. Let $h(x)$ be a smooth function, whose support is inside the $\eta$-neighborhood of $S$, and such that $h(x) = 1$ on an $\eta/2$-neighborhood of $S$. Note that for all $N$ sufficiently large we have that $\mu_N$ and $\hat\mu_N$ are both supported in the $\eta/2$-neighborhood of $S$. Setting $g(x) = (v - x)^{-1}$ we have
$$G^d_N(v) - G^c_N(v) = \int_{\mathbb{R}} g(x)h(x) \mu_N(dx) - \int_{\mathbb{R}} g(x)h(x) \mu(dx).$$ 
 We can apply Corollary \ref{BC1} for the function $g \cdot h$ with $a = r \cdot N^{1/2n - 1/2}$, $r > 0$ and $p = 3$ to get
$$\mathbb{P}_N \left( \left| G^d_N(v) - G^c_N(v) \right| \geq c_1 r N^{-1/2 + 1/2n} + c_2 N^{-3} \right) \leq \exp\left( CN \log(N) - 2\theta \pi^2 r^2 N^{1+1/n}\right),$$
which implies (\ref{WAE2}).\\

{\raggedleft {\bf Step 2.}} In this step we reduce the proof of the proposition to the establishment of the following self-improvement estimate claim.\\

{\raggedleft {\bf Claim:}} Suppose that for some $n, M \in \mathbb{N}$ we have that 
\begin{equation}\label{assume}
\E \left[ \prod_{a = 1}^m \left| \Delta G_N(v_a) \right| \right] = O(1) + O\left(N^{m/2 + 1 - M/2}\right) \mbox{ for } m = 1,\dots, 4n + 4,
\end{equation}
where the constaints in the big $O$ notations are uniform as $v_a$ vary over compact subsets of $U$ for $a = 1,\dots,4n+4$. Then we have
\begin{equation}\label{deduce}
\E \left[ \prod_{a = 1}^m \left| \Delta G_N(v_a) \right| \right] = O(1) + O\left(N^{m/2 + 1 - (M+1)/2}\right) \mbox{ for } m = 1,\dots, 4n.
\end{equation}
The proof of the above claim will be established in the following steps. For now we assume its validity and conclude the proof of the proposition.\\

Notice that (\ref{WAE}) implies that (\ref{assume}) holds for the pair $n = 2l$ and $M = 1$. The conclusion is that (\ref{assume}) holds for the pair $n = 2l -1$ and $M = 2$. Iterating the argument an additional $l$ times we conclude that (\ref{assume}) holds with $n = l-1$ and $M = l+2$, which implies the proposition. \\

{\raggedleft {\bf Step 3.}} In this step we prove that 
\begin{equation}\label{Y8}\begin{split}
&M_N(\Delta G_N(v_0) | v_a, \{1, \dots, m\}) = O(1) + O\left(N^{m/2 + 1 - M/2}\right)\\ &\mbox{ for $m = 1,\dots,4n + 2$} 
 \mbox{ and }  \E \left[\Delta {G}_N(v_0)\right] = O(1) + O\left(N^{ 1 - M/2}\right), \mbox{ where }
\end{split}
\end{equation}
the constants in the big $O$ notation are uniform over $v_0,\dots,v_m$ in compact subsets of $U$. \\

We start by fixing a compact subsets $\mathcal{V} \subset U $, which is invariant under conjugation and let $\Gamma' = \Gamma_1 \cup \Gamma_2$ be as in Section \ref{Section5.3} with $\sigma(\Gamma_1)$ and the bounded components of $U \setminus \sigma(\Gamma_1)$ disjoint from $\mathcal{V}$. For  $m=  1,\dots, 4n + 2$, we differentiate both sides of (\ref{S7}) with respect to $t_1,\dots,t_m$ and set $t_1= \dots = t_m = 0$. 
Combining (\ref{CumEq1}), (\ref{S91}) and (\ref{S92}) the result we obtain is
\begin{equation}\label{Y1}
\begin{split}
&M_N(\Delta G_N(v_0) |v_a, \{1, \dots, m\}) =  \frac{NP^{m}_N(v_0) }{D_N(v_0)}+  \frac{{\bf 1}_{\{m = 1\} }}{4\pi \iu \cdot  \prod_{j = 1}^k \sqrt{\left(v_0 -   r_j\right)\left(v_0  -  s_j\right)}}  \times \\
& \oint_{\sigma(\Gamma_1)}  \frac{ \prod_{j = 1}^k \sqrt{\left(w -   r_j \right)\left(w  -  s_j \right)}}{(w - v_0)(w - v_1)^2}dw +  \frac{-D_N(v_0)^{-1}}{2\pi \iu} \oint_{\Gamma'} \frac{N\cdot \R^{m}_{N}(z)dz} {zH_N(z)(v_0 -\sigma(z))}, \mbox{ where }
\end{split}
\end{equation} 
\begin{equation}\label{RemandP}
\R^{m}_{N}(z) = \frac{\partial^m}{\partial t_1 \cdots \partial t_m}   \frac{ \R^{\t, \v}_{N}(z)}{W^{\t, \v}_m( z) }   \Bigg\rvert_{t_a=0, 1 \leq a \leq m}  P^{m}_N(v_0) = \frac{\partial^m}{\partial t_1 \cdots \partial t_m} P^{\t, \v}_N(v_0)  \Bigg\rvert_{t_a=0, 1 \leq a \leq m}.
\end{equation}

By the same arguments following (\ref{S92}) we may apply the map $\tilde{\Upsilon}^N_{v_0}$ from (\ref{Upsilon}) for the contours $\sigma(\gamma'_1), \dots, \sigma(\gamma'_k)$ and the points $r_j, s_j$ for $j = 1, \dots, k$ to both sides of (\ref{Y1}) to get
\begin{equation}\label{Y111}
\begin{split}
&M_N(\Delta G_N(v_0) |v_a, \{1, \dots, m\}) =  \tilde{\Upsilon}^N_{v_0} \left[  \frac{-D_N(v_0)^{-1}}{2\pi \iu} \oint_{\Gamma'} \frac{N\cdot \R^{m}_{N}(z)dz} {zH_N(z)(v_0 -\sigma_N(z))}\right] +\\
&\tilde{\Upsilon}^N_{v_0} \left[  \frac{{\bf 1}_{\{m = 1\} }}{4\pi \iu \cdot  \prod_{j = 1}^k \sqrt{\left(v_0 -   r_j\right)\left(v_0  -  s_j\right)}}   \oint_{\sigma(\Gamma_1)}  \frac{ \prod_{j = 1}^k \sqrt{\left(w -   r_j \right)\left(w  -  s_j \right)}}{(w - v_0)(v_1 - w)^2}dw \right].
\end{split}
\end{equation}

Combining (\ref{Y111}) and an application of $\tilde{\Upsilon}^N_{v_0}$ to both sides of (\ref{S7}) we get
\begin{equation}\label{Y2}
\begin{split}
&M_N(\Delta G_N(v_0) |v_a, \{1, \dots, m\}) = O(1) +  \tilde{\Upsilon}^N_{v_0} \left[\frac{-D_N(v_0)^{-1}}{2\pi \iu} \oint_{\Gamma'} \frac{ N\cdot \R^{m}_{N}(z)  dz}{zH_N(z)(v_0 -\sigma(z))} \right], \\
&\\
&\E \left[\Delta {G}_N(v_0)\right] = O(1) +  \tilde{\Upsilon}^N_{v_0} \left[  \frac{-D_N(v_0)^{-1}}{2\pi \iu} \oint_{\Gamma'} \frac{N \cdot \R_{N}(z)dz}{zH_N(z)(v_0 -\sigma(z))} \right].
\end{split}
\end{equation}
 The constants in the big $O$ notation are uniform over $v_0, v_1,\dots,v_m$ in compact subsets of $\mathcal{V}$. \\

At this time we recall (\ref{StartHARD}), which states that for $m = 1,\dots,4n + 2$ we have
\begin{equation*}
\begin{split}
&N^2 \cdot  \R^{m}_{N}(z) =  \sum_{A \subset \{1,\dots,m\}}  \big[ M_N\left(  \partial_A \mathfrak{c}_N(z; {\bf 0},\v)|v_a, A^c \right)+ M_N(  \partial_A\mathfrak{c}'_N(z; {\bf 0},\v) \cdot \Delta G_N(\sigma(z)) |v_a,  A^c) \big] + \\
&M_N\left( \xi_N(z)  \Delta G_N(\sigma(z))^2| v_a, \{1, \dots,m\}\right)   +  M_N\left(\xi'_N(z)  \partial_z  \Delta G_N(\sigma(z)) |v_a, \{1, \dots,m\} \right). 
\end{split}
\end{equation*}
Recall that $\xi_N(z)$, $\xi_N'(z)$, $ \partial_A\mathfrak{c}_N(z; {\bf 0},\v)$ and $ \partial_A\mathfrak{c}'_N(z; {\bf 0},\v)$ are all $O(1)$ if $v_1, \dots, v_m \in \mathcal{V}$ and $z \in \Gamma'$. The latter and (\ref{assume}) imply
\begin{equation}\label{Y6}
\R^{m}_{N}(z) = O\left(N^{-2}\right) +  O\left(N^{m/2  - M/2}\right), \text{ for } m = 1,\dots,4n + 2.
\end{equation}
By combining (\ref{Spec0}) and (\ref{assume}) we get that (\ref{Y6}) holds for $m = 0$ as well. Finally, (\ref{Y2}), (\ref{Y6})  and Property (1) in Proposition \ref{PropUpsilon} together imply (\ref{Y8}).\\

{\raggedleft {\bf Step 4.}} In this step we will establish the validity of (\ref{deduce}) except for a single case, which will be handled separately in the next step.

Notice that by H{\" o}lders inequality we have
$$ \sup_{v_1,\dots,v_{m} \in \mathcal{V} }  \E \left[ \prod_{a= 1}^{m} \left| \Delta G_N(v_a) \right|\right] \leq \sup_{v \in \mathcal{V}} \E \left[\left|\Delta G_N(v)  \right|^m \right],$$
and so to finish the proof it suffices to show that for $m = 1,\dots,4n$ we have
\begin{equation}\label{DIM}
\E \left[\left|\Delta G_N(v)  \right|^m \right] = O(1) + O \left( N^{m/2 + 1/2 - M/2}\right).
\end{equation}

Since centered moments are linear combinations of products of joint cumulants, we deduce from the first line in (\ref{Y8}) that for $m = 1,\dots,4n + 2$ we have
\begin{equation}\label{Y9}
\sup_{v_0,\dots,v_{m-1} \in \mathcal{V}} \E \left[ \prod_{a= 0}^{m-1} \left(\Delta G_N(v_a) - \E [ \Delta G_N(v_a) ]\right) \right] = O(1) + O\left(N^{(m-1)/2 + 1 - M/2}\right).
\end{equation}
 Combining the latter with the first and second lines of (\ref{Y8}) we see that
\begin{equation}\label{Y10}
\sup_{v_0,\dots,v_{m-1} \in \mathcal{V}} \E \left[ \prod_{a= 0}^{m-1} \Delta G_N(v_a) \right] = O(1) + O\left(N^{(m-1)/2 + 1 - M/2}\right).
\end{equation}
If $m = 2m_1$ then we can set $v_0 = \cdots v_{m_1 - 1} = v$ and $v_{m_1} = \cdots = v_{2m_1- 1} = \overline{v}$ in (\ref{Y9}), which yields
\begin{equation}\label{Y11}
\sup_{v \in \mathcal{V}} \E \left[  \left| \Delta G_N(v) \right|^{m} \right] = O(1) + O\left(N^{m/2 + 1/2 - M/2}\right) \mbox{ for $m = 2, 4,6,\dots, 4n + 2 $ .} 
\end{equation}

We next let $m = 2m_1 + 1$ be odd and notice that by the Cauchy-Schwarz inequality and (\ref{Y11}) 
\begin{equation}\label{CSI}
\begin{split}
&\sup_{v \in \mathcal{V}} \E \left[   \left| \Delta G_N(v) \right|^{2m_1 + 1} \right]  \leq \sup_{v \in \mathcal{V}} \left[   \E \left[  \left| \Delta G_N(v) \right|^{2m_1 + 2} \right]\right]^{1/2} \cdot  \left[   \E \left[  \left| \Delta G_N(v) \right|^{2m_1 } \right]\right]^{1/2} = \\
& O(1) + O\left(N^{m_1 + 1 - M/2}\right) + O\left( N^{m_1/2 + 3/4 - M/4}\right)  .
\end{split}
\end{equation}
We note that the bottom line of (\ref{CSI}) is $O(1) + O \left( N^{m_1 + 1 - M/2}\right)$ except when $M = 2m_1 + 2$, since
$$m_1/2 + 3/4 - M/4 \leq \begin{cases} m_1 + 1  - M/2 &\mbox{ when $M \leq 2m_1 + 1$,}  \\ 0 &\mbox{ when $M \geq 2m_1 + 3$.}\end{cases}$$ 
Consequently, (\ref{Y11}) and  (\ref{CSI}) together imply (\ref{DIM}) except when $M = 2m_1 + 2$ and $m = 2m_1 +1$. We will handle this case in the next step.\\

{\raggedleft {\bf Step 5.}} In this last step we will show that (\ref{DIM}) holds even when $M = 2m_1 + 2$ and $4n > m = 2m_1 +1$. In the previous step we showed in (\ref{Y11}) that $\sup_{v \in \mathcal{V}} \E \left[  \left| \Delta G_N(v) \right|^{2m_1 + 2} \right] = O\left(N^{1/2}\right)$, and below we will improve this estimate to 
\begin{equation}\label{FC1}
\sup_{v \in \mathcal{V}} \E \left[  \left| \Delta G_N(v) \right|^{2m_1 + 2} \right] = O(1).
\end{equation}
The trivial inequality $x^{2m_1+2} + 1 \geq |x|^{2m_1+1}$ together with (\ref{FC1}) implies 
$$ \sup_{v \in \mathcal{V}} \E \left[  \left| \Delta G_N(v) \right|^{2m_1 + 1} \right] = O(1).$$
 Consequently, we have reduced the proof of the claim to establishing (\ref{FC1}).\\

Let us list the relevant estimates we will need 
\begin{equation}\label{FC2}
\begin{split}
& \E \left[ \prod_{a = 1}^{2m_1 + 4} \left| \Delta G_N(v_a) \right| \right] =  O\left(N^{3/2}\right),\hspace{2mm} \E \left[ \prod_{a = 1}^{2m_1 + 2} \left| \Delta G_N(v_a) \right| \right] = O\left(N^{1/2}\right),\hspace{2mm}\\
&\E \left[ \prod_{a = 1}^{j} \left| \Delta G_N(v_a) \right| \right] =  O(1) \mbox{ for $0 \leq j \leq 2m_1$,} \hspace{2mm} \E \left[ \prod_{a = 1}^{2m_1 + 3} \left| \Delta G_N(v_a) \right| \right] = O\left(N\right).
\end{split}
\end{equation}
The above identities follow from (\ref{Y11}) and (\ref{CSI}). All constants are uniform over $v_a \in \mathcal{V}$. 
Below we feed the improved estimates of (\ref{FC2}) into Steps 3. and 4., which would ultimately yield (\ref{FC1}). \\

In Step 3. we note that we have the following version of (\ref{Y6}) 
\begin{equation}\label{Y12}
\R_N^{m}(z) = O\left(N^{-1}\right) \mbox{ whenever  $0 \leq m \leq 2m_1 + 1$.}
\end{equation}
The latter statement follows from  (\ref{StartHARD}) for $2m_1 + 1  \geq m \geq 1$ and (\ref{Spec0}) for $m = 0$, combined with
$$\E \left[ \prod_{a = 1}^{m+2}  \left| \Delta G_N(v_a) \right| \right] = O(N) \mbox{ for $m = 0, \dots, 2m_1 + 1$} ,$$
 the latter being a consequence of (\ref{FC2}). Using (\ref{Y12}) instead of (\ref{Y6}) in Step 3. we obtain the following improvement over (\ref{Y8})
\begin{equation}\label{Y13}\begin{split}
M_c(\Delta G_N(v), \dots, \Delta G_N(v_m)) = O(1), \text{ for } m = 1,\dots, 2m_1 + 1
 \text{ and }  \E \left[\Delta {G}_N(v)\right] = O(1) .
\end{split}
\end{equation}

We next repeat the arguments in Step 4. and note that by using (\ref{Y13}) in place of (\ref{Y8}) we obtain the following improvement over (\ref{Y10})
\begin{equation}\label{Y14}
\sup_{v_0,\dots,v_{2m_1 +1} \in \mathcal{V}} \E \left[ \prod_{a= 0}^{2m_1 + 1} \Delta G_N(v_a) \right] = O(1).
\end{equation}
Setting $v_0 = \cdots v_{m_1 } = v$ and $v_{m_1 + 1} = \cdots = v_{2m_1 + 1} = \overline{v}$ in (\ref{Y14}) we get (\ref{FC1}).

\section{$q$-Racah tiling models and ensembles}\label{Section7}
As discussed in Section \ref{Section1} our main motivation for studying discrete log-gases on shifted quadratic lattices comes from the $q$-Racah tiling model that was introduced in \cite{BGR}. In Section \ref{Section7.1} we give a formal definition of the model and in Section \ref{Section7.2} we state the main results we prove about it in Theorems \ref{TLLN} and \ref{TCLT}. In Section \ref{Section7.3} we explain how the model is related to a certain random particle system that we call the $q$-Racah ensemble and state a law of large numbers and central limit theorem for the latter as Theorems \ref{BT1} and \ref{QRCLT} in Section \ref{Section7.4}.

\subsection{The $q$-Racah tiling model}\label{Section7.1}

\subsubsection{Lozenge tilings}\label{Section7.1.1}
\begin{wrapfigure}[9]{r}{0.3\textwidth}
\vspace{-12mm}
  \begin{center}
    \includegraphics[width=0.28\textwidth]{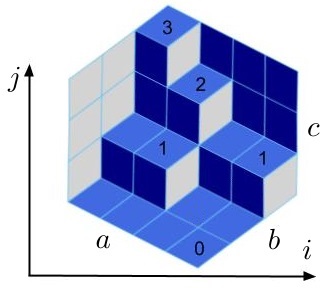}
  \end{center}
  \caption{Tiling of a $3\times 3\times 3$ $\text{ }$ $\text{ }$$\text{ }$$\text{ }$$\text{ }$$\text{ }$$\text{ }$$\text{ }$$\text{ }$$\text{ }$$\text{ }$$\text{ }$$\text{ }$$\text{ }$$\text{ }$$\text{ }$$\text{ }$$\text{ }$hexagon}
 \label{S7F1}
\end{wrapfigure}
Denote by $\Omega_{a \times b \times c}$ the set of all tilings of a hexagon  with side lengths $a, b, c$ by rhombi (or alternatively boxed plane partitions), see Figure \ref{S7F1}. Denote the horizontal rhombi by $\diam$ and  introduce coordinate axes $(i, j)$. Given two parameters $q$ and $\kappa$ we define the probability of an element $\mathcal{T} \in \Omega_{a \times b \times c}$ through
\begin{equation}\label{DefEq1}
\begin{split}
 &\hspace{-40mm}\mathbb{P}(\mathcal T )=\frac{w(\mathcal T)} {\sum\limits_{ \mathcal S \in \Omega_{a \times b \times c} } w(\mathcal
  S)} \mbox{, where } w(\mathcal T)=\prod\limits_{\diam \in \mathcal
  T} w(\diam), \mbox{ and }\\
& \hspace{-40mm} w(\diam)=  \kappa^2 q^{ j -(c+1)/2 }- q^{-j + (c+1)/2}.
\end{split}
\end{equation}
In the above formula the product is over all horizontal lozenges $\diam$ that belong to $\mathcal{T}$ and $j$ denotes the $j$-th coordinate of the topmost point of $\diam$. We call the probability measure in (\ref{DefEq1}) the $q$-{\em Racah tiling model}.

It was shown in \cite{BGR} that the {\em partition function} (the sum of all weights $w(\mathcal{T})$ or the normalization term in (\ref{DefEq1})) has a nice {\em product form}, which generalizes the famous MacMahon formula for the number of boxed plane partitions \cite{St}. Note that the number of horizontal rhombi in all tilings of a given hexagon is the same, hence $\mathbb{P}$ is invariant under multiplication of $w(\diam)$ by a constant.

In order for (\ref{DefEq1}) to define an honest probability measure, one requires that the weights $w(\mathcal{T})$ be non-negative. This imposes certain restrictions on the parameters $q, \kappa$ and there are three possible cases that lead to positive weights:
\begin{enumerate}[label=(\roman*)]
\item {\em imaginary $q$-Racah case:} $q$ is a positive real number and $\kappa$ is a purely imaginary number;
\item {\em real $q$-Racah case:} $q$ is a positive real number and $\kappa$ is a real number that cannot lie inside the interval $\left[q^{-a + 1/2}, q^{(b+c-1)/2} \right]$ if $q > 1$ or the interval $\left[q^{(b+c-1)/2}, q^{(a-1)/2} \right]$ if $q < 1$;
\item {\em trigonometric $q$-Racah case:} $q$ and $\kappa$ are complex numbers on the unit circle, i.e. $q = e^{\i \alpha}$, $\kappa = e^{\i \beta}$, where $\alpha,\beta$ must be such that $-\alpha(b + c - 1)/2 + \beta$ and $\alpha(a - 1/2) + \beta$ must lie in the same interval of the form $[\pi k, \pi (k+1)]$, $k \in \mathbb{Z}$.
\end{enumerate}
The names of the above cases are related to those of the classical orthogonal polynomials that appear in the analysis. In this paper, we will only consider the real $q$-Racah case with $q \in (0,1)$ and $\kappa \in \left[0, q^{(b+c - 1)/2 } \right)$ although most of our arguments can be extended to other cases.

If we let $\kappa \rightarrow 0$ then we get the $q$-{\em Hahn case} $w(\diam) = q^{-j}.$
In this case the probability of a plane partition is proportional to $q^{-V}$, where $V$ denotes the {\em volume} of the plane partition, i.e. the number of cubes that it contains. If we send $\kappa \rightarrow \infty$ we get that the probability of a partition is proportional to $q^{V}$. In this sense, one can view our model as an interpolation between the models $q^{V}$ and $q^{-V}$. Finally, if one sends $\kappa \rightarrow 0$ and $q \rightarrow 1$, one recovers the uniform measure on boxed plane partitions.

\subsubsection{Particle configurations}\label{Section7.1.2}
In what follows we describe an alternative formulation of our model that is more suitable for stating our results. We perform a simple affine transformation of the hexagon and lozenges, detailed in Figure \ref{S1F3}.
 \begin{figure}[h]
\includegraphics[width=0.75\linewidth]{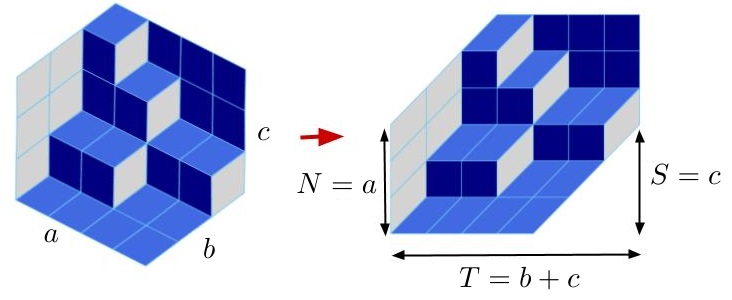}
\caption{Affine transformation of lozenges}
  \label{S1F3}
\end{figure}  

Let us introduce new parameters $N,T,S$ that are related to $a,b,c$ through $N = a$, $T = b+c$ and $S = c$. Each tiling in $\Omega_{a \times b \times c}$ naturally corresponds to a family of $N$ non-intersecting up-right paths as shown in Figure \ref{S1F4}. For each $0 \leq t \leq T$ we draw a vertical line through the point $(t, 0)$ and denote by $x^t_1 < x^t_2< \cdots < x^t_N$ the intersection of the line with the $N$ up-right paths. We interpret the intersection points as particles and will typically use the same letter to refer to a particle and its location. In this way, we can view a tiling as an $N$-point (or particle) configuration, which varies in time $t = 0,\dots, T$. Observe that when $t = 0$ the configuration consists of the points $\{0,1, \dots, N-1\}$ and when $t = T$ the configuration consists of the points $\{S, S + 1, \dots, S + N - 1\}$.

\begin{figure}[h]
\includegraphics[width=0.40\linewidth]{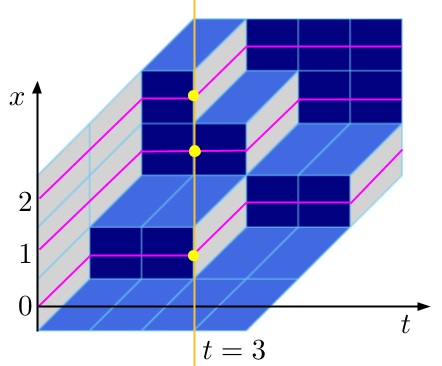}
\caption{Modified hexagon and up-right path configuration (in purple). The yellow dots are the particles at time $t = 3$ and we have $x^t_1 = 1, x^t_2 = 3$ and $x^t_3 = 4$}
  \label{S1F4}
\end{figure}

Given a random configuration $\{x^t_k\}$ we define the random {\em height function}
\begin{equation}\label{htFT}
h:  \mathbb{Z}_{\geq 0} \times \left(\mathbb{Z} + \frac{1}{2} \right)  \rightarrow \mathbb{Z}_{\geq 0} \mbox{ as } h(t,s) =  | \{ k \in \{1, \dots, N\} : x_t^k < s\}|.
\end{equation}
In terms of the tiling in Figure \ref{S1F4} the height function is defined at the vertices of rhombi, and it counts the number of particles below a given vertex. The latter definition is in agreement with the standard three-dimensional interpretation of the tiling as a stack of boxes \cite{K2}.

\subsection{Main results for the $q$-Racah tiling model}\label{Section7.2}
Our results are about the global fluctuations of a random lozenge tiling with distribution (\ref{DefEq1}) when the parameters $q,\kappa$ and the sizes of the hexagon $N,T,S$ scale in a particular fashion that we detail below.

\begin{defi}\label{ParScale}
We assume that we are given real numbers ${\tt N}, {\tt T}, {\tt S}, {\tt q} $ and ${\tt k} $ such that 
$${\tt N}, {\tt T}, {\tt S}, {\tt q} > 0, \hspace{5mm} \lk \geq 0, \hspace{5mm} {\tt q } < 1, \hspace{5mm}  {\tt N} < {\tt T}, \hspace{5mm} {\tt S} < {\tt T},\hspace{5mm}   {\tt k}^2{\tt q}^{ - {\tt T}} < 1.$$
Given such a choice of parameters and $\varepsilon \in (0,1)$ we let $\mathbb{P}_\varepsilon$ be the probability measure in (\ref{DefEq1}) with 
$$q = {\tt q }^{\varepsilon} + O(\varepsilon^2),  \hspace{5mm} N = {\tt N} \varepsilon^{-1} + O (1) , \hspace{5mm} T = {\tt T} \varepsilon^{-1} + O(1), \hspace{5mm} S = {\tt S} \varepsilon^{-1} + O(1),\hspace{5mm}   \kappa = {\tt k} + O(\varepsilon).$$
\end{defi}

\subsubsection{Limit shape}
Our first result concerns the hydrodynamic limit of the height function $h$, with distribution $\mathbb{P}_{\varepsilon}$, under the parameter scaling in Definition \ref{ParScale} when $\varepsilon$ converges to zero. On a macroscopic scale the random height function concentrates around a deterministic {\em limit shape}, i.e.
\begin{equation}\label{heightConv}
\varepsilon \cdot h\left( \lfloor x \varepsilon^{-1} \rfloor , \lfloor y \varepsilon^{-1} \rfloor +  1/2 \right) \rightarrow \hat{h}(x,y) \mbox{ as $\varepsilon \rightarrow 0^+$},
\end{equation}
where $(x,y)$ are the new global continuous coordinates, $\hat{h}(x,y)$ is the function whose graph is the limit shape and the convergence is in probability. The new coordinates $(x,y)$ are assumed to belong to the limiting hexagon $\mathcal{P}$, which is parametrized by $\lN, \lS, \lT$ the same way that our discrete hexagon was parametrized by $N,S,T$; see the central part of Figure \ref{S1F5}. The limit shape can then be understood as a continuous function on $\mathcal{P}$ and we describe it next.

\begin{figure}[h]

 \includegraphics[width=1\linewidth]{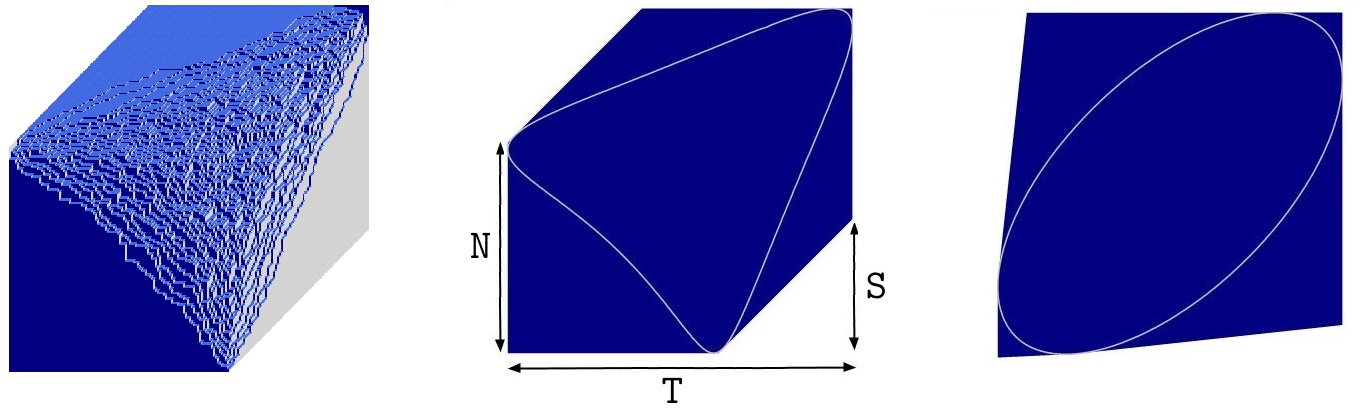}
\caption{The left part shows a simulation of a tiling. The middle part shows the hexagon $\mathcal{P}$ and the liquid region $\mathcal{D}$ is the region inside the gray curve. The right part denotes the image of $\mathcal{P}$ and $\mathcal{D}$ under the map $(x,y) \rightarrow  (\lq^{-x}, \lq^{-y} + \lk^2 \lq^{-\lS - x + y})$ }
  \label{S1F5}
\end{figure}

With parameters as in Definition \ref{ParScale} we define $\phi$ as
$$\phi(x,y) = \mbox{arccos} \left( \frac{\left(\lq^{-\lN} - 1\right)\left(1 - \lq^{-\lN-\lT}\right)\left( \lq^{-2x}- \lk^2 \lq^{-y - \lS } \right)^2 + A +B}{2 \sqrt{A B}} \right), \mbox{ where }$$ 
$$A = \left( \lq^{-x} - \lq^{-\lS - \lN}\right) \left( \lq^{-x} - \lk^2\lq^{-\lT }\right) \left( \lq^{-x}- \lq^{-y - \lN }\right) \left( \lq^{-x} - \lk^2\lq^{-y - \lS }\right) ,$$
$$B = \lq^{-2\lN - \lT} \left( \lq^{-x} - 1\right) \left( \lq^{-x} - \lk^2\lq^{- t + \lN }\right) \left( \lq^{-x} - \lq^{-y - \lS + \lT}\right) \left( \lq^{-x} - \lk^2\lq^{-\lS + \lN}\right).$$
If the expression inside the arccosine is greater than $1$, then we set $\phi = 0$ and if it less than $-1$, then we set $\phi = \pi$. In terms of the above function $\phi$ we define the limit shape $\hat{h}$ as
\begin{equation}\label{limitshapeH}
\hat{h}(x,y) = \frac{1}{\pi} \int_0^y \phi(x, u) du, \hspace{2mm} \mbox{ for }(x,y) \in \mathcal{P}.
\end{equation}
With the above notation we can state our limit shape result.
\begin{thm}\label{TLLN} Suppose that  ${\tt N}, {\tt T}, {\tt S}, {\tt q}$, ${\tt k}$ and $\mathbb{P}_\varepsilon$ are as in Definition \ref{ParScale} and that $h$ is distributed according to $\mathbb{P}_\varepsilon$. Then for any $(x,y) \in \mathcal{P}$ and $\eta > 0$ we have
\begin{equation}\label{LLNH}
\lim_{\varepsilon \rightarrow 0^+} \mathbb{P}_\varepsilon\left( \left| |\varepsilon \cdot  h\left( \lfloor x \varepsilon^{-1} \rfloor , \lfloor y \varepsilon^{-1} \rfloor +  1/2 \right) -  \hat{h}(x,y) \right| > \eta \right)  = 0.
\end{equation}
\end{thm}
\begin{rem}
The formula for $\phi(x,y)$ was derived in Theorem 8.1 in \cite{BGR}. We remark that while an explicit formula for the limit shape was obtained in \cite{BGR}, it was not proved that the height function actually converges to it.  Theorem \ref{TLLN} constitutes a proof of this fact. 
\end{rem}

An important feature of our model is that the limit shape develops {\em frozen facets} where the function $\hat{h}(x,y)$ is linear. In terms of the tiling a frozen facet corresponds to a region where asymptotically only one type of lozenge is present. In addition, there is a connected open {\em liquid region} $\mathcal{D} \subset \mathcal{P}$, which interpolates the facets. Explicitly, the liquid region $\mathcal{D}$ is given by the set of points $(x,y) \in \mathcal{P}$ where the expression inside the arccosine in the definition of $\phi$ is in $(-1,1)$, i.e.
$$\mathcal{D} = \left\{ (x,y) \in \mathcal{P}:  \left[\left(\lq^{-\lN} - 1\right)\left(1 - \lq^{-\lN-\lT}\right)\left( \lq^{-2x}- \lk^2 \lq^{-y - \lS } \right)^2 + A +B\right]^2 <  4 AB \right\}.$$

If $(x,y) \in \mathcal{D}$ then the limiting height function $\hat{h}$ is curved near $(x,y)$: asymptotically inside the liquid region one observes all three types of lozenges, see e.g. \cite{CLP, KO, K2} for further discussion regarding frozen and liquid regions in related contexts. In addition, the local distribution of the tiling near  $(x,y) \in \mathcal{D}$ is described asymptotically by a certain ergodic translation-invariant Gibbs measure on lozenge tilings of the whole plane. Such a measure is unique up to fixed proportions of lozenges of all three types \cite{SheffGibbs}, and these proportions depend on the slope of $\hat{h}$ at the point $(x,y)$. We refer the reader to \cite{BGR} for a more detailed discussion of this fact for the model we consider, and also to \cite{SheffGibbs, K2, OR, KOS } for analogous results in general dimer models.

\subsubsection{Central limit theorem} Before stating our central limit theorem for the measures $\mathbb{P}_\varepsilon$ we introduce a transformation of our particle configuration from Section \ref{Section7.1}. This transformation is (in some sense) the natural way to view the particle system, and it allows us to identify its global asymptotic fluctuations with a $1D$ section of the two-dimensional Gaussian free field (GFF for short).

Given a point configuration $\{ (t, x_k^t)\}$ we define a new point configuration $\{(U,V)\}$ through
\begin{equation}\label{transform}
U(t,k) = q^{-t} \mbox{ and } V(t,k) = q^{-x^t_k} + \kappa^2 q^{x^t_k - S - t}  \mbox{ for $0 \leq t \leq T$ and $1 \leq k \leq N$.}
\end{equation}
Similarly to before, we define a random height function for the new particle system
\begin{equation}\label{htFP}
\mathcal{H}:  \left\{ q^0, q^{-1}, \dots, q^{-T} \right\} \times \mathbb{R} \rightarrow  \mathbb{Z}_{\geq 0 } \mbox{ as } \mathcal{H}(q^{-t}, v) =  | \{ k \in \{1, \dots, N\} : V(t, k) < v\}|.
\end{equation}
One can formulate an equivalent statement to Theorem \ref{TLLN} for the height function $\mathcal{H}$. I.e. there will be asymptotically a deterministic limiting height function $\hat{\mathcal{H}}$, near which $\mathcal{H}$ concentrates with high probability. Moreover, if we set $\sigma_{\lq}(x,y) = (\lq^{-x}, \lq^{-y} + \lk^2 \lq^{-\lS - x + y})$ then we have the explicit relationship $\hat{h}(x,y) = \hat{\mathcal{H}}\left( \sigma_{\lq}(x,y) \right)$. 

The function $\sigma_{\lq}$ maps the liquid region $\mathcal{D}$ bijectively to a new region $\mathcal{D}'$, parametrized through
$$\mathcal{D}' =  \left\{ (u,v) \in \mathbb{R}^2: \tilde{Q}(u,v) < 0 \right\},$$
where $\tilde{Q}(u,v) = \tilde{A}  u^2 + \tilde{B} v^2 + \tilde{C} uv + \tilde{D}u + \tilde{E}v +\tilde{F}$, and $ \tilde{A}, \tilde{B}, \tilde{C}, \tilde{D}, \tilde{E},\tilde{F}$ are explicitly computable constants that depend on $\lq, \lS, \lT, \lN$ and $\lk$. Consequently, $\mathcal{D}'$ is an ellipse, see the right of Figure \ref{S1F5}.

Our next goal is to define a {\em complex structure} on the limit shape surface --- this is a bijective diffeomorphism $\Omega : \mathcal{D}' \rightarrow \mathbb{H}$. The significance of this map is that the fluctuations of $\mathcal{H}$ will be asymptotically described by the pullback of the GFF on $\mathbb{H}$ under the map $\Omega$. The function $\Omega(u,v)$ is algebraic and it satisfies the following quadratic equation
\begin{equation}\label{ComplexStruct}
\begin{split}
a_2 \Omega^2  + a_1 \Omega + a_0 = 0, \mbox{ where }
\end{split}
\end{equation}
$a_2, a_1, a_0$ are explicit {\em linear }functions of $u$ and $v$ and are such that $a_1^2 - 4a_2a_0 = \lq^{2\lN}\cdot \tilde{Q}(u,v)$ (see Section \ref{Section8.2} for the details). Whenever $(u,v) \in \mathcal{D}'$ the polynomial (\ref{ComplexStruct}) has two complex conjugate roots and we define $\Omega(u,v)$ to be the the one that lies in $\mathbb{H}$. 

We are now ready to state our main theorem for the $q$-Racah tiling model, giving the asymptotics of the global $1D$ fluctuations of $\mathbb{P}_\varepsilon$ in terms of the two-dimensional Gaussian free field. In Section \ref{Section8.1} we recall the definition and basic properties of the GFF.
\begin{thm}\label{TCLT}Suppose that  ${\tt N}, {\tt T}, {\tt S}, {\tt q}$, ${\tt k}$ and $\mathbb{P}_\varepsilon$ are as in Definition \ref{ParScale} and that $\mathcal{H}$ is as in (\ref{htFP}) for the distribution $\mathbb{P}_\varepsilon$. Fix $u \in (1,\lq^{-\lT})$ and let $t(\varepsilon)$ be a sequence of integers such that $q^{-t(\varepsilon)} = u + O(\varepsilon)$. Then the centered random height function
$$\sqrt{\pi} \left( \mathcal{H}(q^{-t}, v) - \mathbb{E}_{\mathbb{P}_\varepsilon}\left[ \mathcal{H}(q^{-t}, v) \right] \right) $$
converges to the $1$d section of the pullback of the Gaussian free field with Dirichlet boundary conditions on the upper half-plane $\mathbb{H}$ with respect to the map $\Omega$ in the following sense: For any set of polynomials $f_i \in \mathbb{R}[x]$ for $i = 1,\dots,m$ the joint distribution of
\begin{equation}\label{TCLTeq1}
\int_{\mathbb{R}} \sqrt{\pi} \left(\mathcal{H}(q^{-t}, v) - \mathbb{E}_{\mathbb{P}_\varepsilon} \left[ \mathcal{H}(q^{-t}, v) \right] \right) f_i(v)dv, \hspace{5mm} i = 1,\dots,m,
\end{equation}
converges to the joint distribution of the similar averages 
$$\int_{a(u)}^{b(u)} \mathcal{F}(\Omega(u, y))f_i(y)dy,\hspace{5mm} i = 1,\dots,m$$
of the pullback of the GFF. In the above formula $a(u), b(u)$ are the $v$-coordinates of the two points where the vertical line through $u$ intersects the ellipse $\tilde{Q}(u,v) = 0$.

Equivalently, the variables in (\ref{TCLTeq1}) converge jointly to a Gaussian vector $(X_1, \dots, X_m)$ with mean zero and covariance
\begin{equation}
\begin{split}
&\mathbb{E}[X_i X_j] =  \int_{a(u)}^{b(u)} \int_{a(u)}^{b(u)} f_i(x) f_j(y) \left(  - \frac{1}{2\pi} \log \left| \frac{\Omega(u,x) - \Omega(u,y)}{\Omega(u,x) - \overline{\Omega}(u,y)}\right| \right) dx dy.
\end{split}
\end{equation}
\end{thm}

As discussed in Section \ref{Section1} the GFF is assumed to be a universal scaling limit in tiling models, which motivates its appearance in our setup. Another reason one might expect to see the GFF in our tiling model comes from its connection to the $\beta$-log gas with $\beta = 2$. We will elaborate this idea later in Section \ref{Section9}, but essentially there is a natural way to view the particle configuration on a fixed vertical slice as a discrete log-gas on a quadratic gases. Log-gases appear naturally in random matrix theory, and in that context there are several models that are known to converge to the GFF \cite{BorW, BorW2, BGJ}. 

We end the section by remarking that Theorem \ref{TCLT} admits a natural two-dimensional generalization, which we formulate as Conjecture \ref{GFFConj} in Section \ref{Section8.4}. At this time our methods only provide access to the global fluctuations at fixed vertical sections of the model, and so we cannot establish the full $2D$ result. Nevertheless, we provide some numerical simulations that give evidence for the validity of the conjecture.

\subsection{The $q$-Racah ensemble}\label{Section7.3} In this section we define the $q$-{\em Racah ensemble}. 
\begin{defi}\label{qrw}
Let $q\in(0,1),$ $M\in \mathbb{Z}_{\geq 0},$ $\a,$ $\b,$ $\delta\in \mathbb R$ and $\gamma=q^{-M-1}.$ For $x \in \{0, 1, \dots, M\}$ 
we introduce the following weight function
\begin{equation} \label{eq:w_qR}
w^{qR}(x)=\frac{(\a q,\b\delta q,\gamma q,\gamma \delta
 q;q)_x}{(q,\a^{-1}\gamma \delta q, \b^{-1}\gamma q, \delta q;q )_x}\frac{(1-\gamma\delta q^{2x+1})}{(\a\b q)^{x}(1-\gamma\delta q)},
\end{equation}
 where 
$(y_1,\dots, y_i;q)_k=(y_1;q)_k\cdots (y_i;q)_k$, and $(y;q)_k=(1-y)(1-yq)\cdots (1-yq^{k-1})$ is the $q$-Pochhammer symbol.
\end{defi}
 \begin{rem} The weight $w^{qR}$ is the weight function of the $q$-Racah orthogonal polynomials, see e.g. \cite[Section 3.2]{KS98}. One can more generally have $\a=q^{-M-1}$ or $\b\delta=q^{-M-1}$ instead of $\gamma=q^{-M-1}$. Our choice is dictated by the fact that under the substitutions $\gamma=q^{-M-1}$ and $\delta=0$ the $q$-Racah weight reduces to the $q$-Hahn weight. 
\end{rem}   

With the above notation we can define the $q$-Racah ensemble as follows.
\begin{defi} \label{DefQR} Fix $N \in \mathbb{N}$ and let $\alpha, \beta, \gamma, \delta, q$ and $M$ be as in Definition \ref{qrw} with $M \geq N-1$. Denote by $\mathfrak{X}$ the collection of $N$-tuples of integers
$$\mathfrak{X} = \{ (\lambda_1, \dots, \lambda_N) \in \mathbb{Z}^N : 0\leq \lambda_1<\lambda_2<\dots<\lambda_N\leq M \}.$$
The $q$-Racah ensemble is a probability measure $\mathbb{P}^{qR}$ on the set $\mathfrak{X}$, given by
\begin{equation} \label{eq:distr}
\mathbb{P}^{qR}(\lambda_1,\dots ,\lambda_N)= \frac{1}{Z(N,M, \alpha, \b, \gamma, \delta,q)} \prod_{1\leq i<j \leq N} \left(\sigma(q^{-\lambda_i})-\sigma(q^{-\lambda_j}) \right)^2 \cdot \prod_{i=1}^{N}{w^{qR}(\lambda_i)},
\end{equation}
where  $\sigma(z) = z + \gamma\delta q z^{-1}$ and $Z$ is a normalization constant that makes the sum over $\mathfrak{X}$ equal to $1$.
\end{defi} 

Observe that for general choice of parameters the expressions in (\ref{eq:distr}) need not be non-negative. Consequently, we need to restrict the space of parameters so that $\mathbb{P}^{qR}$ is an honest probability measure. We isolate one possible choice that accomplishes this in the following definition.
\begin{defi}\label{ParSetQR}
We assume that the parameters $\alpha, \beta, \gamma, \delta, q \in \mathbb{R}$ and $M, N \in \mathbb{Z}$ are such that 
$$ M \geq N-1 \geq 0, \hspace{2mm} 1 > q > 0, \hspace{2mm} \alpha, \beta > 0, \hspace{2mm} \delta \geq 0, \hspace{2mm} \gamma = q^{-M-1}, \hspace{2mm} 1 > \beta \delta, \hspace{2mm} \beta \geq \gamma, \hspace{2mm} \alpha \geq \gamma.$$
One readily verifies that the above choice of parameters makes (\ref{eq:distr}) non-negative on all of $\mathfrak{X}$.
\end{defi}

We end this section by detailing the connection between $\mathbb{P}^{qR}$ and the measure on tilings from Section \ref{Section7.1} in the following theorem. 
\begin{thm}\label{TilingtoParticle}
Fix $a,b,c \geq 1$ and set $N = a$, $T = b+c$ and $S =c$. Let $\mathbb{P}$ denote the probability distribution of (\ref{DefEq1}) with parameters $q \in (0,1)$ and $\kappa \in \left[0,q^{(T-1)/2} \right)$. Fix $t \in \{0,1, \dots, T\}$ and let $(x^t_1, \dots, x^t_N)$ denote the random $N$-point configuration of Section \ref{Section7.1.2}. We have that 
\begin{enumerate}
\item if $t<S $ and $t < T - S$ then the distribution of $(x^t_1, \dots, x^t_N)$ is $\mathbb{P}^{qR}$ with $M = t + N - 1$, $\alpha = q^{-S-N}$, $\beta = q^{S - T - N}$, $\gamma = q^{-t - N}$ and $\delta = \kappa^2 q^{-S + N}$;
\item if $S- 1 < t < T-S+1$ then the distribution of $(x^t_1, \dots, x^t_N)$ is $\mathbb{P}^{qR}$ with $M = S + N - 1$, $\alpha = q^{-t-N}$, $\beta = q^{t - T - N}$, $\gamma = q^{-S - N}$ and $\delta = \kappa^2 q^{-t + N}$;
\item if $ T-S+1 < t < S$ then the distribution of $(T - t -S + x^t_1, \dots, T - t -S +  x^t_N)$ is $\mathbb{P}^{qR}$ with $M = T-S+N - 1$, $\alpha = q^{-T - N + t}$, $\beta = q^{ - t - N}$, $\gamma = q^{-T - N + S}$ and $\delta = \kappa^2 q^{-T + t + N}$;
\item if $S- 1 < t$ and $T-S-1 < t$ then the distribution of $(T - t -S + x^t_1, \dots, T - t -S +  x^t_N)$ is $\mathbb{P}^{qR}$ with $M = T-t+N - 1$, $\alpha = q^{-T - N + S}$, $\beta = q^{ - S - N}$, $\gamma = q^{-T -N  + t}$ and $\delta = \kappa^2 q^{-T + S +N}$.
\end{enumerate}
In all cases the parameter $q$ in the definition of $\mathbb{P}^{qR}$ is the same as the one that is given.
\end{thm}
\begin{proof}
This is essentially \cite[Theorem 4.1]{BGR} and we refer to the same paper for the details.
\end{proof}

\subsection{Global asymptotics for $q$-Racah ensembles}\label{Section7.4}
In this section we state a law of large numbers and a central limit theorem for the $q$-Racah ensembles ---  Theorems \ref{BT1} and \ref{QRCLT} below.

We begin by explaining how we are scaling the parameters in the $q$-Racah ensemble.
\begin{defi}\label{ScaleQR} We assume that we have parameters $ {\la}, {\lb}, {\tc}, {\td}, {\tq}$ and ${\tM}$ such that
$$1 > {\lq} > 0 , \hspace{2mm} {\la}, {\lb}, {\lM} > 0, \hspace{2mm}  {\ld} \geq 0, \hspace{2mm}1 > {\lb} {\ld}, \hspace{2mm} {\lc} = {\lq}^{-{\lM}}, \hspace{2mm} {\lb} \geq {\lc},  \hspace{2mm} {\la} \geq {\lc}, \hspace{2mm} \lc \lq > 1.$$
For future reference we denote the set of parameters $ {\la}, {\lb}, {\lc}, {\ld}, {\lq}$ and ${\lM}$  that satisfy the above conditions by ${\tt P}$ and view it as a subset of $\mathbb{R}^6$ with the subspace topology.

In addition, we assume that we have a sequence of parameters $\alpha_N, \beta_N, \gamma_N, \delta_N, q_N$ and $M_N$ that satisfy the conditions in Definition \ref{ParSetQR} and such that for some constant $A>0$ we have
$$\max \left(N\left| q_N - {\lq}^{1/N} \right|, \left| \alpha_N - {\la} \right|, \left|\beta_N -{\lb} \right|, \left|\gamma_N - {\lc} \right|,  \left|\delta_N - {\ld} \right| , \left| N^{-1} M_N - {\lM}\right| \right) \leq A N^{-1}.$$
We let $\mathbb{P}_N$ be the measure from Definition \ref{DefQR} with parameters $\alpha_N, \beta_N, \gamma_N, \delta_N, q_N, M_N$ and $N$.
\end{defi}

\begin{defi}\label{Polys} Suppose we are given parameters $ \pa, \pb, \pc, \pd$ and $\pq$ such that
$$1 > \pq> 0 , \hspace{2mm} \pa, \pb, \pc > 0, \hspace{2mm}  \pd \geq 0, \hspace{2mm}1 > \pb \pd, \hspace{2mm} \pb \geq \pc,  \hspace{2mm} \pa \geq \pc,  \hspace{2mm} \pc\pq \geq 1.$$
 We define the following polynomials
$$\Phi^{+}(z)=(z-\pa)(z-\pb\pd)(z-\pc)(z-\pc\pd), \hspace{5mm} \Phi^{-}(z)=(z-1)(\pa z-\pc \pd)(\pb z-\pc)(z-\pd)$$
$$R(z) = \Phi^+(z) + \Phi^-(z) - (\pa \pb \pq - 1)(\pq^{-1} - 1) (z^2 - \pc \pd)^2, \hspace{5mm} Q(z)^2 = R(z)^2 - 4\Phi^-(z) \Phi^+(z).$$
With the above data we define 
\begin{equation}\label{eq:mu}
\mu(x) = \begin{cases}\frac{1}{\pi} \cdot \mbox{arccos} \left( \frac{R(\pq^{-x})}{2\sqrt{\Phi^-(\pq^{-x}) \Phi^+(\pq^{-x})}} \right)   &\mbox{ when $x \in (0, -\log_{\pq} (\pc) )$}, \\  0 &\mbox{ otherwise}. \end{cases}
\end{equation}
If the expression inside the arccosine is bigger than $1$ we set $\mu = 0$ and if it is less than $-1$ we set $\mu = 1$. The square root is the usual one as on $(0, -\log_q(c))$ both $\Phi^-(\pq^{-x})$ and $\Phi^+(\pq^{-x})$ are positive.
\end{defi}

We also isolate the following fact.
\begin{lem}\label{LemmaQF} The polynomial $Q^2$ from Definition \ref{Polys} factors completely over $\mathbb{R}$. If we enumerate its roots in increasing order $x_1,\dots,x_8$ we get 
\begin{equation}\label{RootsQ}
x_1 = x_2 = - \sqrt{cd}, \hspace{5mm} d \leq x_3  \leq x_4 \leq c d, \hspace{5mm} x_5 = x_6 = \sqrt{c d}, \hspace{5mm} 1 \leq x_7 \leq x_8 \leq c.
\end{equation}
Moreover, we have $x_3 = \frac{cd}{x_8}$ and $x_4 = \frac{cd}{x_7}$.
\end{lem}
\begin{proof}
One readily checks that 
\begin{equation}\label{Q0}
z^{-4} q^2 Q(z)^2 = \left( z - c d z^{-1} \right)^2\cdot   Q_0\left( z + c d z^{-1}\right), \mbox{ where } Q_0(x) = a_2 x^2 + a_1 x + a_0, \mbox{ with }
\end{equation}
\begin{itemize}
\item $a_2 = (ab q^2 - 1)^2$
\item $a_1 = -2q(a^2 b^2 d q^2 + ab^2cd q^2 + a^2b^2 q^2 + a^2 b c q^2 + a b^2 d q^2 + a b c d q^2 + a^2 b q^2 - 2a b^2 d q - 2ab c d q+ a b c q^2 - 2a^2 b q  -2ab c q - 2 a b d q - 2 b c d q  + ab d - 2a b q - 2a c q + b c d + a b + a c + b d + c d + a + c),$
\item $a_0 = 4a^2b^2 c d q^3 - 4a^2 b^2 c d q^4 + 4ab^2c d^2 q^3 + a^2b^2 d^2 q^2 + 4a^2 b^2 d q^3 + 4a^2 bcd q^3 - 2ab^2cd^2q^2 + 4 a b^2 c d q^3 + 4abc^2 d q^3 + b^2 c^2d^2 q^2 - 2a^2 b^2 d q^2 - 2a^2 b c d q^2 + 4a^2 bc q^3 - 2 a b^2 c d q^2 - 2 a b^2 d^2 q^2 - 2a b c^2 d q^2 - 2a b c d^2 q^2 + 4 a b c d q^3 - 2b^2cd^2 q^2 - 2b c^2 d^2 q^2 + a^2 b^2 q^2 - 2a^2 b c q^2 - 2a^2 bd q^2 + a^2c^2 q^2 - 2ab^2 d q^2 - 16 ab c d q^2 - 2a c^2 d q^2 + b^2 d^2 q^2 - 2b d^2 d q^2 - 2b c d^2 q^2 + c^2 d^2 q^2 - 2a^2 b q^2 - 2a^2 c q^2 + 4 ab c d q - 2abc q^2 - 2ab c q^2 - 2a c^2 q^2 - 2ac d q^2 + 4b c d^2 q - 2 b c d q^2 - 2 c^2 d q^2 + a^2q^2 + 4a b d q + 4 ac d q  - 2a c q^2 + 4b c d q + 4 c^2 d q + c^2 q^2 + 4a c q + 4 c d q - 4c d$.
\end{itemize}
Consequently, what remains is to show that $Q_0$ has two real roots $y_1$ and $y_2$ such that $1 + c d \leq y_1 \leq y_2 \leq c + d$. Indeed, if the latter is true we would have that $z +cd z^{-1} = y_1$ (resp. $z + cd z^{-1} = y_2$) has two real roots $x_3, x_8$ (resp. $x_4, x_7$) and these satisfy the conditions of the lemma.

A direct calculation shows that the discriminant of $Q_0$ equals
$$D = a_1^2 - 4a_0 a_2 = 16 (q -1)(b q - 1) (b d q -1) (c q -1 ) (a q -1) (a q - d) (a b q - 1) (a b q - c) \geq 0.$$
Thus $Q_0$ indeed has two real roots and to show that they both lie in the interval $[1 + c d, c + d]$ it suffices to show that $Q_0'(1 + c d) \leq 0 $ and $Q_0'(c + d) \geq 0$. We notice that 
$$F_1 = Q_0'(1 + c d) = 2a_2(1 + c d) + a_1 (1 + c d) \mbox{ and } F_2 =Q_0'(c + d) = 2a_2(c + d ) + a_1 (c + d) $$
are both linear functions of $\gamma $ and by assumption $q^{-1} \leq c \leq \min(a,b)$. In particular, it suffices to check that $F_1 \leq 0$ when $c = q^{-1}$ and $c =b$, while $F_2 \geq 0$ when $c = q^{-1}$ and $c = a$. 

When $c = q^{-1}$ we have that 
$$F_1  = 2 (1 -q^{-1}) (b q -1 )( a b q^2 -1 )(a q -d) \leq 0 \mbox{ and } F_2 = 2 (1 -q^{-1})(a q -1) (a b q^2 -1) (b d q - 1) \geq 0.$$
When $c = b$ we have
$$F_1 = 2(b q -1)(a q -1)G_1 \mbox{ where } G_1 = ab^2 d q^2 + ab q^2 - 2a b q - 2bd q + b d + 1,$$
while when $c = a$ we have
$$F_2 = 2(b q -1)(a q -1)G_2 \mbox{ where } G_2 = a^2b  q^2 + ab d q^2 - 2a b d q - 2a q + a  + d.$$
What remains is to show that $G_1 \leq 0$ and $G_2 \geq 0$. Note that $G_1$ and $G_2$ are linear functions in $d$ and $d \in [0, b^{-1})$. Thus, it suffices to check that $G_1 \leq 0$ and $G_2 \geq 0$ when $d = 0$ and $d  = b^{-1}$.

When $d = 0$ we have  $G_1  =   ab q(q -1)  + (1 - a b q) \leq 0 \mbox{ and } G_2 =  a q( a b q - 1) + a(1 -q)  \geq 0.$
When $d = b^{-1}$ we have $G_1 = 2(q -1)(a b q - 1) \leq 0 \mbox{ and }  G_2 = b^{-1} \left( (a b q -1)^2 + ab (q - 1)^2 \right) \geq 0.$
\end{proof}

The formula for $Q$ that we will use in the paper is
\begin{equation}\label{QSquareRoot}
Q(x) = (abq - q^{-1}) \cdot (z^2 - cd)\cdot \sqrt{(z - x_3)(z - x_4)} \cdot\sqrt{ (z -x_7)(z - x_8)}
\end{equation} 

\subsubsection{Law of large numbers}\label{Section2.3.1} In this section we state a law of large numbers theorem for the $q$-Racah ensembles as Theorem \ref{BT1} below. Its proof will be established in Section \ref{Section9.2}. We assume we have the same parameters and measures $\mathbb{P}_N$ as in  Definition \ref{ScaleQR} and define the empirical measures $\mu_N$
\begin{equation}\label{MUN}
\mu_N = \frac{1}{N} \sum_{i = 1}^N \delta \left( \frac{\lambda_i}{N}\right) \mbox{ where } (\lambda_1,\dots,\lambda_N) \mbox{ is } \mathbb{P}_N-\mbox{distributed}.
\end{equation}

\begin{thm}\label{BT1} Under the assumptions in this section, we have that the measures $\mu_N$ concentrate (in probability) near $\mu(x) dx$, where $\mu(x)$ is as in Definition \ref{Polys} with parameters $a = \la, b = \lb, c =\lc, d = \ld$ and $q = \lq$. More precisely, for each Lipschitz function $f(x)$ defined in a real neighborhood of the interval $[0, \tM]$ and each $\varepsilon > 0$ the random variables
$$N^{1/2 - \varepsilon} \left|\int_{\mathbb{R}} f(x) \mu_N(dx) -  \int_{\mathbb{R}} f(x) \mu(x)dx\right|$$
converge to $0$ in probability and in the sense of moments.
\end{thm}

\subsubsection{Central limit theorem}\label{Section2.3.2}  In this section we state a central limit theorem for the $q$-Racah ensembles as Theorem \ref{QRCLT} below. Its proof will be established in Section \ref{Section9.2}. We assume the same parameters and measures $\mathbb{P}_N$ as in Definition \ref{ScaleQR}. It turns out that to better see the Gaussian structure of the $q$-Racah ensemble it is convenient to consider a transformed particle system, given by $(y_1, \dots, y_N)$ with $y_i = \sigma_N(q_N^{-\lambda_i})$ and $\sigma_N(z) = z + \gamma_N\delta_Nq_Nz^{-1}$. Then the transformed empirical measure of the system is given by
\begin{equation}\label{RHON}
\rho_N = \frac{1}{N} \sum_{i = 1}^N \delta \left( y_i\right)  = \frac{1}{N} \sum_{i = 1}^N \delta \left( \sigma_N(q_N^{-\lambda_i}) \right) \mbox{ where } (\lambda_1,\dots,\lambda_N) \mbox{ is } \mathbb{P}_N-\mbox{distributed}.
\end{equation}

\begin{thm}\label{QRCLT}
Take $m\geq 1$ polynomials $f_1, \dots, f_m \in  \mathbb{R}[x]$. Let $\rho_N$ be as in (\ref{RHON}) and define 
$$\mathcal L_{f_i}=N \int_\mathbb{R} f_j(x) \rho_N(dx) -N \mathbb E _{\mathbb P_N}  \left[ \int_\mathbb{R} f_j(x) \rho_N(dx) \right] \mbox{ for $i = 1, \dots, m$}.$$
Then the random variables $\mathcal{L}_{f_i}$ converge jointly in the sense of moments to an $m$-dimensional centered Gaussian vector $X = (X_1,\dots, X_m)$ with covariance
$$Cov(X_i, X_j) = \frac{1}{(2 \pi \iu)^2}\oint_{\Gamma} \oint_{\Gamma } f_i(s)f_j(t) \mathcal C (s, t) ds dt,$$
where $\Gamma$ is a positively oriented contour, which encloses the interval $[1 + \tc \td, \tc + \td]$. The covariance kernel $\mathcal C(s,t)$ is given by
\begin{equation}\label{QRCOV}
 \mathcal C(s,t) = -\frac{1}{2(s-t)^2} \left(1 - \frac{(s - a_-)(t- a_+) + (t - a_- )(s-a_+)}{2\sqrt{(s - a_- )(s- a_+) }\sqrt{(t - a_- )(t- a_+)} }\right),
\end{equation}
where $a_- = x_3 + x_8$, $a_+ = x_4 + x_7$ and $x_1, \dots, x_8$ are the ordered roots of the polynomial $Q(z)^2$ from Definition \ref{Polys} with parameters $a = \la, b = \lb, c =\lc, d = \ld$ and $q = \lq$, cf. Lemma \ref{LemmaQF}.
\end{thm}

\section{Global asymptotics for the $q$-Racah tiling model} \label{Section8}
As discussed in Section \ref{Section7.2} the $1D$ global fluctuations of our model are asymptotically described by an appropriate pullback of the Gaussian free field in $\mathbb{H}$. In Section \ref{Section8.1} we provide some preliminaries on the two-dimensional Gaussian free field. In Section \ref{Section8.2} we describe the complex structure $\Omega$ on the liquid region $\mathcal{D}'$ of our tiling model and show that $\Omega$ defines a bijection between $\mathcal{D}'$ and $\mathbb{H}$. In Section \ref{Section8.3} we give the proof of Theorems \ref{BT1} and \ref{QRCLT}. Finally, in Section \ref{Section8.4} we state our $2$D conjecture and give some numerical evidence that supports it.

\subsection{Gaussian free field}\label{Section8.1}
In this section we briefly recall the formulation and some basic properties of the Gaussian free field (GFF). Our discussion will follow the exposition in \cite[Section 4.5]{BGJ} and for a more thorough background on the subject we refer to \cite{Sheff}, \cite[Section 4]{Dub}, \cite[Section 2]{HMP}, and the references therein.

\begin{defi} The Gaussian free field with Dirichlet boundary conditions in the upper half-plane $\mathbb{H}$ is a (generalized) centered Gaussian field $\mathcal{F}$ on $\mathbb{H}$ with covariance given by
\begin{equation}\label{GFFCov}
\mathbb{E} \left[ \mathcal{F}(z) \mathcal{F}(w)\right] = - \frac{1}{2\pi} \log \left| \frac{z - w}{z - \overline{w}} \right|, \hspace{2mm} z,w \in \mathbb{H}.
\end{equation}
\end{defi}
We remark that $\mathcal{F}$ can be viewed as a probability Gaussian measure on a suitable class of generalized functions on $\mathbb{H}$; however, one cannot define the value of $\mathcal{F}$ at a given point $z \in \mathbb{H}$ (this is related to the singularity of (\ref{GFFCov}) at $z = w$).

Even though $\mathcal{F}$ does not have a pointwise value; one can define the (usual distributional) pairing $\mathcal{F}(\phi)$, whenever $\phi$ is a smooth function of compact support, and the latter is a mean zero normal random variable. In general, one can characterize the distribution of $\mathcal{F}$ through pairings with test functions as follows. If $\{\phi_k\}$ is any sequence of compactly supported smooth functions on $\mathbb{H}$ then the pairings $\{ \mathcal{F}(\phi_k) \}$ form a sequence of centered normal variables with covariance
$$\mathbb{E} \left[ \mathcal{F}(\phi_k) \mathcal{F} (\phi_l)\right] = \int_{\mathbb{H}^2} \phi_k(z) \phi_l(w) \left( - \frac{1}{2\pi} \log \left| \frac{z - w}{z - \overline{w}} \right| \right)|dz|^2|dw|^2.$$

An important property of $\mathcal{F}$ that will be useful for us is that it can be integrated against smooth functions on smooth curves $\gamma \subset \mathbb{H}$. We isolate the statement in the following lemma.
\begin{lem}\label{GFFCH}\cite[Lemma 4.6]{BGJ} Let $\gamma \subset \mathbb{H}$ be a smooth curve and $\mu$ a measure on $\mathbb{H}$, whose support is $\gamma$ and whose density with respect to the natural (arc length) measure on $\gamma$ is a given by a smooth function $g(z)$ such that
\begin{equation}\label{VarCont}
\iint\limits_{\gamma \times \gamma} g(z) g(w) \left(  - \frac{1}{2\pi} \log \left| \frac{z - w}{z - \overline{w}}\right| \right) dz dw < \infty.
\end{equation}
Then
$$ \int_{\mathbb{H}} \mathcal{F} d\mu = \int_\gamma \mathcal{F}(u) g(u) du$$
 is a well-defined Gaussian centered random variable of variance given by (\ref{VarCont}).
 Moreover, if we have two such measures $\mu_1$ and $\mu_2$ (with two curves $\gamma_1$ and $\gamma_2$ and two densities $g_1$ and $g_2$), then $X_1 = \int_{\gamma_1} \mathcal{F}(u) g_1(u)du$, $X_2 = \int_{\gamma_2} \mathcal{F}(u) g_2(u)du$ are jointly Gaussian with covariance 
$$\mathbb{E}[X_1 X_2] = \iint\limits_{\gamma_1 \times \gamma_2} g_1(z) g_2(w) \left(  - \frac{1}{2\pi} \log \left| \frac{z - w}{z - \overline{w}}\right| \right) dz dw.$$
\end{lem}

Another property of $\mathcal{F}$ that we require is that it behaves well under bijective maps, which leads to the notion of pullback.
\begin{defi}\label{Pullback}
Given a domain $D$ and a bijection $\Omega: D \rightarrow \mathbb{H}$, the pullback $\mathcal{F} \circ \Omega$ is a generalized centered Gaussian field on $D$ with covariance
$$\mathbb{E} \left[\mathcal{F}(\Omega(z)) \mathcal{F}(\Omega(w)) \right] =  - \frac{1}{2\pi} \log \left| \frac{\Omega(z) - \Omega(w)}{\Omega(z) - \overline{\Omega}(w)} \right|, \hspace{2mm} z,w \in D.$$
Integrals of $\mathcal{F} \circ \Omega$ with respect to measures can be computed through
$$ \int_D (\mathcal{F} \circ \Omega)d\mu = \int_{\mathbb{H}} \mathcal{F} d\Omega(\mu),$$
where $d\Omega(\mu)$ stands for the pushforward of the measure $\mu$.
\end{defi}
The above definition immediately implies the following analogue of Lemma \ref{GFFCH}.
\begin{lem}\label{GFFCD}\cite[Lemma 4.8]{BGJ} In the notation of Definition \ref{Pullback}, let $\mu$ be a measure on $D$ whose support is a smooth curve $\gamma$ and whose density with respect to the natural (length) measure on $\gamma$ is given by a smooth function $g(z)$ such that
\begin{equation}\label{VarContD}
\iint\limits_{\gamma \times \gamma} g_1(z) g_2(w) \left(  - \frac{1}{2\pi} \log \left| \frac{\Omega(z) - \Omega(w)}{\Omega(z) - \overline{\Omega}(w)}\right| \right) dz dw < \infty.
\end{equation}
Then
$$ \int_{D} (\mathcal{F} \circ \Omega) d\mu = \int_\gamma \mathcal{F}(\Omega(u)) g(u) du$$
 is a well-defined Gaussian centered random variable of variance given by (\ref{VarContD}).
 Moreover, if we have two such measures $\mu_1$ and $\mu_2$ (with two curves $\gamma_1$ and $\gamma_2$ and two densities $g_1$ and $g_2$), then $X_1 = \int_{\gamma_1} \mathcal{F}(\Omega(u)) g_1(u)du$, $X_2 = \int_{\gamma_2} \mathcal{F}(\Omega(u)) g_2(u)du$ are jointly Gaussian with covariance 
$$\mathbb{E}[X_1 X_2] = \iint\limits_{\gamma_1 \times \gamma_2} g_1(z) g_2(w) \left(  - \frac{1}{2\pi} \log \left| \frac{\Omega(z) - \Omega(w)}{\Omega(z) - \overline{\Omega}(w)}\right| \right) dz dw.$$
\end{lem}

We end this section by remarking that the Gaussian free field is conformally invariant: if $\phi$ is an automorphism of $\mathbb{H}$ (i.e. $\phi(z) = \frac{az + b}{cz +d}$ with $a,b,c,d \in \mathbb{R}$ and $ad - bc = 1$) then the distributions of $\mathcal{F}$ and $\mathcal{F} \circ \phi$ are the same.

\subsection{Complex structure}\label{Section8.2} In this section we adopt the same notation as in Section \ref{Section7.2} and formulate the map $\Omega$. We first observe that if $(x,y) \in \mathcal{P}$ then we have that $\phi(x,y) = \pi \cdot \mu\left(  y / \lN\right) $, where $\mu$ is as in Definition \ref{Polys} for the parameters $\pq = \lq^{\lN}$, $\pa = \lq^{-\lS - \lN}$, $\pb = \lq^{\lS - \lT - \lN}$, $\pc = \lq^{-x -\lN}$ and $d = \lk^2 \lq^{-\lS + \lN}$. If we have that $R, Q, \Phi^{\pm}$ are as in Definition \ref{Polys} with the same parameters then the liquid region $\mathcal{D}$ is given by
\begin{equation}
\mathcal{D} = \left\{ (x,y) \in \mathcal{P}: Q(\lq^{-y })^2 < 0 \right\} =  \left\{ x \in (0, \lT) \mbox{ and } y \in (- \log_{\lq} (x_7),- \log_{\lq} (x_8)) \right\},
\end{equation}
where $x_7,x_8$ stand for the two roots of $Q^2$ in $(1, c)$ --- see Lemma \ref{LemmaQF}. From (\ref{Q0}) we know
$$z^{-4} q^2 Q(z)^2 = \left( z - c d z^{-1} \right)^2\cdot   Q_0\left( z + c d z^{-1}\right), $$ 
and  if we set $u = \lq^{-x}$ and $v = \lq^{-y} + \lk^2 \lq^{-\lS - x + y}$ then we see that
$$Q_0\left( \lq^{-y} + c d \lq^{y}\right) = \tilde{Q}(u,v) = \tilde{A} u^2 + \tilde{B} v^2 + \tilde{C} uv + \tilde{D}u + \tilde{E}v + \tilde{F},$$
where $\tilde{A},\tilde{B},\tilde{C},\tilde{D},\tilde{E},\tilde{F}$ are explicit constants that depend only on $\lq, \lS, \lN, \lT$ and $\lk^2$ and not on $x,y$. Combining the last two observations, we see that
\begin{equation}\label{Dprime}
\mathcal{D}' = \left\{ (u,v) : \tilde{Q}(u,v) < 0 \right\} =  \left\{ (u,v): u \in (1, \lq^{-\lT} ) \mbox{ and } v \in \left(x_7 + cd x_7^{-1}, x_8 + cdx_8^{-1}\right) \right\}.
\end{equation}
where we recall from Section \ref{Section7.2} that $\mathcal{D}'$ is the image of $\mathcal{D}$ under the map $\sigma_{\lq}(x,y) = (\lq^{-x}, \lq^{-y} + \lk^2 \lq^{-\lS - x + y})$. In particular, we have that $\mathcal{D}'$ is an ellipse.

We next consider the quadratic equation
\begin{equation}\label{QuadCompl}
\begin{split}
&P(w; u,v):= a_2(u,v) w^2 + a_1(u,v) w + a_0(u,v) = 0, \mbox{ where } a_2 = \lq^{\lN}( v - 1 - \lk^2 \lq^{-\lS}u), \\
&a_1 = v\lq^{\lN}(\lq^{-\lT} - 1) + \left(u(\lq^{-\lS} - \lq^{\lN}) - \lq^{-\lS + \lN} - \lq^{-\lT} + 2\lq^{\lN}\right) + u \lk^2 \lq^{\lN}(\lq^{-\lT} + \lq^{-\lS + \lN} - 2\lq^{-\lS - \lT})\\
&+ \lk^2 \lq^{-\lT + \lN}(\lq^{-\lS} - \lq^{\lN} ) \mbox{ and } a_0 = (u - 1)(\lq^{-\lT} - \lq^{\lN}) (\lq^{-\lS} - 1)(1 - \lk^2 \lq^{-\lT + \lN}).
\end{split}
\end{equation}
For the above equation one calculates $a_1^2 - 4a_2a_0 = \lq^{2\lN} \cdot \tilde{Q}(u,v)$ and so for $(u,v) \in \mathcal{D}'$ we have that the equation has two complex conjugate roots. We define the map $\Omega: \mathcal{D}' \rightarrow \mathbb{H}$ as 
\begin{equation}\label{OmegaDef}
\Omega(u,v) = w(u,v) \mbox{ such that } P(w(u,v); u,v) = 0 \mbox{ and } w(u,v) \in \mathbb{H} \mbox{ for $(u,v) \in \mathcal{D}'$}
\end{equation}
and from our earlier discussion $\Omega$ is well-defined and algebraic. \\

In the remainder of this section we show that $\Omega$ defines a bijective diffeomorphism between $\mathcal{D}'$ and $\mathbb{H}$ satisfies an important property that is used in the proof of Theorem \ref{TCLT} in the next section.

For convenience we denote 
\begin{equation*}
\begin{split}
&\lambda_3 = (1-\lq^{\lN})(1 - \lk^2 \lq^{\lN})(\lq^{-\lS} - \lq^{-\lT}), \hspace{2mm}\lambda_2 = \lk^2 \lq^{\lN} \left(\lq^{-\lS - \lT } -  \lq^{-\lS + \lN} +\lq^{-\lS } - \lq^{-\lT }\right) - \lq^{-\lS} + \lq^{\lN}, \\
&\lambda_1 = -(\lq^{-\lT} - \lq^{\lN})(1 - \lk^2 \lq^{-\lT + \lN})(\lq^{-\lS} - 1)  \mbox{ and } \lambda_0 = -(\lq^{-\lT} - 1)(\lq^{-\lT} - \lq^{\lN})(1 - \lk^2\lq^{-\lT + \lN})(\lq^{-\lS} - 1).
\end{split}
\end{equation*}
Also we define the map $f: \mathbb{H} \rightarrow \mathbb{R}^2 \mbox{ through } f(r + \iu s) = (f_1(r,s), f_2(r,s)) \mbox{ with }$
\begin{equation}\label{inverseOmega}
\begin{split}
&f_1(r,s) = 1 + \frac{\lambda_3(r^2 + s^2)}{\lambda_2 (r^2 + s^2) + 2r \lambda_1 +\lambda_0}, \hspace{2mm } f_2(r,s) =  1 + \lk^2 \lq^{-\lT} + \\
&+ \frac{\lk^2 \lq^{\lN} (\lq^{-\lS} - 1)(\lq^{-\lT} - \lq^{-\lS})(1 - \lk^2 \lq^{-\lT})}{\lambda_2}- \frac{\lambda_1\lambda_3(\lambda_2 + \lk^2\lq^{-\lS + \lN} (\lq^{-\lT} + 2r - 1))}{\lq^{\lN}\lambda_2(\lambda_2 (r^2 + s^2) + 2r \lambda_1 +\lambda_0)}.
\end{split}
\end{equation}
We observe that 
$$\tilde{Q}(f_1(r,s), f_2(r,s)) = - \frac{4 \cdot \lq^{-2\lN} \lambda_3^2 \lambda_1^2 s^2 }{(\lambda_2 (r^2 + s^2) + 2r \lambda_1 +\lambda_0)^2} < 0,$$
and so $f$ maps $\mathbb{H}$ in $\mathcal{D}'$. One directly checks that $f \circ \Omega$ and $\Omega \circ f$ are the identities on $\mathcal{D}'$ and $\mathbb{H}$ respectively, which shows that $\Omega$ has our desired properties. 

\begin{rem}
Let us give some ideas about how the formula for $\Omega$ was discovered. Once the appropriate physical coordinates $u,v$ for the system are found, which lead to the liquid region $\mathcal{D}'$ being an ellipse, one suspects that the map $\Omega$ should be given by the solution in $\mathbb{H}$ of {\em some} quadratic equation $a_2 w^2 + a_1w + a_0$, whose discriminant $D = a_1^2 - 4a_0 a_2$ is negative precisely on $\mathcal{D}'$. In particular, we expect that $a_1^2 - 4a_0 a_2 = \lambda \tilde{Q}(u,v)$ for some positive parameter $\lambda$.

In \cite{P2} the complex structure for the uniform tiling case (this is $\kappa = 0$ and $q = 1$ in our model) was given by a quadratic equation, whose coefficients are {\em linear} in the coordinates of the system. By analogy we guess that $a_i = a_i^1u + a_i^2v + a_i^3$ for $i = 1,2,3$ in our case as well, where the new coefficients do not depend on $u$ and $v$. This gives us a nine parameter system. 

When searching for a map $\Omega$ one has a choice of which point of the boundary of $\mathcal{D}'$ should be sent to infinity. In our case, we choose the boundary point at $(u, 1 + u\lk^2 \lq^{-\lS})$ with $u = 1 + \lambda_3 \cdot \lambda_2^{-1}$  to be sent to infinity, which gives us $2$ equations for our $9$ parameters. In addition, the relationship $a_1^2 - 4a_0 a_2 = \lambda \tilde{Q}(u,v)$ gives an additional $6$ equations (comparing the coefficients in front of $u^iv^j$) and an extra parameter $\lambda$. Overall we have a ten parameter system with eight equations.

The resulting system has a $2$-parameters set of solutions. The extra freedom comes from multiplying $a_2, a_1, a_0$ by the same positive constant and also from multiplying $a_2$ and dividing $a_0$ by the same positive constant. Observe that the resulting complex structures are all equivalent modulo a multiplication by a positive constant, which is an automorphism of $\mathbb{H}$. Our particular, choice for the parameters is dictated by the {\em product form} of the coefficient $a_0$ in (\ref{QuadCompl}).
\end{rem}

\begin{rem}\label{remCS}
As mentioned in Section \ref{Section1.1.2} there is a natural complex coordinate one can define on the liquid region $\mathcal{D}$, called {\em complex slope}. Let us explain how to construct it briefly -- see \cite{KO, K1} for more details. Suppose $(x,y) \in \mathcal{D}$ and set $(p_1, p_2, p_3)$ to be the normal vector to the limit shape $\hat{h}$ at $(x,y)$ such that $p_1 + p_2 + p_3 = 1$. Then the complex slope $z(x,y)$ is the unique point in $\mathbb{H}$ such that the triangle $(0,1,z)$ has angles $(\pi p_1, \pi p_2, \pi p_3)$. In the case of the uniform tilings of the hexagon (and more general domains) it is known that there is an {\em algebraic} relationship between $z(x,y)$ and the complex structure $\Omega(x,y)$ , whose pullback establishes the connection with the GFF on $\mathbb{H}$, \cite{KO, P2}. For the $q$-Racah tiling model an expression for $z(x,y)$ was obtained in \cite[Section 8.1]{BGR} and it is related to $\Omega(x,y)$ from (\ref{OmegaDef}) as follows. If we set 
$$U = U(x,y) = \frac{z(x,y) \lq^{x} - \lk^2 \lq^{-\lS + 2y}}{1 - z(x,y)\lk^2 \lq^{-\lS + 2y - x}} \mbox{ and  } \Omega = \Omega(x,y) \mbox{ then }$$
$$U = \frac{ \Omega ( \lq^{-\lS} - \lq^{\lN}) + (\lq^{-\lS } - 1)( \lq^{-\lT} - \lq^{\lN}) - \lk^2\lq^{\lN} (\Omega +\lq^{-\lT}  - \lq^{\lN})(\Omega \lq^{-\lS} + \lq^{-\lT - \lS}  - \lq^{-\lT})}{(\Omega \lq^{\lN} + \lq^{-\lT} - \lq^{\lN})(\Omega + \lq^{-\lS} -1) - \lk^2  \lq^{\lN - \lT} [ (\lq^{-\lS} - \lq^{\lN})\Omega + (\lq^{-\lT} - \lq^{\lN})(\lq^{-\lS} - 1)]}.$$
\end{rem}

We end the section with the following result that will be required in the next section.
\begin{lem}Suppose that $u, v_1, v_2 \in \mathbb{R}$ are such that $(u, v_1), (u,v_2) \in \mathcal{D}'$. Then we have
\begin{equation}\label{OmegaLog}
-\log\left| \frac{\Omega(u, v_1) - \Omega(u,v_2)}{\Omega(u, v_1) - \overline{\Omega}(u, v_2)}\right| = \log\left| \frac{\sqrt{(v_1-a)(b-v_2)}  + \sqrt{(v_2-a)(b-v_1)}}{\sqrt{(v_1-a)(b-v_2)} -\sqrt{(v_2-a)(b-v_1)}}\right|,
\end{equation}
where $a < b$ denote the intersection points of the vertical line through $u$ with the ellipse $\tilde{Q}(u,v) = 0$.
\end{lem}
\begin{proof}
Note that if $\phi(z)$ is an automorphism of $\mathbb{H}$, i.e. $\phi(z) = \frac{m \cdot z+ n}{k \cdot z + l}$ with $m,n,k,l \in \mathbb{R}$ and $ml - $ $nk = 1$, then the LHS of (\ref{OmegaLog}) is the same upon replacing $\Omega(u,v_i)$ with $\phi(\Omega(u,v_i))$ for $i = 1,2$. Set 
\begin{equation}
m = \frac{-a_1 + v_i\lq^{\lN}(\lq^{-\lT} - 1)}{\sqrt{2 a_0(u,v_i)}}, \hspace{2mm} n = -\sqrt{2a_0(u,v_i)}, \hspace{2mm} k = \frac{1}{\sqrt{2a_0(u,v_i)}}, \hspace{2mm} l = 0,
\end{equation}
and observe that the above do not change if we take $i = 1$ or $2$. Moreover, by our choice of parameters we know that $m,n,k,l$ satisfy the earlier conditions and we let $\phi$ denote the automorphism corresponding to this quadruple. Setting $D(u,v_i) = a_1(u,v_i)^2 - 4a_0(u,v_i)a_2(u,v_i)$ we see that
$$\phi( \Omega(u,v_i)) = \frac{(-a_1 + v_i \lq^{\lN}(\lq^{-\lT} - 1))(-a_1+ \sqrt{D(u,v_i)}) - 4a_0a_1}{- a_1 + \sqrt{D}} = v_i\lq^{\lN}(\lq^{-\lT} - 1) - \sqrt{D},$$
where in the second equality we multiplied the numerator and denominator by $-a_1 - \sqrt{D}$ and used that $a_1^2 - D = 4a_0a_2$. Recalling that $ \sqrt{D(u,v_i)} = \lq^{\lN} \sqrt{\tilde{Q}(u,v_i)} =  \lq^{\lN} (\lq^{-\lT} - 1)\sqrt{ (v_i - a)(v_i - b)}$, where $a,b$ are as in the statement of the lemma, we see that 
$$
 \frac{\Omega(u, v_1) - \Omega(u,v_2)}{\Omega(u, v_1) - \overline{\Omega}(u, v_2)} = \frac{(v_1-v_2) + \sqrt{ (v_1 - a)(v_1 - b)} - \sqrt{ (v_2 - a)(v_2 - b)}}{(v_1-v_2) + \sqrt{ (v_1 - a)(v_1 - b)} + \sqrt{ (v_2 - a)(v_2 - b)}} .$$
Taking absolute value on both sides above and squaring we get
$$\left|\frac{\Omega(u, v_1) - \Omega(u,v_2)}{\Omega(u, v_1) - \overline{\Omega}(u, v_2)} \right|^2 = \frac{(b-v_1)(v_2 -a) + (b-v_2)(v_1 - a) - 2\sqrt{(v_1 - a)(b- v_1) (v_2 - a)(b - v_2)}}{(b-v_1)(v_2 -a) + (b-v_2)(v_1 - a) + 2\sqrt{(v_1 - a)(b- v_1) (v_2 - a)(b - v_2)}}.$$
If we take logarithms on both sides of the above and multiply the result by $-1/2$ we get (\ref{OmegaLog}).
\end{proof}

\subsection{Proof of Theorems \ref{TLLN} and \ref{TCLT} } \label{Section8.3} 

\subsubsection{Proof of Theorem  \ref{TLLN} }
 We suppose that we have a sequence $\varepsilon_k$, which converges to $0^+$ and also sequences $q(\varepsilon_k),  N(\varepsilon_k), T(\varepsilon_k), S(\varepsilon_k)$ and $\kappa(\varepsilon_k)$ as in Definition \ref{ParScale}. Let us define $t(\varepsilon_k) = \lfloor x \varepsilon_k^{-1} \rfloor$ and observe that in this notation we have for all large $k$ that
\begin{equation}\label{LLNeq1}
\varepsilon_k \cdot h\left( \lfloor x \varepsilon_k^{-1} \rfloor , \lfloor y \varepsilon_k^{-1} \rfloor +  1/2 \right) = \varepsilon_k \cdot  N \cdot \int_{\mathbb{R}} {\bf 1}_{\left\{r < y \varepsilon_k^{-1} \right\}} \mu^t(dr),
\end{equation}
where $\mu^t = N^{-1} \sum_{i = 1}^N \delta \left( x^t_i / N\right) $. By possibly passing to a subsequence we may assume that the parameters $t, S,T$ fall into one of the four cases in Theorem \ref{TilingtoParticle}. These cases need to be handled separately, but as the arguments are analogous we assume that we are in the case $t <\min(S, T - S)$.

It follows from Theorem \ref{TilingtoParticle}, (\ref{LLNeq1}) and the definition of $\hat{h}$ that for large $k$ we have
\begin{equation}\label{LLNeq2}
\begin{split}
&p(\varepsilon_k):= \mathbb{P}_{\varepsilon_k} \left( \left|  \varepsilon_k \cdot h\left( \lfloor x \varepsilon_k^{-1} \rfloor , \lfloor y \varepsilon_k^{-1} \rfloor +  1/2 \right)  -  \hat{h}(x,y) \right| > \eta \right) = \\
& \mathbb{P}_N\left( \left|  \varepsilon_k \cdot  N \cdot \int_{\mathbb{R}} {\bf 1}_{\left\{ r < y \varepsilon_k^{-1}N^{-1} \right\}} \mu_N(dr) -  \int_0^y \mu\left(  r / \lN\right) dr \right| > \eta \right),
\end{split}
\end{equation}
where $\mathbb{P}_N$, $\mu_N$ are as in the statement of Theorem \ref{BT1} for the parameters $q_N = q$,  $M_N = t + N - 1$, $\alpha_N = q_N^{-S-N}$, $\beta_N = q_N^{S - T - N}$, $\gamma_N = q_N^{-t - N}$ and $\delta_N = \kappa^2 q_N^{-S + N}$ and $\mu$ is as in Definition \ref{Polys} for the parameters $\pq = \lq^{\lN}$, $\pa = \lq^{-\lS - \lN}$, $\pb = \lq^{\lS - \lT - \lN}$, $\pc = \lq^{-x -\lN}$ and $d = \lk^2 \lq^{-\lS + \lN}$.

For $\tilde{\delta} > 0$ we let $f_{\tilde{\delta}}$ be a smooth function, such that $f_{\tilde{\delta}} = 1$ on $[0,y/\lN]$, its support lies in $(-\tilde{\delta}, y/\lN + \tilde{\delta})$, $f_{\tilde{\delta}}(x) \in [0,1]$ for all $x$. Then choosing $\tilde{\delta}$ sufficiently small we have for all large $k$ that
$$p(\varepsilon_k) < \mathbb{P}_N\left( \left|  \lN \cdot \int_{\mathbb{R}} f_{\tilde{\delta}}(r) \mu_N(dr) - \lN \cdot \int_{\mathbb{R}} f_{\tilde{\delta}}(r) \mu(r) dr \right| > \eta/2 \right),$$
where we used that $N = \lN\epsilon_k^{-1}  + o(1)$, $f_{\tilde{\delta}}(x)$, $\mu(x)$ are both in $[0,1]$ and we performed a change of varibles for the integral involving $\mu$. From Theorem \ref{BT1} we know that the RHS above converges to $0$ as $k\rightarrow \infty$, which proves the theorem. 

\subsubsection{Proof of Theorem  \ref{TCLT} }
We suppose that we have a sequence $\varepsilon_k$, which converges to $0^+$ and also sequences $q(\varepsilon_k),  N(\varepsilon_k), T(\varepsilon_k), S(\varepsilon_k)$ and $\kappa(\varepsilon_k)$ as in Definition \ref{ParScale}. In addition, we assume that $R_i$ are real polynomials such that $R'_i(x) = f_i(x)$ for $i = 1,\dots, m$.

Observe that for all large $k$ we have
\begin{equation}\label{CLTeq1}
\begin{split}
\int_{\mathbb{R}} \left(\mathcal{H}(q^{-t}, v) - \mathbb{E}_{\mathbb{P}_{\varepsilon_k}} \left[ \mathcal{H}(q^{-t}, v) \right] \right) f_i(v)dv = \int_{1}^{R}  \left(\mathcal{H}(q^{-t}, v) -  \mathbb{E}_{\mathbb{P}_{\varepsilon_k}}  \left[ \mathcal{H}(q^{-t}, v) \right] \right) f_i(v)dv,
\end{split}
\end{equation}
where $R = \lq^{-\lS - \lN} + \lk^2 \lq^{-\lS - \lT} + 1$. The latter truncation is allowed since a.s. all particles will have $v$-coordinate in $[1, R]$, which makes the height function $\mathcal{H}$ deterministic outside this interval and the above integrand zero there. In addition, we observe that
\begin{equation}\label{CLTeq2}
\begin{split}
\int_{1}^{R} \mathcal{H}(q^{-t}, v)  f_i(v)dv = \sum_{j= 1}^N \int_{V_j}^{V_{j+1}} j \cdot  f_i(v)dv = - \sum_{j = 1}^N R_i(V_j) + NR_i(R),
\end{split}
\end{equation}
where $V_j = V(t,j)$ for $j = 1, \dots, N$ (see (\ref{transform})) and $V_{N+1} = R$. Combining (\ref{CLTeq1}) and (\ref{CLTeq2}) we conclude that for all large $k$ we have
\begin{equation}\label{CLTeq3}
\begin{split}
\int_{\mathbb{R}} \left(\mathcal{H}(q^{-t}, v) - \mathbb{E}_{\mathbb{P}_{\varepsilon_k}} \left[ \mathcal{H}(q^{-t}, v) \right] \right) f_i(v)dv = - N \hspace{-1mm} \int_\mathbb{R} \hspace{-1mm} R_j(x)\rho^t(dx) + N \mathbb{E}_{\mathbb{P}_{\varepsilon_k}}   \left[ \int_\mathbb{R} \hspace{-1mm} R_j(x) \rho^t(dx) \right] ,
\end{split}
\end{equation}
where $\rho^t = \frac{1}{N} \sum_{j = 1}^N  \delta \left( V_j\right) $. By possibly passing to a subsequence we may assume that the parameters $t, S,T$ fall into one of the four cases in Theorem \ref{TilingtoParticle}. These cases need to be handled separately, but as the arguments are analogous we will assume that we are in the first case $t <\min(S, T - S)$. 

It follows from Theorem \ref{TilingtoParticle} that $\rho^t$ under law $\mathbb{P}_{\varepsilon_k}$ has the same distribution as $\rho_N$ under law $\mathbb{P}_N$, where  $\rho_N$ is as in (\ref{RHON}) and $\mathbb{P}_N$ is as in Definition \ref{ScaleQR} for the parameters $q_N = q$,  $M_N = t + N - 1$, $\alpha_N = q_N^{-S-N}$, $\beta_N = q_N^{S - T - N}$, $\gamma_N = q_N^{-t - N}$ and $\delta_N = \kappa^2 q_N^{-S + N}$. If we denote by $X^{\varepsilon_k}_i$ the RHS of (\ref{CLTeq3}) for $i = 1,\dots, m$ we conclude from Theorem \ref{QRCLT} that $X^{\varepsilon_k}_i$ converge as $k\rightarrow \infty$ to a Gaussian vector $(X_1, \dots,X_m)$ which has zero mean and covariance
$$\mathbb{E} \left[ X_i X_j \right] = \frac{1}{(2\pi \iu)^2} \oint_{\gamma} \oint_{\gamma} \frac{ R_i(v_1)R_j(v_2)}{2(v_1-v_2)^2} \left(-1 + \frac{(v_1 - a)(v_2- b) + (v_2 - a) (v_1-b)}{2\sqrt{v_1 - a}\sqrt{v_1- b}\sqrt{v_2 - a }\sqrt{v_2- b }} \right) dv_1 dv_2,$$
where $\gamma$ is a positively oriented contour, which encloses the interval $[1 +  u\lk^2 \lq^{-\lS},  u\lq^{ -\lN} + \lk^2 \lq^{-\lS + \lN}]$ and the square roots are defined with respect to the principal branch of the logarithm. In deriving the above we implicitly used Lemma \ref{LemmaQF} and (\ref{Dprime}). To complete the proof it suffices to show  
\begin{equation}\label{StartGFF}
\begin{split}
&\frac{\pi}{(2\pi \iu)^2} \oint_{\gamma} \oint_{\gamma}  \frac{R_i(v_1)R_j(v_2)}{2(v_1-v_2)^2} \left(-1 + \frac{(v_1 - a)(v_2- b) + (v_2 - a) (v_1-b)}{2\sqrt{v_1 - a}\sqrt{v_1- b}\sqrt{v_2 - a }\sqrt{v_2- b }} \right)dv_1 dv_2 =\\ 
& \int_a^b \int_a^b R'_i(x) R_j'(y) \left(  - \frac{1}{2\pi} \log \left| \frac{\Omega(u,x) - \Omega(u,y)}{\Omega(u,x) - \overline{\Omega}(u,y)}\right| \right) dx dy, \mbox{where $\Omega$ is as in (\ref{OmegaDef}).}
\end{split}
\end{equation}

We start with the LHS of (\ref{StartGFF}) and deform the $v_2$ contour so that it traverses the segment $[a,b]$ once in the positive and once in the negative direction. Observe the square roots are purely imaginary and come with opposite sign when we approach $[a,b]$ from the upper and lower half-planes. On the other hand, the term $\frac{1}{2(v_1-v_2)^2} $ cancels when we integrate over $[a,b]$ in the positive and negative direction. By Cauchy's theorem we do not change the value of the integral during the deformation and so from the Bounded convergence theorem we see that the LHS of (\ref{StartGFF}) equals
$$\frac{\pi \iu}{(2\pi \iu)^2}\oint_{\gamma} \int_{a}^b  \frac{R_i(v_1)R_j(v_2)}{(v_1-v_2)^2} \cdot \frac{(v_1 - a)(v_2- b) + (v_2 - a) (v_1-b)}{2\sqrt{v_1 - a}\sqrt{v_1- b}\sqrt{v_2 - a }\sqrt{b-v_2 }}dv_2 dv_1.$$
We integrate by parts in the $v_2$ variable and change the order of the integrals, which leads to the following expression for the LHS in (\ref{StartGFF})
\begin{equation}\label{GFF1}
\frac{\pi \iu}{(2\pi \iu)^2} \int_{a}^b\oint_{\gamma} \frac{R_i(v_1)R'_j(v_2)}{(v_1-v_2)}\cdot \frac{\sqrt{v_2 - a }\sqrt{b-v_2 }}{\sqrt{v_1 - a}\sqrt{v_1- b}}dv_1 dv_2.
\end{equation}

At this time we claim that for each $v_2 \in (a,b)$ we have
\begin{equation}\label{GFF2}
\begin{split}
&\oint_{\gamma} \frac{R_i(v_1)}{(v_1-v_2)} \cdot \frac{\sqrt{v_2 - a }\sqrt{b-v_2 }}{\sqrt{v_1 - a}\sqrt{v_1- b}}dv_1  = 2 \iu  \int_{a}^b R'_i(v_1) \times \\
&\left[ 2 \log \left( \sqrt{(v_1-a)(b-v_2)} + \sqrt{(v_2-a)(b-v_1)} \right) - \log|v_1-v_2| - \log(b-a)\right] dv_1.
\end{split}
\end{equation}
We will prove (\ref{GFF2}) below. For now we assume its validity and finish the proof of (\ref{StartGFF}). 

From (\ref{GFF1}) and (\ref{GFF2}) we see that to show (\ref{StartGFF}) it suffices to have
\begin{equation}\label{GFF3}
\begin{split}
-\log\left| \frac{\Omega(u,v_1) - \Omega(u,v_2)}{\Omega(u,v_1) - \overline{\Omega}(u,v_2)}\right| = &2 \log \left( \sqrt{(v_1-a)(b-v_2)} + \sqrt{(v_2-a)(b-v_1)} \right) - \\
&- \log|v_1-v_2| - \log(b-a).
\end{split}
\end{equation}
From (\ref{OmegaLog}) we know that 
$$-\log\left| \frac{\Omega(u,v_1) - \Omega(u,v_2)}{\Omega(u,v_1) - \overline{\Omega}(u,v_2)}\right| = \log\left| \frac{\sqrt{(v_2-a)(b-v_1)}  + \sqrt{(v_1-a)(b-v_2)}}{\sqrt{(v_2-a)(b-v_1)} -\sqrt{(v_1-a)(b-v_2)}}\right|.$$
In addition, one readily checks that $\log|v_1-v_2| +\log(b-a)$ is equal to
$$\log\left|\sqrt{(v_2-a)(b-v_1)}  + \sqrt{(v_1-a)(b-v_2)}\right| + \log\left|\sqrt{(v_2-a)(b-v_1)}  - \sqrt{(v_1-a)(b-v_2)}\right| .$$
The last two statements imply (\ref{GFF3}), which concludes the proof of (\ref{StartGFF}).\\

In the remainder of the section we establish (\ref{GFF2}). Fix $v_2 \in (a,b)$ and let $\varepsilon > 0$ be such that $(v_2- \varepsilon, v_2+ \varepsilon) \subset [a,b]$. For $\delta \in (0,\varepsilon)$ we define the contour $\Gamma_{\delta, \varepsilon}$ as follows. $\Gamma_{\delta, \varepsilon}$ starts from the point $b - \iu \delta$ and follows the circle centered at $b$ with radius $\delta$ counterclockwise until the point $b + \iu \delta$, afterwards it goes to the left along the segment connecting the points $b+ \iu \delta$ and $v_2 + \varepsilon + \iu \delta$; it follows the circle centered at $v_2+ \iu \delta$ and radius $\varepsilon$ counterclockwise until the point $v_2 - \varepsilon + \iu \delta$ and goes to the left along the segment connecting $v_2 - \varepsilon + \iu \delta$ and $a + \iu \delta$; it then follows the circle centered at $a$ with radius $\delta$ counterclockwise until the point $a - \iu \delta$ and then goes to the right along the segment connecting $a - \iu \delta$ and $v_2 - \varepsilon -\iu \delta$; finally, it follows the circle centered at $v_2 - \iu \delta$ and radius $\varepsilon$ counterclockwise until the point $v_2 - \varepsilon - \iu \delta$ and goes to the right along the segment connecting $v_2 - \varepsilon - \iu \delta$ and $b - \iu \delta$, see Figure \ref{S8F1}.
\begin{figure}[h]
\includegraphics[width=0.75\linewidth]{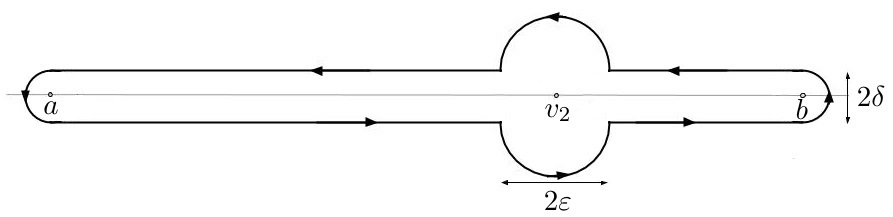}
\caption{The contour $\Gamma_{\delta, \varepsilon}$}
  \label{S8F1}
\end{figure}

By Cauchy's theorem we see that
\begin{equation}\label{GFFS1}
\oint_{\gamma} \frac{R_i(v_1)}{(v_1-v_2)} \cdot \frac{\sqrt{v_2 - a }\sqrt{b-v_2 }}{\sqrt{v_1 - a}\sqrt{v_1- b}}dv_1 = \oint_{\Gamma_{\delta, \varepsilon}} \frac{R_i(v_1)}{(v_1-v_2)} \cdot \frac{\sqrt{v_2 - a }\sqrt{b-v_2 }}
{\sqrt{v_1 - a}\sqrt{v_1- b}}dv_1 
\end{equation}
We next let $\delta$ go to $0^+$ and see that 
\begin{equation}\label{GFF4}
\begin{split}
&\lim_{\delta \rightarrow 0^+} \oint_{\Gamma_{\delta, \varepsilon}}\frac{ R_i(v_1)}{(v_1-v_2)}\cdot\frac{\sqrt{v_2 - a }\sqrt{b-v_2 }}{\sqrt{v_1 - a}\sqrt{v_1- b}}dv_1 = T_1(\varepsilon) + T_2(\varepsilon) + T_3(\varepsilon), \mbox{ where }\\
&T_1(\varepsilon) = 2 \iu  \int_a^{v_2 - \varepsilon}\frac{ R_i(v_1)}{(v_1-v_2)}\cdot \frac{\sqrt{v_2 - a }\sqrt{b-v_2 }}{\sqrt{v_1 - a}\sqrt{b-v_1}}dv_1, \\
& T_2(\varepsilon) = 2 \iu  \int_{v_2 + \varepsilon}^{b} \frac{R_i(v_1)}{(v_1-v_2)}\cdot\frac{\sqrt{v_2 - a }\sqrt{b-v_2 }}{\sqrt{v_1 - a}\sqrt{b- v_1}}dv_1,\\
&T_3(\varepsilon) = \int_{C^+_\varepsilon(v_2)}  \frac{R_i(v_1)}{(v_1-v_2)}\cdot\frac{\sqrt{v_2 - a }\sqrt{b-v_2 }}{\sqrt{v_1 - a}\sqrt{v_1- b}}dv_1 +  \int_{C^-_\varepsilon(v_2)}  \frac{R_i(v_1)}{(v_1-v_2)}\cdot\frac{\sqrt{v_2 - a }\sqrt{b-v_2 }}{\sqrt{v_1 - a}\sqrt{v_1- b}}dv_1 
\end{split}
\end{equation}
with $C^{+}_\varepsilon(v_2),C^{-}_\varepsilon(v_2)$ being positively oriented half-circles of radius $\varepsilon$ around $v_2$ in the upper and lower half-planes respectively. In deriving the above expression we used the Bounded convergence theorem and the fact that the square roots are purely imaginary and come with opposite sign when we approach $[a,b]$ from the upper and lower half-planes.

We next integrate by parts the integrals in $T_1(\varepsilon)$ and $T_2(\varepsilon)$ to get
\begin{equation}\label{GFF5}
\begin{split}
&T_1(\varepsilon) = -2 \iu  \int_a^{v_2 - \varepsilon}R_i'(u)G_v(v_1) dv_1 + 2\iu \cdot [R_i(v_2-\varepsilon)G_{v_2}(v_2-\varepsilon) - R_i(a)G_{v_2}(a)],\\
&T_2(\varepsilon) = -2 \iu  \int_{v_2 + \varepsilon}^{b}R_i'(v_1)G_{v_2}(v_1)dv_1 + 2\iu \cdot [R_i(b)G_{v_2}(b) -  R_i(v_2+\varepsilon)G_{v_2}(v_2+\varepsilon)], \mbox{ where }\\
&G_{v_2}(v_1) = -2 \log \left( \sqrt{(v_1-a)(b-v_2)} + \sqrt{(v_2-a)(b-v_1)} \right) + \log|v_1-v_2|.
\end{split}
\end{equation}
We observe that $G_{v_2}(a) =  -\log(b-a)  = G_{v_2}(b)$ and $R_i(v_2-\varepsilon)G_{v_2}(v_2-\varepsilon) - R_i(v_2+\varepsilon)G_{v_2}(v_2+\varepsilon) = O\left( \varepsilon\log \varepsilon^{-1} \right)$.
The latter statements together with the Dominated convergence theorem imply that
\begin{equation}\label{GFF6}
\lim_{\varepsilon \rightarrow 0^+} T_1(\varepsilon) + T_2(\varepsilon) = -2 \iu \int_a^{b} R_i'(v_1)G_{v_2}(v_1) dv_1 - 2\iu \log(b-a)\cdot [R(b) - R(a)].
\end{equation}

We next turn to $T_3(\varepsilon)$ and parametrize $C^{+}_\varepsilon(v)$ through $v_1 = v_2 + \varepsilon e^{\iu \theta}$ with $\theta \in (0, \pi)$ and $C^{-}_\varepsilon(v)$ through $v_1 = v_2 + \varepsilon e^{\iu \theta}$ with $\theta \in (-\pi, 0)$. This leads to
\begin{equation*}
T_3(\varepsilon) =  \iu \sqrt{v_2 - a }\sqrt{b-v_2 } \cdot \left[ \int_{0}^\pi \frac{R_i(v_2 + \varepsilon e^{\iu \theta} ) d\theta}{\sqrt{\varepsilon e^{\iu \theta} + v_2 - a}\sqrt{\varepsilon e^{\iu \theta} + v_2- b}} + \int_{-\pi}^0 \frac{R_i(v_2 + \varepsilon e^{\iu \theta} ) d\theta}{\sqrt{\varepsilon e^{\iu \theta} + v_2 - a}\sqrt{\varepsilon e^{\iu \theta} + v_2- b}} \right] \hspace{-1.5mm}.
\end{equation*}
We can let $\varepsilon$ converge to $0^+$ above, which by the Bounded convergence theorem implies
\begin{equation}\label{GFF7}
\lim_{\varepsilon \rightarrow 0^+} T_3(\varepsilon) =   \sqrt{v_2 - a }\sqrt{b-v_2 } \cdot \left[ \int_{0}^\pi \frac{R_i(v_2 ) d\theta}{\sqrt{ v_2 - a}\sqrt{b - v_2}} - \int_{-\pi}^0 \frac{R_i(v_2  ) d\theta}{\sqrt{ v_2 - a}\sqrt{b-v_2}} \right] = 0,
\end{equation}
where the sign change came from the fact that we are approaching the real line from the upper and lower half-planes in the two cases. Combining (\ref{GFFS1}, \ref{GFF4}, \ref{GFF6}, \ref{GFF7}) we conclude that
$$\oint_{\gamma} \frac{R_i(v_1)}{(v_1-v_2)} \cdot \frac{\sqrt{v_2 - a }\sqrt{b-v_2 }}{\sqrt{v_1 - a}\sqrt{v_1- b}}dv_1 =  -2 \iu \int_a^{b} R_i'(v_1)G_{v_2}(v_1) dv_1 - 2\iu \log(b-a)\cdot [R(b) - R(a)].$$
The latter is equivalent to (\ref{GFF2}) once we use that $R(b) - R(a) = \int_a^b R'(v_1)dv_1.$

\subsection{Conjectural $2$d fluctuations}\label{Section8.4} We begin our discussion by formulating a two-dimensional (conjectural) extension to Theorem \ref{TCLT}. 

\begin{conjecture}\label{GFFConj}Assume the same notation as in Theorem \ref{TCLT}. Then the centered random height function
$$\sqrt{\pi} \left( \mathcal{H}(q^{-t}, v) - \mathbb{E}_{\mathbb{P}_\varepsilon}\left[ \mathcal{H}(q^{-t}, v) \right] \right) $$
converges to the pullback of the Gaussian free field with Dirichlet boundary conditions on the upper half-plane $\mathbb{H}$ with respect to the map $\Omega$ in the following sense: For any set of polynomials $f_i \in \mathbb{R}[x]$, numbers $u_i \in (1, \lq^{-\lT})$ and sequences $t_i(\epsilon)$ such that $q^{-t_i(\epsilon)} = u_i + O(\epsilon)$ for $i = 1,\dots,m$ the joint distribution of
\begin{equation}\label{GFFConjeq1}
\int_{\mathbb{R}} \sqrt{\pi} \left(\mathcal{H}(q^{-t_i}, v) - \mathbb{E}_{\mathbb{P}_\varepsilon} \left[ \mathcal{H}(q^{-t_i}, v) \right] \right) f_i(v)dv, \hspace{5mm} i = 1,\dots,m,
\end{equation}
converges to the joint distribution of the similar averages 
$$\int_{a(u_i)}^{b(u_i)} \mathcal{F}(\Omega(u, y))f_i(y)dy,\hspace{5mm} i = 1,\dots,m$$
of the pullback of the GFF. In the above formula $a(u), b(u)$ are the $v$-coordinates of the two points where the vertical line through $u$ intersects the ellipse $\tilde{Q}(u,v) = 0$.

Equivalently, the variables in (\ref{GFFConjeq1}) converge jointly to a Gaussian vector $(X_1, \dots, X_m)$ with mean zero and covariance
\begin{equation}\label{conjeq1}
\begin{split}
&\mathbb{E}[X_i X_j] =  \int_{a(u_i)}^{b(u_i)} \int_{a(u_j)}^{b(u_j)} f_i(x) f_j(y) \left(  - \frac{1}{2\pi} \log \left| \frac{\Omega(u_i,x) - \Omega(u_j,y)}{\Omega(u_i,x) - \overline{\Omega}(u_j,y)}\right| \right) dx dy.
\end{split}
\end{equation}
\end{conjecture}
\begin{rem}
We emphasize that our methods only allow us to study the global fluctuations of the tiling model for a {\em single} vertical section, and in order to establish the above statement one needs to be able to study the {\em joint} distribution of the particles on {\em several} vertical slices. 
\end{rem}

In the remainder of this section we present some numerical evidence supporting Conjecture \ref{GFFConj}. The general strategy is to use the exact sampling algorithm detailed in \cite{BGR} to generate many samples of large $q$-Racah tilings and compare the empirical distribution of the samples with the conjectural one. The limiting distribution in Conjecture \ref{GFFConj} is uniquely specified as being Gaussian with covariance as in (\ref{conjeq1}). While showing Gaussianity numerically is difficult, we will try to match the covariance -- this in particular gives some confidence to our complex structure $\Omega$. 

To be more specific we will take $n$ sample tilings of a hexagon of size $a\times b\times c$ with parameters $q$ and $\kappa.$ We fix $m \geq 1$ and $m$ sections $t_1, \dots, t_m$ and polynomials $f_i(x) \in \mathbb{R}[x]$ for $i = 1,\dots,m$. Set 
$$Y_i = \int_{\mathbb{R}} \sqrt{\pi} \left(\mathcal{H}(q^{-t_i}, v) - \mathbb{E}_{\mathbb{P}_\varepsilon} \left[ \mathcal{H}(q^{-t_i}, v) \right] \right) f_i(v)dv, \hspace{5mm} i = 1,\dots,m.$$
It follows from the strong law of large numbers that if $(Y_1^k, \dots, Y_m^k)$ form an i.i.d. sequence of samples with law $(Y_1,\dots,Y_m)$ then a.s. we have
$$\lim_{ n \rightarrow \infty} \frac{1}{n} \sum_{ k = 1}^n Y^k_i Y^k_j = \mathbb{E}[ Y_i Y_j] \mbox{ for $1 \leq i,j \leq m$}.$$
On the other hand, Conjecure \ref{GFFConj} suggests that if $a,b,c$ are large and $q$ is appropriately close to $1$, then for $1 \leq i,j \leq m$ we have
$$\mathbb{E}[ Y_i Y_j] \approx \int_{a(u_i)}^{b(u_i)} \int_{a(u_j)}^{b(u_j)} f_i(x) f_j(y) \left(  - \frac{1}{2\pi} \log \left| \frac{\Omega(u_i,x) - \Omega(u_j,y)}{\Omega(u_i,x) - \overline{\Omega}(u_j,y)}\right| \right) dx dy,$$
where $u_i = q^{-t_i}$ for $i = 1, \dots,m$. Combining the last two statement we see that
\begin{equation}\label{Numerics}
 \frac{1}{n} \sum_{ k = 1}^n Y^k_i Y^k_j \approx  \int_{a(u_i)}^{b(u_i)} \int_{a(u_j)}^{b(u_j)} f_i(x) f_j(y) \left(  - \frac{1}{2\pi} \log \left| \frac{\Omega(u_i,x) - \Omega(u_j,y)}{\Omega(u_i,x) - \overline{\Omega}(u_j,y)}\right| \right) dx dy 
\end{equation}
with high probability whenever $n$ is large.  Denoting the LHS (\ref{Numerics}) by $Cov_{Emp}(i,j)$ and the right by $Cov_{GFF}(i,j)$ we form the ratios 
$$r_{ij}:= \frac{Cov_{Emp}(i,j)}{Cov_{GFF}(i,j)}.$$
 We want to show through our simulations that if we take $a,b,c$ large, $q$ close to $1$ and a large number of samples $n$ then $r_{ij}$ are all close to one.\\

We cosider two sets of parameters. The first is given by 
$$ q=0.995, \quad \kappa^2=0.005, \quad a=300, \quad b=500, \quad c=200.$$
The slices are $t_1=100, t_2=200, t_3=400, t_4=450$
and polynomials $f_i(x) = x^i$ for $i = 1,2,3,4$. Using the latter parameters we perform $n = 1000$ simulations to obtain $r_{ij}$ and summarize our result in the following table.
\begin{center}
    \begin{tabular}{| l | l | l | l | l | }
    \hline
     $1.05$ & $1.06$ & $0.85$ & $1.07$   \\ \hline
     $1.06$ & $1.00$ & $0.92$ & $0.9$ \\ \hline
      $0.85$  & $0.92$ & $0.94$ & $0.95$ \\ \hline
     $1.07$ & $0.9$ & $0.95$ & $0.97$  \\ 
    \hline
    \end{tabular}
\captionof{table}{The entry on the $i$th row and $j$th column is $r_{ij}$.  } \label{tableTag} 
\end{center}

Another set of parameters we take is 
$$ q=0.99, \quad \kappa^2=0.01, \quad a=500, \quad b=450, \quad c = 350$$
The slices are $t_1=100,$ $t_2=200,$ $t_3=400,$ $t_4=600$ and polynomials $f_i(x) = x^i$ for $i = 1,2,3,4$. Using the latter parameters we perform $n = 1000$ simulations to obtain $r_{ij}$ and summarize our result in the following table.
\begin{center}
    \begin{tabular}{| l | l | l | l | l | }
    \hline
     $1.02$ & $1.03$ & $1.03$ & $1.04$   \\ \hline
     $1.03$ & $1.04$ & $1.05$ & $1.05$ \\ \hline
      $1.03$  & $0.97$ & $1.05$ & $1.05$ \\ \hline
     $1.04$ & $1.05$ & $1.05$ & $1.06$  \\ 
    \hline
    \end{tabular}
\captionof{table}{The entry on the $i$th row and $j$th column is $r_{ij}$.  } \label{tableTag} 
\end{center}
As can be seen from the results, the empirical covariance nicely agrees with the limiting covariance, with error that is around $5$ percent. The above data is just a sample and one obtains similar results for different hexagonal sizes and choices of polynomials.

\section{Connection to log-gases on a quadratic lattice} \label{Section9}
In this section we explain how our model fits into the framework of a discrete log-gas on a quadratic lattice as in Section \ref{Section2}. The latter will allows us to deduce Theorems \ref{BT1} and \ref{QRCLT} as consequences of Theorems \ref{GLLN} and \ref{CLTfun} respectively.

\subsection{Asymptotics of the weight function}\label{Section9.1} We first consider with the weight function $w^{qR}(x)$ of the $q$-Racah ensemble defined by \eqref{eq:w_qR}.  We are interested in understanding the asymptotic behavior of $w^{qR}(x)$ when the parameters $\alpha, \beta, \gamma, \delta, q$ scale as in Definition \ref{ScaleQR}. In order to do this we will need the following technical lemma.

\begin{lem}\label{L1}
Let $a, b \in (0,1)$ and $\lq \in [a,b]$ be given. Suppose that $x_N, q_N$ for $N \in \mathbb{N}$ are sequences such that $\left| q_N -  \lq^{1/N } \right| \leq  AN^{-2}$ and $x_N \in [0, q_N]$, where $A$ is a positive constant. Then we have
\begin{equation}\label{A2}
(x_N;q_N)_{\infty}  = \exp\left( N\cdot \frac{\mbox{ \em Li}(x_N)}{\log(\lq)} + O\left(\log (N)\right)\right),
\end{equation}
where $\mbox{\em Li}(x) =  \sum_{k = 1}^\infty \frac{x^k}{k^2}$ is the dilogarithm function and the constant in the big $O$ notation depends on $a$, $b$ and $A$.
\end{lem}
\begin{proof}
Taking logarithm of $ \prod_{i = 1}^\infty ( 1 - x_N q_N^{i-1})$ and power expanding $\log(1 - y)$ for $0 < y < 1$ gives
$$\log \left[(x_N;q_N)_{\infty}\right] = -\sum_{ r= 1}^\infty \sum_{k = 1}^\infty \frac{x_N^k }{k} \cdot q_N^{k(r-1)}.$$
We change the order of the sums and use the geometric series formula to get
$$\log \left[(x_N;q_N)_{\infty}\right]  = - \sum_{k = 1}^\infty \frac{x_N^k}{k} \cdot \frac{1}{1- q_N^k} =- \frac{1}{1 - q_N} \sum_{k = 1}^\infty \frac{x_N^k}{k} \cdot \frac{1 - q_N}{1- q_N^k} = A_N + B_N + C_N,$$
where 
$$A_N = -\frac{1}{1 - q_N} \sum_{k = 1}^\infty \frac{x_N^k}{k} \left[ \frac{1 - q_N}{1- q_N^k} - \frac{1}{k} \right], \hspace{2mm} B_N =- \frac{1}{1 - q_N} \sum_{k = 1}^\infty \frac{x_N^k}{k^2} -  N \frac{\mbox{  Li}(x_N)}{\log(\lq)}, \hspace{2mm} C_N = N\frac{\mbox{  Li}(x_N)}{\log(\lq)}.$$
What we need to show is that $A_N$ and $B_N$ are both $O\left( \log (N)\right)$.

Notice that 
$$|B_N| = \mbox{  Li}(x_N)\left| -\frac{1}{1 - q_N} - \frac{N}{{-\log(\lq)}}\right| \leq \mbox{  Li}(1) + O(1),$$
where we used that $x_N \in [0,1]$ and $1 - q_N  = -\log(\lq)/N + O(N^{-2})$. This proves that $B_N = O(1)$ and we focus on $A_N$ for the remainder.

Combining $ \frac{1 - q_N}{1- q_N^k} - \frac{1}{k}  \geq 0$ with $x_N \in [0, q_N]$ we conclude
$$  -\frac{1}{1 - q_N} \sum_{k = 1}^\infty \frac{q_N^k}{k} \left[ \frac{1 - q_N}{1- q_N^k} - \frac{1}{k} \right] \geq A_N \geq 0.$$
Since $1 - q_N = -\log(\lq)/N + O(N^{-2})$, we see that what remains to be shown is that 
\begin{equation}\label{A3}
\sum_{k = 1}^\infty \frac{q_N^k}{k} \left[ \frac{1 - q_N}{1- q_N^k} - \frac{1}{k} \right] = O\left( \frac{\log (N)}{N}\right).
\end{equation}

Suppose that $-\frac{1}{4 \log(\lq)} \geq \varepsilon_0 > 0$ is sufficiently small so that when $\varepsilon \in \left[0, - \log (\lq)\varepsilon_0\right]$ we have $1 - e^{-\varepsilon} \geq \varepsilon - 2\varepsilon^2 \geq \varepsilon/2$. Using the latter together with the inequality $ \dfrac{1 - q_N}{1- q_N^k} - \dfrac{1}{k}  \geq 0$ we see that
$$\sum_{k \leq \varepsilon_0 N} \frac{q_N^k}{k} \left[\frac{1 - q_N}{-(k/N) \log(\lq) - 2\log(\lq)^2 (k/N)^2} -  \frac{1}{k} \right]  \geq \sum_{k \leq \varepsilon_0 N} \frac{q_N^k}{k} \left[\frac{1 - q_N}{1- q_N^k} -  \frac{1}{k} \right] \geq 0.$$
As $1 - q_N= -\log(\lq)/N + O(N^{-2}) $ the above statement implies that for sufficiently large $C$ we have
$$\frac{C}{N} \sum_{k \leq \varepsilon_0 N} \frac{q_N^k}{k} + \sum_{k \leq  \varepsilon_0 N} \frac{q_N^k}{k} \left[\frac{-  \log(\lq)  }{-k \log(\lq) - 2\log(\lq)^2 k^2N^{-1}} -  \frac{1}{k} \right]    \geq \sum_{k \leq \varepsilon_0 N} \frac{q_N^k}{k} \left[\frac{1 - q_N}{1- q_N^k} -  \frac{1}{k} \right] \geq 0.$$
Notice that
$$\frac{-\log(\lq)  }{k \log(\lq) +2\log(\lq)^2 k^2N^{-1}} -  \frac{1}{k}  = \frac{1  }{k + 2\log(\lq) k^2N^{-1}} -  \frac{1}{k} = -\frac{2\log(\lq)  N^{-1}}{1 + 2\log(\lq) kN^{-1}} \leq  -4\log(\lq)  N^{-1},$$
where the last inequality holds since $kN^{-1} \leq \varepsilon_0 \leq -\frac{1}{4 \log(\q)} $. The latter estimates show that for some (possibly different than before) constant $C > 0$ we have
$$\frac{C}{N} \sum_{k \leq \varepsilon_0 N} \frac{q_N^k}{k}\geq \sum_{k \leq \varepsilon_0 N} \frac{q_N^k}{k} \left[\frac{1 - q_N}{1- q_N^k} -  \frac{1}{k} \right] \geq 0.$$
Since $ \sum\limits_{k \leq \varepsilon_0 N} \frac{q_N^k}{k} \leq - \log( 1 - q_N) = O(\log(N))$ we conclude that
\begin{equation}\label{A4}
\sum_{k \leq \varepsilon_0 N}\frac{q_N^k}{k}\left[ \frac{1 - q_N}{1- q_N^k} - \frac{1}{k} \right] = O\left( \frac{\log (N)}{N}\right).
\end{equation}

We next have that 
$$\frac{1 - q_N}{1 - \lq^{\varepsilon_0}} \sum_{ k \geq \varepsilon_0 N} \frac{q^k_N}{k}\geq \sum_{k \geq \varepsilon_0 N} \frac{q_N^k}{k}\cdot \frac{1 - q_N}{1- q_N^k}  \geq \sum_{k \geq \varepsilon_0 N} \frac{q_N^k}{k} \left[ \frac{1 - q_N}{1- q_N^k} - \frac{1}{k} \right] \geq 0.$$
Since $1 - q_N = O(N^{-1})$ and $ \sum_{k \geq \varepsilon_0 N} \frac{q_N^k}{k} \leq - \log( 1 - q_N) = O(\log(N))$ we conclude that 
\begin{equation}\label{A5}
 \sum_{k \geq \varepsilon_0 N}^\infty \frac{q_N^k}{k} \left[ \frac{1 - q_N}{1- q_N^k} - \frac{1}{k} \right] = O\left( \frac{\log (N)}{N}\right).
\end{equation}
Combining (\ref{A4}) and (\ref{A5}) we conclude (\ref{A3}), which proves the lemma.
\end{proof}

In addition, we require the following alternative formula for the weight $w^{qR}(x)$.
\begin{lem}
Suppose that we have parameters $\alpha, \beta, \gamma, \delta, q$ and $M$ as in Definition \ref{ParSetQR}. Then we have
\begin{equation}\label{wQR}
\begin{split}
&w^{qR}(x)  =  \frac{( \beta \delta q, \gamma \delta q; q)_\infty }{(q, \alpha^{-1}\gamma \delta q, \beta^{-1} \gamma q, \delta q, \alpha^{-1}, \gamma^{-1}; q)_\infty} \tilde{w}^{qR}(x), \mbox{ where }  \\
&\tilde{w}^{qR}(x) = (\gamma/\beta)^x q^{x^2}\frac{1 - \gamma \delta q^{2x +1}}{1 - \gamma \delta q} \frac{(q^{x+1}, \alpha^{-1}\gamma \delta q^{x+1}, \beta^{-1} \gamma q^{x+1}, \delta q^{x+1}, \alpha^{-1}q^{-x}, \gamma^{-1}q^{-x}; q)_\infty}{( \beta \delta q^{x+1}, \gamma \delta q^{x+1}; q)_\infty } .
\end{split}
\end{equation}
\end{lem}
\begin{proof}
We recall the definition of $w^{qR}(x)$ from Definition \ref{qrw} for the reader's convenience.
$$w^{qR}(x)=\frac{(\a q,\b\delta q,\gamma q,\gamma \delta q;q)_x}{(q,\a^{-1}\gamma \delta q, \b^{-1}\gamma q, \delta q;q )_x} \cdot \frac{(1-\gamma\delta q^{2x+1})}{(\a\b q)^{x}(1-\gamma\delta q)}.$$
Observe that 
$$(\a q;q)_x = \prod_{i = 1}^x (1 - \a q^i) = q^{x(x+1)/2} \a^x  (-1)^x \cdot \prod_{i = 1}^x (1 - \a^{-1} q^{-i}) = q^{x(x+1)/2} \a^x (-1)^x \cdot (\a^{-1} q^{-x}; q)_x.$$
Similarly, we have $(\gamma q;q)_x = q^{x(x+1)/2} \gamma^x (-1)^x \cdot (\gamma^{-1} q^{-x}; q)_x$. Substituting the latter identities and performing a bit of cancellation we arrive at
$$w^{qR}(x)=\frac{(\a^{-1} q^{-x},\b\delta q,\gamma^{-1} q^{-x},\gamma \delta q;q)_x}{(q,\a^{-1}\gamma \delta q, \b^{-1}\gamma q, \delta q;q )_x}\frac{(1-\gamma\delta q^{2x+1})}{(1-\gamma\delta q)} \cdot (\gamma/\beta)^x q^{x^2}.$$

Observe that for $a\in [0,1)$ we have  $(a;q)_x = \frac{(a;q)_\infty}{(aq^x; q)_\infty}.$  Substituting the latter identity in the above expression we see that $w^{qR}(x)$ equals
$$\frac{(\a^{-1} q^{-x},\b\delta q,\gamma^{-1} q^{-x},\gamma \delta q;q)_\infty}{(q,\a^{-1}\gamma \delta q, \b^{-1}\gamma q, \delta q;q )_\infty}\frac{(q^{x+1},\a^{-1}\gamma \delta q^{x+1}, \b^{-1}\gamma q^{x+1}, \delta q^{x+1};q )_\infty}{(\a^{-1} ,\b\delta q^{x+1},\gamma^{-1},\gamma \delta q^{x+1};q)_\infty}\frac{(1-\gamma\delta q^{2x+1})}{(1-\gamma\delta q)} \cdot (\gamma/\beta)^x q^{x^2 }.$$
From here (\ref{wQR}) is immediate.
\end{proof}

The main statement of this section is the following.
\begin{lem}\label{LM1}
Assume that we have the same notation as in Definition \ref{ScaleQR} and let $\tilde{w}_N^{qR}(x)$ be as in (\ref{wQR}) with parameters $\alpha_N, \beta_N, \gamma_N, \delta_N, q_N$.
Then, for $x \in \{ 0,\dots, M_N\}$ we have the following asymptotic expansion of $\tilde{w}_N^{qR}(x)$
\begin{equation}\label{wQRT}
\begin{split}
 &\tilde{w}_N^{qR}(x) = \exp \left( - N V\left( \frac{x}{N}; \la, \lb, \lc, \ld\right)  + O\left(\log (N)\right)\right), \mbox{ where } \\
& V(s;   \la, \lb, \lc, \ld) = \frac{1}{- \log \lq} \Big{[} \mbox{{\em  Li}}(\lq^{s}) + \mbox{{\em  Li}}( \la^{-1}\lc  \ld \lq^{s}) + \mbox{{\em  Li}}( \lb^{-1}\lc \lq^{s}) + \mbox{{\em  Li}}( \ld \lq^{s} ) + \\
&+   \mbox{{\em  Li}}( \la^{-1} \lq^{-s} ) + \mbox{{\em  Li}}( \lc^{-1}\lq^{-s}) - \mbox{{\em  Li}}( \lb \ld \lq^{s}) - \mbox{{\em  Li}} (\lc \ld \lq^{s})  + \log(\lq)^2 s^2 + s \log(\lq)\log(\lc/\lb)\Big{]} 
\end{split}
\end{equation}
and the constant in the big $O$ notation depends on the parameters $A$ and $\la, \lb, \lc, \ld, \lq, \lM$ and is uniform as the latter vary over compact subsets of ${\tt P}$ (recall that $A$ and ${\tt P}$ were given in Definition \ref{ScaleQR}).
\end{lem}
\begin{proof}
Using Lemma \ref{L1} we have that
$$\tilde{w}^{qR}(x) = (\gamma/\beta)^x q^{x^2}\frac{1 - \gamma \delta q^{2x +1}}{1 - \gamma\delta q} \cdot \exp{\Big (} \frac{N}{\log \q} \Big{[} \mbox{Li}(q^{x+1}) + \mbox{Li}( \alpha^{-1}\gamma \delta q^{x+1}) + \mbox{Li}( \beta^{-1}\gamma q^{x+1}) + $$
$$ + \mbox{Li}( \delta q^{x+1} ) + \mbox{Li}( \alpha^{-1}q^{-x} ) + \mbox{Li}( \gamma^{-1}q^{-x}) - \mbox{Li}( \beta \delta q^{x+1}) - \mbox{Li}(\gamma \delta q^{x+1})\Big{]}  + O(\log(N))\Big{)}, $$
where for brevity we suppressed the dependence of the parameters on $N$. Since ${\tt c}/{\tt b} = \gamma/\beta + O(N^{-1})$, $q = {\tt q}^{1/N} + O(N^{-2})$ and $\gamma\delta q$ is bounded away from $1$ we see that 
$$(\gamma/ \beta)^x q^{x^2}\frac{1 - \gamma \delta q^{2x +1}}{1 - \gamma\delta q}  = \exp \left(  N \cdot  \frac{x}{N} \cdot \log(\lc/ \lb) + N \cdot \frac{x^2}{N^2} \cdot \log( \lq) + O(1)  \right).$$
This handles the first factor in $\tilde{w}^{qR}(x)$ and we only need to match the dilogarithms with (\ref{wQRT}).

We claim that if $x, y \in [0,1]$ and $C > 0$ are such that $|x - y| \leq CN^{-1}$ then 
\begin{equation}\label{A211}
\left| \mbox{Li}(x)  - \mbox{Li}(y) \right| \leq C \cdot  \frac{\log(N) + 1}{N} + \frac{1}{N}.
\end{equation}
It is clear that applying (\ref{A211}) to each of the dilogarithms in $\tilde{w}^{qR}(x)$ we can match the corresponding ones in (\ref{wQRT}) upto an error of order $\log (N)$. Thus, to prove the lemma it suffices to show (\ref{A211}).

Without loss of generality suppose that $ x \leq y$ and set $\varepsilon = y - x$. Then we have 
$$0 \leq \mbox{Li}(y)  - \mbox{Li}(x) = \sum_{k = 1}^\infty \frac{y^k -(y - \varepsilon)^k }{k^2} \leq \sum_{k = 1}^\infty \frac{1 - (1 - \varepsilon)^k}{k^2} ,$$
where we used that the functions $y^k - (y - \varepsilon)^k$ are increasing on $[\varepsilon, 1]$ for $k \geq 1$. In addition,
\begin{equation*}
\begin{split}
\sum_{k = 1}^N \frac{1 - (1 - \varepsilon)^k}{k^2}\leq \sum_{k = 1}^N \frac{k \varepsilon }{k^2} \leq C \cdot  \frac{\log(N) + 1}{N} \mbox{ and }\sum_{k = N + 1}^\infty \frac{1 - (1 - \varepsilon)^k}{k^2} \leq \sum_{k = N + 1}^\infty \frac{1}{k^2} \leq \int_{N}^\infty \frac{dz}{z^2} = \frac{1}{N}.
\end{split}
\end{equation*}
Combining the above inequalities we conclude (\ref{A211}) and hence the lemma.
\end{proof}

\subsection{Proof of Theorems \ref{BT1} and \ref{QRCLT} }\label{Section9.2} 
Our first task is to verify that $\mathbb{P}_N(\lambda_1,\dots,\lambda_N) $  from Definition \ref{ScaleQR} satisfy Assumptions 1-4 and 6-7 in Section \ref{Section2.2}.
In view of (\ref{wQR}) we have the following alternative representation  for $\mathbb{P}_N(\lambda_1,\dots,\lambda_N) $
\begin{equation}\label{S41}
\mathbb{P}_N(\lambda_1,\dots,\lambda_N) = \frac{{\bf 1} _{(0 \leq \lambda_1 < \lambda_2 < \cdots < \lambda_N \leq M_N)}}{\tilde{Z}(N,M_N,\alpha_N,\beta_N,\gamma_N,\delta_N)} \prod_{1 \leq i < j \leq N} \left(\sigma_N(q_N^{-\lambda_i}) - \sigma_N(q_N^{-\lambda_j})\right)^2\prod_{i = 1}^N \tilde{w}_N^{qR}(\lambda_i),
\end{equation}
where $\tilde{Z}$ is a new normalization constant, $\sigma_N(z) = z+ u_N z^{-1}$ with $u_N = \gamma_N \delta_N q_N$ and $\tilde{w}_N^{qR}$ is as in (\ref{wQR}) for the parameters $\alpha_N, \beta_N, \gamma_N, \delta_N, q_N$ and $M_N$. 

If we set $\ell_i = q_N^{-\lambda_i} + u_N q_N^{\lambda_i}$ for $i = 1, \dots, N$ then we see that the induced law on particles $\ell_i$ for $i = 1, \dots,N$ from (\ref{S41}) agrees with (\ref{PDef}) for $\theta = 1$. Specifically, we are in the single-cut case with $a_1(N) = 0 $, $b_1(N) = M_N - N +2$ and setting $w (\ell_i;N) := \tilde{w}_N^{qR}(\lambda_i)$ we have
\begin{equation}\label{S411}
\mathbb{P}_N(\ell_1,\dots,\ell_N) =Z_N^{-1} \cdot {\bf 1}_{\{(\ell_1, \dots, \ell_N) \in \mathfrak{X}^1_N\}} \cdot  \prod_{1 \leq i < j \leq N} \left(\ell_i - \ell_j\right)^2\prod_{i = 1}^N w(\ell_i;N),
\end{equation}

It is clear that Assumptions 1 and 3 in Section \ref{Section2.2} are satisfied in this case. In addition, in view of Lemma \ref{LM1}, we know that Assumption 2 holds for the function $V(x) =  V(\sigma_{\lq}^{-1}(x); \la, \lb, \lc, \ld)$, where $V(\cdot;   \la, \lb, \lc, \ld)$ is as in (\ref{wQRT}) and $\sigma_{\lq}(x) = \lq^{-x} + \lu \lq^{x}$ with $\lu = \lc \ld$. 

Let $ \Phi^+_N, \Phi^-_N$  be as in  Definition \ref{Polys} with parameters $a= \alpha_N, b = \beta_N, c = \gamma_N, d = \delta_N$ and $q = q_N^N$. Observe that  $\Phi^{\pm}_N$ satisfy Assumptions 4 and 6 in Section \ref{Section2.2}. In particular, we have that $\Phi_N^{\pm}$ converge to $\Phi^{\pm}_\infty$, where the latter are as in Definition \ref{Polys} for the parameters $a = \la$, $b =\lb$, $c = \lc$, $d = \ld$ and $q = \lq^{\lN}$. With the same choice of parameters we also define $R_\infty$ and $Q_\infty$ as in that definition and (\ref{QSquareRoot}). Finally, one checks that Assumption 7 holds from the definition of $V$ and $\Phi^{\pm}_\infty$.

Since Assumptions 1-4 and 6 hold we can apply Theorem \ref{NekGen} and obtain that the function $\tilde R_N$ on the right side of (\ref{REQV2}) is a degree $4$ polynomial. We isolate the following asymptotic statement about $\tilde{R}_N$, which will be used here and whose proof is the focus of Section \ref{Section10}.
\begin{fact}\label{fact_7}(see (\ref{proofFact7})) If $R_\infty$ and $\tilde{R}_N$ are as above then for all $z \in \mathbb{C}$ 
\begin{equation}\label{RFE2}
\lim_{N \rightarrow \infty} R_\infty(z) - \tilde{R}_N(z) = 0.
\end{equation}
\end{fact}

\begin{proof} (Theorem \ref{BT1})
As discussed earlier we know that $\mathbb{P}_N(\ell_1,\dots,\ell_N)  $ in (\ref{S411}) satisfies Assumptions 1-4 and 6-7 in Section \ref{Section2.2}. We conclude from Theorem \ref{GLLN} that the empirical measures 
$$\rho_N =  \frac{1}{N} \sum_{i = 1}^N \delta \left( \sigma_N(q_N^{-\lambda_i}) \right)$$
converge to a limiting measure $\rho$, which by Lemma \ref{Lsupp} has density
\begin{equation}\label{rhoForm}
 \rho(y_0 + \lu y_0^{-1}) =  \frac{1}{ \log (\lq)\pi (y_0 - \lu y_0^{-1}) } \cdot \mbox{arccos} \left( \frac{R_\rho(y_0)}{2 \sqrt{\Phi_\infty^-(y_0)  \Phi_\infty^+(y_0)}}\right)\mbox{ for $y_0 \in [1, \lc]$}.
\end{equation}
Combining (\ref{RLimNek}) and (\ref{RFE2}) we conclude that $R_\rho =  R_\infty$. The latter implies that the measures $\mu_N$ in Theorem \ref{BT1} satisfy the conditions in that theorem for the measure $\mu := \rho \circ \sigma_{\lq}^{-1}$, which in view of (\ref{rhoForm}) agrees with Definition \ref{Polys}.
\end{proof}

\begin{proof} (Theorem \ref{QRCLT}) From the proof of Theorem \ref{BT1} we know that $R_\rho = R_\infty$, and so $Q_\rho^2 = R_\rho^{2} -4 \Phi^{+}_\infty \Phi^-_\infty = Q_\infty^2$. The formula for $Q_\infty$ in (\ref{QSquareRoot}) implies that $Q_\rho$ satisfies Assumption 5 for $r_1 = a_-$, $s_1 = a_+$ and $H(z) = z (z^2 - \lc\ld)(\la \lb \lq^{\lN} - \lq^{-\lN})$. Overall, the measures (\ref{S411}) satisfy Assumptions 1-6 in Section \ref{Section2.2} and so Theorem \ref{QRCLT} follows from Theorem \ref{CLTfun} and Remark \ref{meq1}.
\end{proof}

\section{Proof of Fact \ref{fact_7}}\label{Section10}
The goal of this section is to prove Fact \ref{fact_7}, which is the missing ingredient necessary to complete the proofs of Theorems \ref{BT1} and \ref{QRCLT}. We summarize some basic facts about discrete Riemann-Hilbert problems and $q$-Racah orthogonal polynomials in Section \ref{Section10.1}. In Section \ref{Section10.2} we introduce a matrix-valued function $A_N(z)$ and derive some of its properties. Section \ref{Section10.3} contains some asymptotic results about $A_N(z)$, which suffice to show Fact \ref{fact_7}.

\subsection{Discrete Riemann-Hilbert problems and orthogonal polynomials}\label{Section10.1}

\subsubsection{Discrete Riemann-Hilbert problems}\label{DRHP}
In this section we relate solutions of discrete Riemann-Hilbert problems (DRHP) for jump matrices of a special type to orthogonal polynomials. Our exposition closely follows that in \cite[Section 2]{BB}, which in turn dates back to \cite{B00, B002}. 

Let $\mathfrak X$ be a finite subset of $\mathbb C$ such that $\mbox{card}(\mathfrak X)=M+1<\infty$ and $w: \mathfrak{X} \rightarrow \mbox{Mat}(2, \mathbb{C})$ be any function. We say that an analytic function
$$m\colon \mathbb C \setminus \mathfrak X \rightarrow \mbox{Mat}(2, \mathbb C)$$ 
solves the DRHP$(\mathfrak X, w)$ if $m$ has simple poles at the points of $\mathfrak X$ and its residues at these points are given by the {\em jump} (or {\em residue}) condition 
\begin{equation}\label{jumpCond}
\residue\limits_{\zeta=x}\text{ }m(\zeta)=\lim\limits_{\zeta\rightarrow x}^{ }m(\zeta)w(\zeta),\text{  } x\in \mathfrak X.
\end{equation} 
We will assume that the matrix $w(x)$ depends  on a function $\omega: \mathfrak{X} \rightarrow \mathbb{C}$ and has the form
\begin{equation} \label{eq:jump}
 w(x)=\left [\begin{array}{cc} 0 & \omega(x)\\0 & 0
\end{array}\right],
\end{equation}

 Recall that a collection $\{ P_n(\zeta) \}_{n = 0}^N$ of complex polynomials is called the collection of {\em orthogonal polynomials associated to the weight function} $\omega$ if 
\begin{itemize}
\item $P_n$ is a polynomial of degree $n$ for all $n = 1, \dots, M$ and $P_0 \equiv const$;
\item if $m \neq n$ then  $\sum_{x \in \mathfrak{X}} P_m(x) P_n(x) \omega(x) = 0$.
\end{itemize}
We will always take $P_n$ to be monic, i.e. $P_n(x) = x^n + $ lower terms.

We consider the following inner product on the space $\mathbb C[\zeta]$ of all complex polynomials:
$$\left(f(\zeta), g(\zeta)\right )_{\omega}=\sum\limits_{x\in\mathfrak X}^{ }f(x)g(x)\omega(x).$$
It is clear that there exists a collection of orthogonal polynomials $\{P_n(\zeta)\}_{n = 0}^M$ associated to $\omega$ such that $ (P_n, P_n)_\omega \neq 0$ for all $n = 0, \dots, M$ if and only if the restriction of $(\cdot, \cdot)_\omega$ to the space $\mathbb{C}[\zeta]^{\leq d}$ of polynomials of degree at most $d$ is non-degenerate for all $d = 0, \dots, M$. If this condition holds we say that the function $\omega$ is {\em nondegenerate}, and then it is clear that the collection $\{P_n(\zeta)\}_{n = 0}^M$ is unique. For convenience we isolate the following notation
\begin{equation}\label{DPEN}
c_n := \left( P_n, P_n\right)_{\omega}, \hspace{5mm} H_n (\zeta) : = \sum\limits_{x\in\mathfrak X}\frac{P_{n}(x)\omega(x)}{\zeta - x}, \hspace{5mm} n = 0, \dots, M.
\end{equation}

The connection between the orthogonal polynomials  $\{P_n(\zeta)\}_{n = 0}^M$ and solutions to  DRHP$(\mathfrak X, w)$ is detailed in the following statement.
\begin{theorem} \cite[Lemma 2.1 and Theorem 2.4]{BB}\label{RHPsol} Let $\mathfrak X$  be a finite subset of $\mathbb C$ such that $\mbox{card}(\mathfrak X)=M+1<\infty$ and $w$ be as in (\ref{eq:jump}) with $\omega \colon \mathfrak X\rightarrow \mathbb C$ a nondegenerate weight function. Then for any $k=1,2,\dots,M$ the DRHP$(\mathfrak X, w)$ has a unique solution $m_{\mathfrak X}(\zeta),$ satisfying an assymptotic condition
\begin{equation}\label{DRHP1}
m_{\mathfrak X}(\zeta)\cdot \left [\begin{array}{cc} \zeta^{-k} & 0\\0 & \zeta^{k}
\end{array}\right] =I+O\left( \zeta^{-1} \right)\text{  as  } \zeta\rightarrow \infty,
\end{equation}
where $I$ is the identity matrix. This solution is explicitly given by
$$ m_{\mathfrak X}(\zeta)= \left [\begin{array}{cc} P_k(\zeta) & H_k(\zeta) \\ c_{k-1}^{-1} P_{k-1}(\zeta)   & c_{k-1}^{-1}H_{k-1}(\zeta) 
\end{array}\right], \mbox{ with $c_{n}, H_n$ as in (\ref{DPEN}) } $$  
and satisfies $\mbox{{\em det}} m_{\mathfrak X}(\zeta) \equiv 1$.
\end{theorem}

\subsubsection{$q$-Racah polynomials}
In this section we recall and establish some basic properties of the $q$-Racah orthogonal polynomials, cf. \cite[Section 3.2]{KS2}. Recall from Definition \ref{qrw} that the $q$-Racah weight function is defined on $\mathfrak{X} = \{q^{-x}+\gamma \delta q^{x+1}: x=0,\dots, M\}$ as
\begin{equation}\label{omegaQR}
\omega^{qR}(q^{-x}+\gamma \delta q^{x+1}) =\frac{(\alpha q,\beta\delta q,\gamma q,\gamma \delta q;q)_x}{(q,\alpha^{-1}\gamma \delta q, \beta^{-1}\gamma q, \delta q;q )_x}\frac{(1-\gamma\delta q^{2x+1})}{(\alpha \beta q)^{x}(1-\gamma\delta q)}.
\end{equation}
We assume the parameters are as in Definition \ref{ParSetQR}. It is well known that $\omega^{qR}$ is a nondegenerate weight function and the orthogonal polynomials $\{ P_n(\zeta) \}_{n = 0}^M$  associated to it are the $q$-Racah orthogonal polynomials. Explicitly, they are given by
\begin{equation}\label{qRacP}
\begin{split}
P_n(q^{-x} + \gamma \delta q^{x+1} ) = {}_4 \tilde{\phi}_3\left( \begin{matrix} q^{-n}, \alpha \beta q^{n+1} , q^{-x},  \gamma\delta q^{x+1} \\ \alpha q, \beta \delta q
, \gamma q\end{matrix} \Big\vert q; q\right), \mbox{ where }\\
{}_4 \tilde{\phi}_3\left( \begin{matrix} a_1, a_2,a_3 ,  a_4 \\ b_1, b_2, b_3 \end{matrix} \Big\vert q; z\right) = \sum_{n = 0}^M \frac{(a_1, a_2, a_3,a_4; q)_k}{(b_1, b_2, b_3; q)_k} \frac{z^k}{(q;q)_k}.
\end{split}
\end{equation}
The $q$-Racah polynomials $P_n$ satisfy the following orthogonality relation. 
\begin{equation}\label{QRorthog}
\begin{split}
&\sum_{ x = 0}^M \omega^{qR}(q^{-x} + \gamma \delta q^{x+1} ) P_m(q^{-x} + \gamma \delta q^{x+1} ) P_n(q^{-x} + \gamma \delta q^{x+1} ) = c_n \cdot \delta_{mn}, \mbox{ where }\\
&c_n = \frac{(\gamma \delta q^2, \alpha^{-1} \beta^{-1} \gamma, \alpha^{-1} \delta, \beta^{-1}; q)_\infty}{(\alpha^{-1} \gamma \delta q, \beta^{-1}\gamma q, \delta q, \alpha^{-1} \beta^{-1} q^{-1}; q)_\infty } \frac{(1 - \alpha \beta q) (\gamma \delta q)^n}{ ( 1 - \alpha \beta q^{2n+1})} \frac{(q, \beta q, \alpha \delta^{-1} q, \alpha \beta \gamma^{-1} q; q)_n }{(\alpha \beta q, \alpha q, \beta \delta q, \gamma q;q)_n}.
\end{split}
\end{equation}
In addition, setting $y(x) = P_n(q^{-x} + \gamma \delta q^{x+1} ) $ we have the following $q$-Difference equations
\begin{equation}\label{QRDiffEq}
\begin{split}
&q^{-n}(1 - q^n) (1 - \alpha \beta q^{n+1}) y(x) = \hat B(x) y(x+1) - [\hat B(x) + \hat D(x)] y(x) + \hat D(x) y(x-1) \mbox{ where } \\
&\hat B(x) = \frac{(1 - \alpha q^{x+1})( 1 - \beta \delta q^{x+1})( 1 - \gamma q^{x+1}) (1 - \gamma \delta q^{x+1}) }{(1 - \gamma \delta q^{2x+1}) (1 - \gamma \delta q^{2x+2})},  \mbox{ and }\\ 
&\hat D(x) = \frac{q(1 - q^x)(1 - \delta q^x) (\beta - \gamma q^x)(\alpha - \gamma \delta q^x)}{( 1 -  \gamma \delta q^{2x}) ( 1  - \gamma \delta q ^{2x+1})}.
\end{split}
\end{equation}

We next prove a couple of easy facts about the $q$-Racah polynomials. For convenience we set $\sigma(z) = z+ uz^{-1}$ with $u = \gamma \delta q$.
\begin{proposition}\label{rootsInt} Each $P_n(\zeta)$ with $M  \geq n \geq 0$ has $n$ roots in the interval $( 1 + u ,  q^{-M} + u q^M)$.  
\end{proposition}
\begin{proof}
Let $m$ be the number of roots of $P_n(\zeta)$ in the interval $[ 1 + u ,  q^{-M} + u q^M]$ counted with multiplicities. Since $\mbox{deg}(P_n) = n$, we know that $m \leq n$. If $m < n$ let $x_1, \dots, x_m$ the an enumeration of these roots in some order. By (\ref{QRorthog}) we know that 
\begin{equation}\label{orthQR}
 \sum_{x = 0}^M P_n(q^{-x} + u q^x) \omega^{qR}(q^{-x} + u q^x) \cdot \prod_{i = 1}^m (q^{-x} + u q^x - x_i)  = 0.
\end{equation}
Note that the polynomial $P_n(\zeta)\cdot \prod_{i = 1}^m (\zeta - x_i)$ does not change its sign on $[ 1 + u ,  q^{-M} + u q^M]$ and so all of the above summands must be zero. But then $P_n(q^{-x} + u q^x)  = 0$ for $x = 0, \dots, M$ and so $ M + 1 \leq m < n$, contradicting $M \geq n$. We conclude that $m = n$.

Let $x_1, \dots, x_m$ be the roots of $P_n(\zeta)$  that are not equal to $1 + u$. By our work above we know that $x_i \in ( 1 + u ,  q^{-M} + u q^M]$ for $i = 1, \dots ,m$. If again we suppose that $m < n$ then $P_n(\zeta)\cdot \prod_{i = 1}^m (\zeta - x_i)$ does not change its sign on $[ 1 + u ,  q^{-M} + u q^M]$ and so all the terms in the analogous sum (\ref{orthQR}) must be zero, implying $P_n(q^{-x} + u q^x)  = 0$ for $x = 0, \dots, M$. We conclude that $P_n$  has roots at $q^{-x} + u q^x$ for $x = 0, \dots, M$. But this is impossible, since $M \geq n$ leading to too many roots of $P_n$. We reach a contradiction, which arose from our assumption that $1 + u$ is a root of $P_n(\zeta)$. A similar argument shows $q^{-M} + u q^M$ is also not a root of $P_n(\zeta)$. 
\end{proof}

\begin{lemma}\label{commonRoots} Let $M \geq n \geq 1$ be given. If $t > 1$ is such that $P_n(\sigma(t)) = 0$ then $P_n(\sigma(q^{-1}t)) \neq 0$.
\end{lemma}
\begin{proof}
 Suppose that $\alpha_1, \dots, \alpha_n$ are the roots of $P_n(z)$.  Then (\ref{QRDiffEq}) can be written as
\begin{equation}\label{qDiff}
\begin{split}
&\left[ A(z) + B(z) + C_n(z)\right] Q_n(z) =  q^n A(z) Q_n(q^{-1}z)+ q^{-n}B(z) Q_n(qz), \mbox{ where } \\
&Q_n(z)  = \prod_{i = 1}^n ( z^2 - \alpha_i z + u), \hspace{2mm} A(z) = (z - \alpha q)(z - \beta\delta q)(z - \gamma q)(z - \gamma \delta q)(z^2 - \gamma \delta), \\
&B(z) = q (z - 1)(z - \delta)(z\beta -\gamma)(z\alpha - \gamma \delta)(z^2- \gamma \delta q^2), \\
&C_n(z) = (z^2 - \gamma\delta)(z^2 - \gamma\delta q)(z^2 - \gamma\delta q^2)(q^{-n} - 1)(1 - \alpha\beta q^{n+1}).
\end{split}
\end{equation}
From Proposition \ref{rootsInt} we know that $\alpha_i \in ( 1 + u ,  q^{-M} + u q^M)$ for $i = 1, \dots, n$. Consequently, $z^2 - \alpha_i z + u$ has two real roots $q^{-M} > \beta_i > 1$ and $1 >u \beta_i^{-1} > 0$. 

Suppose that $P_n(\sigma(t)) = 0$ for some $t > 1$. Let $k \geq 1$ be the maximal integer such that $P_n(\sigma(t q^{-i})) = 0$ for $i = 0, \dots, k-1$ and assume for the sake of contradiction that $k \geq 2$. Then we know that $t q^{-k+1}$ is a root of $Q_n(z)$ and $Q_n(qz)$. From the top line in (\ref{qDiff}) we conclude that $A(t q^{-k +1}) Q_n(t q^{-k}) = 0$.
By the maximality of $k$ we must have $A(t q^{-k +1})  = 0$. Since $t q^{-k+1} > 1$ we conclude that $t q^{-k+1} = \alpha q$ and then as $t q^{-k+1} = \beta_j$ for some $n \geq j \geq 1$ we conclude that $q^{-M} > \alpha q$ or $\gamma > \alpha $. This contradicts Definition \ref{ParSetQR} and so $k = 1$ as desired.

\end{proof}

We end this section with a lemma, which classifies the bounded degree polynomials $A(z), B(z), C(z)$ that satisfy the $q$-Difference equation (\ref{qDiff}). 
\begin{lemma}\label{qDiffL}
Fix $M \geq n > 7$ and suppose that $\tilde{A}(z), \tilde{B}(z), \tilde{C}(z)$ are polynomials, each of degree at most $n - 7$ or zero and such that 
\begin{equation}\label{genqDiff}
\left[ \tilde{A}(z) + \tilde{B}(z) + \tilde{C}(z)\right] P_n(\sigma(z)) =  \tilde{A}(z)P_n(\sigma(q^{-1}z))+ \tilde{B}(z) P_n(\sigma(qz)).
\end{equation}
Then
\begin{equation}\label{cr1}
\tilde{B}(z)A(z) =  B(z)\tilde{A}(z) \mbox{ and } \tilde{B}(z)C_n(z) =  B(z)\tilde{C}(z).
\end{equation}
 where $A(z),B(z),C_n(z)$ are as in (\ref{qDiff}).
\end{lemma}
\begin{proof}
Using the notation from the proof of Lemma \ref{commonRoots} we can alternatively rewrite (\ref{genqDiff}) as
$$\left[ \tilde{A}(z) + \tilde{B}(z) + \tilde{C}(z)\right] Q_n(z) =  q^n\tilde{A}(z)Q_n(q^{-1}z)+ q^{-n}\tilde{B}(z) Q_n(qz).$$
We multiply the above by $B(z)$ and subtract it from the top line of (\ref{qDiff}) multiplied by $\tilde{B}(z)$ 
\begin{equation*}
\left[ \tilde{B}(z)(A(z) + C_n(z)) - B(z)(\tilde{A}(z) + \tilde{C}(z)) \right] Q_n(z) = q^n[\tilde{B}(z) A(z) -  B(z)\tilde{A}(z) ]Q_n(q^{-1}z).
\end{equation*}
As discussed in the proof of Lemma \ref{commonRoots} we know that $Q_n(z)$ has $n$ roots $\beta_1, \dots, \beta_n > 1$ and neither is a root of $Q_n(q^{-1}z)$. This implies that 
$\prod_{i = 1}^n (z - \beta_i) $ divides $\tilde{B}(z) A(z) -  B(z)\tilde{A}(z)$ and as the latter is of degree at most $n-1$ or zero, we conclude that it must be zero. This suffices for the proof.
\end{proof}

\subsection{The matrix $A_N(z)$ }\label{Section10.2}
We go back to the notation of Section \ref{DRHP}.  Let $w$ be as in (\ref{eq:jump}) with $\omega \equiv \omega^{qR}$ as in (\ref{omegaQR}). From Theorem \ref{RHPsol} we know DRHP$(\mathfrak{X}, w)$ has a unique solution $m_N(\zeta)$ with
$$m_N(\zeta)\cdot \left [\begin{array}{cc} \zeta^{-N} & 0\\0 & \zeta^{N}
\end{array}\right] =I+O\left(\zeta^{-1}\right)\text{  as  } \zeta \rightarrow \infty.$$

For the sequel we let $\sigma(z) = z + u z^{-1}$ with $u = \gamma \delta q $ and define the following matrix-valued function that will play a central role in the arguments that follow
\begin{equation}\label{AN}
\begin{split}
A_N(z):=m_N \left( \sigma(qz)\right
) \cdot D(z) \cdot m_ N^{-1} \left ( \sigma(z) \right), \mbox{ where } D(z) = \left [\begin{array}{cc} \Phi^-(z)  & 0\\0 &   \Phi^+(z)
\end{array}\right], \mbox{ where }
\end{split}
\end{equation}
\begin{equation}\label{phis}
\Phi^+(z) =  (z-\alpha )(z-\beta\delta )(z-\gamma )(z-\gamma \delta ) \mbox{ and } \Phi^-(z) =(z-1)(\alpha z - \gamma \delta )(\beta z - \gamma )(z - \delta).
\end{equation}
The significance of the functions $\Phi^{\pm}$ is as follows. If $\mathbb{P}_N$ denotes the measure $\mathbb{P}^{qR}$ as in Definition \ref{DefQR} then $\mathbb{P}_N$ satisfies the Nekrasov's equation, Theorem \ref{NekGen}, for the functions $\Phi^{\pm}$, $\theta = 1$, $k = 1$ and $a_1 = 0$, $b_1 = M- N +2$. In particular, if $\tilde{R}$ is given by
\begin{equation} \label{REQ}
\tilde{R}(z)=\Phi^{-}(z)\cdot \mathbb{E}_{\mathbb{P}_N} \left[ \prod\limits^N_{i=1} \frac{\sigma(qz)-\ell_i}{\sigma(z)-\ell_i}\right]+ \Phi^{+}(z)\cdot
\mathbb{E}_{\mathbb{P}_N} \left[ \prod\limits^N_{i=1} \frac{\sigma(z)- \ell_i}{\sigma(qz)-\ell_i}\right],
\end{equation} 
then $\tilde{R}$ is a degree four polynomial. Fact \ref{fact_7} concerns the asymptotic behavior of $\tilde{R}$ when the parameters $\alpha, \beta, \gamma ,\delta, q$ and $M$ scale as in Definition \ref{ScaleQR} and we can relate it to the matrix $A_N$ through the following result.
\begin{lemma}\label{AtoR} If  $A_N(z)$ and $\tilde{R}(z)$ are as in (\ref{AN})  and  (\ref{REQ}) respectively then
\begin{equation}\label{AtoReq}
 \mbox{Tr}(A_N(z)) = \tilde{R}(z).
\end{equation}
\end{lemma}
\begin{proof}
We first recall from \cite[Theorem 2.13]{BDS} that 
$$
\mathbb{E}_{\mathbb{P}_N} \left[ \prod\limits^N_{i=1} \frac{x-\ell_i}{y-\ell_i}\right] = c^{-1}_{N-1}  \mbox{det} \left [\begin{array}{cc}H_{N-1}(y) & H_{N}(y)\\ P_{N-1}(x) &  P_N(x)
\end{array}\right], \mbox{ with $c_{n}, H_n, P_n$ as in (\ref{DPEN}), (\ref{qRacP}) and (\ref{QRorthog})}.
$$
Combining the above with (\ref{REQ}) we obtain
\begin{equation}\label{RNalt}
\begin{split}
 \tilde{R}(z) = &c_{N-1}^{-1} \cdot \Phi^-(z) \cdot \left[H_{N-1}(\sigma(z))P_N(\sigma(qz)) -H_{N}(\sigma(z))P_{N-1}(\sigma(qz))  \right] +  \\& c_{N-1}^{-1} \cdot \Phi^+(z) \cdot \left[ H_{N-1}(\sigma(qz))P_N(\sigma(z)) -H_{N}(\sigma(qz))P_{N-1}(\sigma(z))\right].
\end{split}
\end{equation}

On the other hand, from (\ref{AN}) and Theorem \ref{RHPsol} we know
\begin{equation}\label{ANalt}
\begin{split}
A_N(z)=&\left [\begin{array}{cc} P_N(\sigma(qz)) & H_{N}(\sigma(qz))\\ c_{N-1}^{-1}P_{N-1}(\sigma(qz))  &
    c_{N-1}^{-1}H_{N-1}(\sigma(qz))
\end{array}\right]\cdot \left [\begin{array}{cc} \Phi^-(z) & 0\\0 & \Phi^+(z)
\end{array}\right] \cdot \\
&\left [\begin{array}{cc} c_{N-1}^{-1}H_{N-1}(\sigma(z)) &
  -H_N(\sigma(z))\\-c_{N-1}^{-1}P_{N-1}(\sigma(z)) & P_N(\sigma(z))
\end{array}\right],
\end{split}
\end{equation}
where we also used that det $m_{N}(\zeta) \equiv 1$. The trace in (\ref{ANalt}) matches the rights side in (\ref{RNalt}).
\end{proof}

In the remainder of this section we establish several properties about the matrix $A_N$. 
\begin{proposition}\label{entireL} The matrix $A_N(z)$ is entire.
\end{proposition}
\begin{proof} The result and its proof are analogous to  \cite[ Proposition 3.3]{BB}. Appealing to (\ref{ANalt}) it is clear that the only possible singularities of $A_N(z)$ are simple poles at $z = q^{-x}$ for $x = 0,1, \dots, M +1$ and $z = u q^y$ for $y =  -1, 0, \dots, M $.

From (\ref{jumpCond}) we have for $z$ near $q^{-x}$ and $x = 1, \dots, M+1$  
\begin{equation} \label{eq1}
m_N(\sigma(qz))=F_1(z)\left (I
+\frac{q^{-x}}{(q^{-(x-1)}-u q^{(x-1)})(z-q^{-x})}w\left(q^{-x+1} + u q^{x - 1} \right)\right ),
\end{equation}
where $F_1(z)$ is an analytic, invertible matrix-valued function defined in a neighborhood of $q^{-x}.$ By definition of $w$ we have $w(q^{1} + uq^{-1}) = 0$ and so (\ref{eq1}) holds near $z = 1$ as well.

For $z$ near $q^{-x}$ and $x = 0, \dots, M$ we can similarly write
\begin{equation} \label{eq2} m^{-1}_N
  (\sigma(z))=\left (I-\frac{q^{-x}}{(q^{-x}-u q^{x})(z-q^{-x})}w(q^{-x} + u q^{x})\right ) F_2(z),
\end{equation}
where $F_2(z)$ is an analytic, invertible matrix-valued function defined in a neighborhood of $q^{-x}.$ By definition of $w$ we have $w(q^{-M-1} + u q^{M+1}) = 0$ and so (\ref{eq2}) holds near $z = q^{-M - 1}$ as well.

Overall, to show that the matrix $A_N(z)$ is analytic at $q^{-x}$ for $x = 0,1, \dots, M+1$ it suffices for
$$X(z):=\left (I
+\frac{q^{-x}}{(q^{-(x-1)}-u q^{(x-1)})(z-q^{-x})}w(\sigma(q x))\right )D(x)\left (I-\frac{q^{-x}}{(q^{-x}-u
      q^{x})(z-q^{-x})}w(x)\right ).$$
to be analytic near those points. Since 
$$X(z)=\left [\begin{array}{cc} \Phi^-(x)
    &\dfrac{1}{(z-q^{-x})}\left(\Phi^+(x)\cdot
      \dfrac{q^{-x}}{(q^{-(x-1)}-u q^{(x-1)})}\omega^{qR}(\sigma(q x))-\dfrac{q^{-x}}{(q^{-x}-u
      q^{x})}\omega^{qR}(\sigma(x)) \Phi^-(x) \right)\\0 & \Phi^+(x)
\end{array}\right]$$
$$\mbox{ and } \frac{\Phi^-(x)}{\Phi^+(x)}=\frac{(q^{-x}-u
      q^{x})\omega^{qR}(\sigma(q x))}{\omega^{qR}(\sigma(x))(q^{-(x-1)}-u q^{(x-1)})},$$
we conclude that $X(z)$ is indeed analytic near $q^{-x}$ for $x = 0, 1 , \dots,M$. 

One verifies the analyticity of $A_N(z)$ at the points $uq^{y}$ for $y =  -1, 0, \dots, M $ analogously.

\end{proof}

\begin{proposition}\label{leadProp} Let $A^{ij}_N(z)$ denote the entry of $A_N(z)$ at the $i$-th row and $j$-th column. Then $A^{11}_N(z), A^{22}_N(z)$ are degree four polynomials and $A^{12}_N(z), A^{21}_N(z)$ are degree three polynomials. If $A^{ij}_N(z) = \sum_k a^{ij}_{N,k} z^k$ then 
\begin{equation}\label{leadcoeff}
a^{11}_{N,4} =\alpha \beta q^{N},  a^{22}_{N,4} =  \alpha \beta q^{-N}, a^{12}_{N,3} = c_{N-1}^{-1}c_N  [ - q^N \alpha\beta +  q^{-N}],   a^{21}_{N,3} =   [  q^{-N+1} \alpha \beta -  q^{N-1} ]
\end{equation} 
\end{proposition}
\begin{proof}
From equation (\ref{ANalt}) we have
\begin{equation}\label{ANentries}
\begin{split}
A^{11}_N(z) = &c_{N-1}^{-1} \left[ P_N(\sigma(qz)) H_{N-1}(\sigma(z)) \Phi^-(z) -   H_N(\sigma(qz)) P_{N-1}(\sigma(z)) \Phi^+(z)  \right],  \\
A^{22}_N(z) = &c_{N-1}^{-1} \left[ -P_{N-1}(\sigma(qz)) H_{N}(\sigma(z)) \Phi^-(z) +   H_{N-1}(\sigma(qz)) P_{N}(\sigma(z)) \Phi^+(z)  \right],  \\
A^{12}_N(z) = &c_{N-1}^{-1} \left[- P_N(\sigma(qz)) H_{N}(\sigma(z)) \Phi^-(z) +   H_N(\sigma(qz)) P_{N}(\sigma(z)) \Phi^+(z)  \right],   \\
A^{21}_N(z) = &c_{N-1}^{-1} \left[P_{N-1}(\sigma(qz)) H_{N-1}(\sigma(z)) \Phi^-(z) -   H_{N-1}(\sigma(qz)) P_{N-1}(\sigma(z)) \Phi^+(z)  \right].    \\
\end{split}
\end{equation}
We know from Proposition \ref{entireL} that $A_N^{ij}(z)$ are all entire functions. In addition, by Theorem \ref{RHPsol} we know that $c^{-1}_{N-1} P_N(\sigma(qz)) H_{N-1}(\sigma(z))  = \left[(\sigma(qz))^{-N} P_N(\sigma(qz)) \right] \cdot \left[ (\sigma(z))^{N}c^{-1}_{N-1}  H_{N-1}(\sigma(z)) \right] \cdot  \left[ (\sigma(z))^{-N}(\sigma(qz))^{N}\right] \sim q^{N}$
 as $|z| \rightarrow \infty$. Analogous arguments show that 
$$ c_{N-1}^{-1}  H_{N-1}(\sigma(qz)) P_{N}(\sigma(z)) \sim q^{-N}, \hspace{2mm}H_N(\sigma(qz)) P_{N-1}(\sigma(z)) = O\left(|z|^{-2}\right)= P_{N-1}(\sigma(qz)) H_{N}(\sigma(z)) .$$
The above show that $A^{11}_N(z) = O(|z|^4) = A^{22}_N(z)$ and by Liouville's theorem we conclude that they are at most degree $4$ polynomials. Our work above also shows that as $|z| \rightarrow \infty$ we have
$$A^{11}_N(z) \sim  \alpha \beta q^{N} z^4 \mbox{ and } A^{22}_N(z) \sim  \alpha \beta q^{-N} z^4,$$
which establishes the first two equations in (\ref{leadcoeff}).

Also as a consequence of Theorem \ref{RHPsol} we have that $A^{12}_N(z) = O(|z|^{3}) = A^{21}_N(z)$, which by Liouville's theorem implies that they are at most degree three polynomials. What remains is to show that their leading coefficients are as in (\ref{leadcoeff}).

We first note by definition 
$$\zeta^N H_N(\zeta) = \zeta^N \sum_{x = 0}^M \frac{P_N(q^{-x} + u q^x) \omega^{qR}(q^{-x} + u q^x) }{\zeta - q^{-x} - u q^x} = $$
$$ = \zeta^N \sum_{x = 0}^M P_N(q^{-x} + u q^x) \omega^{qR}(q^{-x} + u q^x) \cdot\left[ \frac{1}{\zeta}  + \frac{[q^{-x} + u q^x]}{\zeta^2} + \cdots + \frac{[q^{-x} + u q^x]^N}{\zeta^{N+1}} +  O(\zeta^{-N-2})\right].$$
Using (\ref{QRorthog}) we get
$$ \sum_{x = 0}^M P_N(q^{-x} + u q^x) w^{qR}(q^{-x} + u q^x) [q^{-x} + u q^x]^i = 0 \mbox{ for $i = 0, \dots, N-1$ and }$$
$$ \sum_{x = 0}^M P_N(q^{-x} + u q^x) w^{qR}(q^{-x} + u q^x) [q^{-x} + u q^x]^N = \sum_{x = 0}^N P_N(q^{-x} + u q^x)^2w^{qR}(q^{-x} + u q^x)= c_N.$$
The above implies that as $|z|\rightarrow \infty$ we have
$$P_N(\sigma(qz)) H_N(\sigma(z))=  \left[(\sigma(qz))^{-N} P_N(\sigma(qz)) \right] \cdot \left[ (\sigma(z))^{N}  H_{N}(\sigma(z)) \right] \cdot  \left[ (\sigma(z))^{-N}(\sigma(qz))^{N}\right] \sim z^{-1} c_N q^N.$$
Analogously, $H_N(\sigma(qz)) P_{N}(\sigma(z))  \sim z^{-1} c_N q^{-N}, $ $ P_{N-1}(\sigma(qz)) H_{N-1}(\sigma(z)) \sim z^{-1} c_{N-1}q^{-N+1}, $ \\ $H_{N-1}(\sigma(qz)) P_{N-1}(\sigma(z))  \sim z^{-1} c_{N-1} q^{N-1}.$ The latter identities and (\ref{ANentries}) imply
$$a^{12}_{N,3} = c_{N-1}^{-1}  [ - c_Nq^N \alpha\beta + c_N q^{-N}] \mbox{ and } a^{21}_{N,3} = c_{N-1}^{-1} [  c_{N-1} q^{-N+1} \alpha \beta -  c_{N-1} q^{N-1} ],$$
which establishes the last two equations in (\ref{leadcoeff}).
\end{proof}

We end this section by relating $A_N(z)$ and the polynomials $A(z),B(z),C_n(z)$ in the $q$-Difference equation (\ref{qDiff}).
From (\ref{AN}) we know that $A_N(z) \cdot m_N(\sigma(z)) = m_N(\sigma(qz)) \cdot D(z)$, and identifying the left two entries of the $2\times 2$ matrices on both sides we arrive at
\begin{equation}\label{system1}
\begin{split}
&A_N^{11}(z) P_N(\sigma(z)) + c_{N-1}^{-1}A_N^{12}(z) P_{N-1}(\sigma(z)) = \Phi^{-}(z) P_N(\sigma(qz)), \\
&A_N^{21}(z) P_N(\sigma(z)) + c_{N-1}^{-1}A_N^{22}(z) P_{N-1}(\sigma(z)) = c_{N-1}^{-1} \Phi^{-}(z) P_{N-1}(\sigma(qz)).
\end{split}
\end{equation}
Expressing $c_{N-1}^{-1} P_{N-1}(\cdot)$ from the first equation and substituting it in the second we obtain
\begin{equation}\label{system2}
\begin{split}
&A_N^{21}(z) P_N(\sigma(z)) + A_N^{22}(z) \left[ \frac{ \Phi^{-}(z) P_N(\sigma(qz)) - A_N^{11}(z) P_N(\sigma(z))}{A_N^{12}(z) } \right]  = \\
& \Phi^{-}(z) \cdot  \left[ \frac{ \Phi^{-}(qz) P_N(\sigma(q^2z)) - A_N^{11}(qz) P_N(\sigma(qz))}{A_N^{12}(qz) } \right].
\end{split}
\end{equation}
Replacing $z$ with $q^{-1}z$, and reorganizing terms we obtain
\begin{equation*}\label{system3}
\begin{split}
&P_N(\sigma(z)) \left[\frac{\Phi^-(q^{-1}z) A^{22}_N(q^{-1}z) }{A_N^{12}(q^{-1}z)} + \frac{\Phi^{-}(q^{-1}z) A_N^{11}(z)}{A_N^{12}(z)} \right] = \\
&P_N(\sigma(q^{-1}z)) \left[ \frac{A_N^{22}(q^{-1}z) A_N^{11}(q^{-1}z) - A_N^{12}(q^{-1}z)A_N^{21}(q^{-1}z)}{A_N^{12}(q^{-1}z)} \right] +    P_N(\sigma(qz)) \left[ \frac{\Phi^-(q^{-1}z) \Phi^-(z) }{A_N^{12}(z)}\right] .
\end{split}
\end{equation*}
Note that the first numerator of the second line above is $\mbox{det}A_N(q^{-1}z)$, which we know to equal $\Phi^+(q^{-1}z) \Phi^-(q^{-1}z) $ (recall that $\mbox{det}m_N = 1$ by Theorem \ref{RHPsol}). Making the substitution, and multiplying both sides by $ A_N^{12}(z) A_N^{12}(q^{-1}z) \cdot \Phi^-(q^{-1}z)^{-1}$ we arrive at
\begin{equation}\label{system4}
\begin{split}
&P_N(\sigma(z)) \left[A^{22}_N(q^{-1}z) A_N^{12}(z)+ A_N^{12}(q^{-1}z)A_N^{11}(z)\right] = \\
&P_N(\sigma(q^{-1}z)) A_N^{12}(z)\Phi^+(q^{-1}z) +    P_N(\sigma(qz))A_N^{12}(q^{-1}z)\Phi^-(z) .
\end{split}
\end{equation}

If we alternatively express $c_{N-1}^{-1} P_{N}(\cdot)$ from the second equation in (\ref{system1}), substitute it into the first and perform the same steps we will arrive at
\begin{equation}\label{system5}
\begin{split}
&P_{N-1}(\sigma(z))\left[A^{22}_N(z) A_N^{21}(q^{-1}z)+ A_N^{21}(z)A_N^{11}(q^{-1}z)\right] = \\
&P_{N-1}(\sigma(q^{-1}z)) A_N^{21}(z)\Phi^+(q^{-1}z) +    P_{N-1}(\sigma(qz))A_N^{21}(q^{-1}z)\Phi^-(z).
\end{split}
\end{equation}

In the remainder we assume that $N \geq 14$. Then applying Lemma \ref{qDiffL} to (\ref{system4}) we conclude 
\begin{equation*}
\begin{split}
&q (z - 1)(z - \delta)(z\beta -\gamma)(z\alpha - \gamma \delta)(z^2- \gamma \delta q^2)  \cdot A_N^{12}(z)\Phi^+(q^{-1}z)  = \\
& (z - \alpha q)(z - \beta\delta q)(z - \gamma q)(z - \gamma \delta q)(z^2 - \gamma \delta) \cdot A_N^{12}(q^{-1}z)\Phi^-(z).
\end{split}
\end{equation*}
Replacing the formulas for $\Phi^{\pm}$ from (\ref{phis}) and canceling common terms we arrive at
\begin{equation*}
\begin{split}
&q^{-3}(z^2- \gamma \delta q^2)   \cdot A_N^{12}(z)  = (z^2 - \gamma \delta)\cdot A_N^{12}(q^{-1}z).
\end{split}
\end{equation*}
The above and Proposition \ref{leadProp} imply  
\begin{equation}\label{A12}
A_N^{12}(z) =   c_{N-1}^{-1}c_N  [ - q^N \alpha\beta +  q^{-N}] \cdot z (z^2 - \gamma \delta) .
\end{equation}
If we alternatively apply  Lemma \ref{qDiffL} to (\ref{system5}) and repeat the same argument we arrive at
\begin{equation}\label{A21}
A_N^{21}(z) =    [  q^{-N+1} \alpha \beta -  q^{N-1} ] \cdot z (z^2 - \gamma \delta) .
\end{equation}

\subsection{Asymptotics of $A_N(z)$} \label{Section10.3}
We now assume that $\alpha,\beta, \gamma, \delta, q$ and $M$ all depend on $N$ and scale as in Definition \ref{ScaleQR}. For this choice of parameters we denote $\tilde{R}$ in (\ref{REQ}) by $\tilde{R}_N$ and $\Phi^{\pm}$ by $\Phi^{\pm}_N$. 

Our first goal is to show that under the above parameter scaling $A_N(z)$ converges, to a fixed matrix-valued function as $N \rightarrow \infty$. Let us first consider the off-diagonal entries $A^{12}_N(z),  A^{21}_N(z)$. In view of (\ref{A12}) and (\ref{A21}) we know that
\begin{equation*}
\begin{split}
& A^{12}_N(z) =  c_{N-1}^{-1}c_N [ - q_N^N \alpha_N\beta_N +  q_N^{-N}] \cdot z (z^2 - \gamma_N \delta_N) \mbox{ and }\\
& A_N^{21}(z) =   [  q_N^{-N+1} \alpha_N\beta_N -  q_N^{N-1} ] \cdot z (z^2 - \gamma_N \delta_N) 
\end{split}
\end{equation*}
Using (\ref{QRorthog}) we have
$$c_N c_{N-1}^{-1} = \frac{ q^{-M} (1 - \alpha_N \beta_N q_N^{2N-1})}{1 - \alpha_N \beta_N q_N^{2N+1}} \cdot \frac{(1 - q_N^{N})(1 - \beta_N q_N^N) ( \delta_N - \alpha_N  q_N^N)(1 - \alpha_N\beta_N q^{M+N+1}) }{( 1-  \alpha_N \beta_N q_N^N) (1 - \alpha_N q_N^N) (1 - \beta_N \delta_N q_N^N) (1 - q_N^{-M_N + N})}.$$
Consequently, we see that 
\begin{equation}\label{cnass}
\lim_{N \rightarrow \infty} c_N c_{N-1}^{-1} =  \frac{\lc(1 - \lq)(1 - \lb \lq) ( \ld - \la  \lq)(1 - \la\lb \lq \lc^{-1}) }{( 1-  \la \lb \lq) (1 - \la \lq) (1 - \lb \ld \lq) (1 - \lc \lq )} =:\lambda.
\end{equation}
From the above work we conclude that 
\begin{equation*}
\begin{split}
& \lim_{N \rightarrow \infty}  A^{12}_N(z) = \lambda  [\lq^{-1} - \lq \la \lb] \cdot z (z^2 - \lc \ld) \mbox{ and }\lim_{N \rightarrow \infty} A_N^{21}(z) =   [  \lq^{-1} \la\lb - \lq ] \cdot z (z^2 - \lc \ld) .
\end{split}
\end{equation*}
Next, we know that $\mbox{det}A_N(z) = \Phi^{+}_N(z) \Phi^-_N(z)$ converges to $ (z-\la )(z-\lb\ld )(z-\lc )(z-\lc \ld )(z-1)(\la z - \lc \ld )(\lb z - \lc )(z - \ld)$ as $N\rightarrow \infty$, while from (\ref{AtoReq}) and (\ref{RLimNek}) we know that Tr$A_N(z)$ converges to some degree four polynomial $\tilde{R}_{\infty}(z)$. This implies that $A^{11}_N(z)$ and $A^{22}_N(z)$ converge to some degree four polynomials $A_\infty^{11}(z)$ and $A_\infty^{22}(z)$. \\

We end this section by proving Fact \ref{fact_7}. From Lemma \ref{qDiffL} applied to (\ref{system4}) we have that
$$B(z) \cdot \left[A^{22}_N(q_N^{-1}z) A_N^{12}(z)+ A_N^{12}(q_N^{-1}z)A_N^{11}(z) \right] = A_N^{12}(q_N^{-1}z)\Phi^-(z) \cdot \left[A(z) + B(z) + C_N(z) \right],$$
Taking the limit as $N \rightarrow \infty$ on both sides we get
$$ (z - 1)(z - \ld)(z\lb -\lc)(z\la - \lc \ld)(z^2- \lc \ld) \lambda  [\lq^{-1} - \lq \la \lb] \cdot z (z^2 - \lc \ld) \cdot \left[A^{22}_\infty(z) + A_\infty^{11}(z) \right] = $$
$$ \lambda  [\lq^{-1} - \lq \la \lb]  z (z^2 - \lc \ld)  \cdot(z-1)(\la z - \lc \ld )(\lb z - \lc )(z - \ld) \times $$
$$ [  (z - \la)(z - \lb\ld)(z - \lc )(z - \lc \ld )(z^2 - \lc \ld) +(z - 1)(z - \ld)(z\lb -\lc)(z\la - \lc \ld)(z^2- \lc \ld) + (z^2 - \lc\ld)^{3}(\lq^{-1} - 1)(1 - \la\lb \lq)] .$$ 
Canceling common factors and utilizing (\ref{AtoReq}) we see that
\begin{equation}\label{proofFact7}
\begin{split}
&\lim_{N \rightarrow \infty} \tilde{R}_N(z) = \lim_{N \rightarrow \infty} \mbox{Tr} A_N(z) = A^{22}_\infty(z) + A_\infty^{11}(z) = (z - \la)(z - \lb\ld)(z - \lc )(z - \lc \ld ) + \\
& (z - 1)(z - \ld)(z\lb -\lc)(z\la - \lc \ld)+ (z^2 - \lc\ld)^2(\lq^{-1} - 1)(1 - \la\lb \lq),
\end{split}
\end{equation}
which concludes the proof of Fact \ref{fact_7}.
\begin{rem}
We presented the computation in the $q$-Racah case but the same ideas would
work for other families of classical orthogonal polynomials.
\end{rem}

\bibliographystyle{amsplain}
\bibliography{WholePaper}

\end{document}